\documentclass[a4paper]{article}
\usepackage{lineno,hyperref}

\usepackage{graphicx}
\usepackage{epstopdf}
\usepackage{amssymb}
\usepackage{amsmath}
\usepackage{amsfonts}
\usepackage{amsthm} 

\usepackage{lipsum}

\usepackage{mathtools}
\usepackage{a4wide}
\usepackage{caption}
\usepackage{subcaption}
\usepackage{setspace}
\usepackage{xr}
\usepackage{color}
\usepackage{soul}
\usepackage{xcolor}
\usepackage{mathtools}
\usepackage{setspace}
\usepackage{float}
\usepackage{enumerate}

\usepackage{url}
\usepackage[utf8]{inputenc}
\usepackage[T1]{fontenc}
\usepackage{lmodern}
\usepackage[capitalise,noabbrev,nameinlink]{cleveref}

\usepackage{tabularx}
\usepackage{array}
\usepackage{booktabs}
\usepackage{multirow}

\usepackage{eurosym}
\usepackage{textcomp}

\usepackage[affil-it]{authblk}
\usepackage{longtable}

\newtheorem{theorem}{Theorem}[section]

\newtheorem{corollary}[theorem]{Corollary}

\newtheorem{lemma}[theorem]{Lemma}

\newtheorem{proposition}[theorem]{Proposition}

\theoremstyle{definition}

\newtheorem{remark}[theorem]{Remark}

\begin{document}



\title{A general model for time-minimizing navigation  on~a~mountain slope under gravity}

\author{Nicoleta Aldea$^1$}
\author{Piotr Kopacz$^{2}$}



\affil{\small{$^1$Transilvania University of Bra\c{s}ov, Faculty of Mathematics and Computer Science\\ Iuliu Maniu 50, Bra\c{s}ov, Romania}}
\affil{$^2$Gdynia Maritime University, Faculty of Navigation \\  Al. Jana Paw{\l}a II 3, 81-345 Gdynia, Poland}

\date{}
\date{{\normalsize {\small {e-mail:} \texttt{codruta.aldea@unitbv.ro, p.kopacz@wn.umg.edu.pl}}}}
\maketitle

\begin{abstract} 

In this work, 
we solve the generalized Matsumoto's slope-of-a-mountain problem by means of Riemann-Finsler geometry, making close links with the Zermelo navigation problem. The time-minimizing navigation under gravity is analyzed in  the general model of a slippery mountain slope  being a Riemannian manifold. Both the transverse and longitudinal gravity-additives with respect to the direction of motion are admitted to vary simultaneously in the full ranges, showing the impact of cross- and along-traction  on the slippery slope.  By the anisotropic deformation of the background Riemannian metric and rigid translation with the use of the rescaled gravitational wind, we obtain the purely geometric solution for optimal navigation, which is given by a new Finsler metric belonging to the class of general $(\alpha, \beta)$-metrics. The related strong convexity conditions are established and time geodesics are described. Moreover, the evolution of time fronts and the behavior of  time-minimizing trajectories in relation to various gravity effects on the slippery slope, gravitational wind force and direction of motion  are thoroughly discussed and visualized by several two-dimensional examples.  
\end{abstract}


\bigskip \noindent \textbf{MSC 2020}: 53B40, 53C60.

\smallskip \noindent \textbf{Keywords:} Riemann-Finsler manifold, Time geodesic, Matsumoto's slope-of-a-mountain problem, Zermelo's navigation problem, Slippery slope metric, Gravity.



\section{Introduction}
The presented study focuses on finding optimal paths in the sense of time (time geodesics) in the general model of a slippery mountain slope  under action of gravity by means of Riemann-Finsler geometry. 

\subsection{Background of the study} 

We start by recalling the slope-of-a-mountain problem of Matsumoto (MAT for brevity) which was initially investigated in a purely geometric way in \cite{matsumoto}, where the author assumed that the slope was not slippery at all in the usual sense. The main objective was to find the time-minimizing paths  on the mountain side under the influence of gravity.  Such setting implied that the self-velocity (the control vector) $u$ of a walker or a moving craft on the slope and the corresponding resultant velocity $v$ always point in the same directions. In addition, the related  speeds differ from each other by the entire along-gravity additive, being the norm of the orthogonal projection of the  component of gravity\footnote{To be precise, the velocity induced on the walker by the gravitational force.},  
 tangent  to a slope on $u$, for any direction of motion. As is natural, there is higher resultant speed obtained in a downhill motion than in an uphill climbing, while  the self-speed of a walker or a craft on the slope is kept maximal and constant.  Thus, there is no drift (sliding) to any side taken into account. In other words, the cross component of the   gravitational force pushing a walker off the $u$-track on the slope is always fully compensated (cancelled) in Matsumoto's  model \cite{matsumoto}. Although somewhat simplified, this basic setting yields a quite natural  description of motion on the hill side for some applications,  where upward walking is more strenuous than going downhill.  It is worth mentioning that a Matsumoto metric has also been applied to a geometric description of the wildfire spread structure  \cite{markvorsen,JPS}.

Recently, a more general approach in the context of time-minimizing solutions has been presented in \cite{slippery}, describing a slippery slope model that admits  the side drifts. This time the velocities $u$ and $v$  are not collinear in general whilst on the move,  pointing  in different directions. In that study (SLIPPERY for short) a \textit{cross-traction coefficient} $\eta\in[0, 1]$ was introduced in particular, by which the transverse  effect (i.e. to the left or right side of the velocity $u$) of the gravitational force acting on a mountain slope was determined. Such reasoning enabled studying the Matsumoto's slope-of-a-mountain problem somewhat in the spirit of the Zermelo navigation problem (ZNP for short) on a Riemannian manifold $(M, h)$ in Finsler geometry as in, e.g. \cite{chern_shen,colleenshen}. Namely, this involved  a \textit{gravitational wind}, i.e. the component\footnote{A gravitational field $\mathbf{G}=\mathbf{G}^{T}+\mathbf{G}^{\perp}$, where $\mathbf{G}^{T}$  is tangent to a  mountain slope $M$ and acts in the steepest downhill direction, considering a 2-dimensional model of the slope, and $\mathbf{G}^{\perp }$ is normal to $M$.} $\mathbf{G}^{T}$ of gravity $\mathbf{G}$, which is tangent to $M$ and ``blows'' in the steepest downhill  direction (along the negative gradient).  We recall briefly that ZNP classically in low-dimensional formulations is to find the least time trajectories of a craft which moves at maximum speed with respect to a surrounding air or water between two positions at sea, on the river or in the air, where the water stream or wind modeled by a variable vector field is always taken into account entirely \cite{Zer}.  

It may be worth reminding the reader from the previous studies (\cite{slippery, cross, slipperyx}) 
 that a part of a gravitational wind $\mathbf{G}^{T}$ is compensated due to nonzero traction during motion on a mountain slope. The cancelled part is named a \textit{dead wind} which has actually no impact on time-minimizing paths. In turn, the remaining part we call an \textit{active wind} and denote by $\mathbf{G}_\eta$ in the slippery slope model (SLIPPERY). This plays an essential role in all  investigations on the slope navigation, being present in the equations of motion.  
 Moreover, the active wind decomposes next into two components, i.e. its orthogonal projections onto $u$ (called an \textit{effective wind}) and on the orthogonal direction to $u$ (called a \textit{cross wind}). Thus, for any given direction of motion and acceleration of gravity in the slippery slope model the sum of the maximum cross wind and maximum effective wind yields the entire  gravitational wind $\mathbf{G}^{T}$. For the sake of clarity, see \cref{fig_slope_general} in the current context.

In the SLIPPERY scenario, the effective wind always acts in full, for each direction of motion and wind force $||\mathbf{G}^{T}||_h$. At the same time, the cross wind is subject to compensation due to cross-traction on the slippery slope. 
  More precisely, the corresponding equation of motion reads
\begin{equation}
v_\eta=u+\mathbf{G}_\eta, 
\label{eqs_slippery}
\end{equation}
where  $\mathbf{G}_{{\eta}}=(1-\eta)\text{Proj}_{u^{\perp}}\mathbf{G}^{T}+\text{Proj}_{u}\mathbf{G}^{T}$, or  equivalently, $\mathbf{G}_{{\eta}}=\eta \mathbf{G}_{MAT}+(1-\eta )\mathbf{G}^{T}, \eta \in[0, 1]$, where  $u^\perp$ denotes the orthogonal direction to $u$ and $\mathbf{G}_{MAT}$ stands for $\text{Proj}_{u}\mathbf{G}^{T}$.  

Hence, in the boundary cases, the original Matsumoto  problem   on the non-slippery slope (in the usual sense) and the Zermelo navigation problem (\cite{Zer, colleenshen}) under a gravitational wind $\mathbf{G}^{T}$ are linked and generalized. Both become the particluar cases in the new setting, i.e. with $\eta=1$ (no lateral drift, MAT) and $\eta=0$ (maximum\footnote{For any given direction of motion indicated by $u$ and gravitational wind force $||\mathbf{G}^{T}||_{h}$.} lateral drift, ZNP), respectively.

 Furthermore, to complement the exposition including the sliding effect, analogous model to the aforementioned has been considered very recently, where the longitudinal drift with respect to the direction of motion indicated by the velocity $u$  was taken into account in the corresponding equations of motion (S-CROSS for short); see \cite{slipperyx} for more details. In this case, in turn, an \textit{along-traction coefficient} $\tilde{\eta}\in[0, 1]$ describes the range of another type of sliding, while the cross-gravity increment is always taken in full, i.e. $\eta=0$, for any direction and gravity force. In this way it was possible to create a direct connection between the cross slope problem (CROSS for short, $\tilde{\eta}=1$) investigated  in \cite{cross} and the classical Zermelo's  navigation ($\tilde{\eta}=0$) again. 
The resultant velocity in this case is defined as follows
\begin{equation}
v_{\tilde{\eta}}=u+\mathbf{G}_{\tilde{\eta}},   
\label{eqs_slipperyx}
\end{equation}

\ 

\noindent where  $\mathbf{G}_{\tilde{\eta}}=\text{Proj}_{u^{\perp}}\mathbf{G}^{T}+(1-\tilde{\eta})\text{Proj}_{u}\mathbf{G}^{T}$, which is equivalent to $\mathbf{G}_{\tilde{\eta}}=-\tilde{\eta}\mathbf{G}_{MAT}+\mathbf{G}^{T}, \tilde{\eta} \in[0, 1]$.    
For the sake of clarity, see also \cref{fig_slope_general} in this regard. 
It is worth noting that ZNP  is positioned  somewhat right in the middle between both above approaches to modeling the slippery slopes  pieced together, namely, SLIPPERY  with scaling the lateral  drift (\cite{slippery}) and S-CROSS with the longitudinal compensation of the gravity impact on time-optimal motion on the mountain slope (\cite{slipperyx}). 

Both introduced parameters settle  the type and range of compensation of the gravitational wind, however, varying individually, i.e. $\eta$ across the $u$-direction as considered in \cite{slippery} and $\tilde{\eta}$ along the $u$-direction  as studied in \cite{slipperyx}. Those compensations determine next the transverse and longitudinal gravity-additives (slides) to self-motion. In the current investigation, we aim at analyzing the general case, admitting arbitrary type of a slide during the least time navigation on the slippery slope.  
  This means that both traction coefficients $\eta, \tilde{\eta}$ will be incorporated in the general equations of motion.  As a consequence, this study will generalize and collect the preceding results on time-optimal navigation under the action of gravity, which were obtained with a purely geometric approach by means of Finsler geometry, in particular in  \cite{slippery, cross, slipperyx}. Moreover, the \textit{slippery slope} will gain now a much broader meaning in the context of modeling time-minimizing motion on the hill side, as explained in the next subsection. The essential part of the study will  refer to the strong convexity conditions, which correspond to the geometrically expressed conditions for optimality in the sense of time.

\subsection{Defining navigation problems on a mountain slope with the use of traction coefficients}

Let us observe that by a pair of the traction coefficients it is possible to define in fact each  \textit{navigation problem} $\mathcal{P}$ in the slope model under action of gravity above mentioned, namely, $\mathcal{P}_{\eta, \tilde{\eta}}=(\eta, \tilde{\eta})$, where both parameters are fixed\footnote{In general, the notation with both lower indices, i.e. $\mathcal{P}_{\eta \tilde{\eta}}=(\eta, \tilde{\eta})$ will be used  especially for the slope problems that have not been specifically named or abbreviated like, for example, MAT or ZNP.}. Thus, we have  $\mathcal{P}_{MAT}=(1, 0)$,  $\mathcal{P}_{ZNP}=(0, 0)$, $\mathcal{P}_{CROSS}=(0, 1)$ and $\mathcal{P}_{RIEM}=(1, 1)$ yields the Riemannian case, where the impact of gravity on  motion is completely cancelled, i.e. $v=u$. Furthermore,  our objective is to present the general solution including all scenarios with the full ranges of both traction coefficients taken into account together, i.e. $\eta, \tilde{\eta}\in[0, 1]$, and not just the boundary  values ($\eta, \tilde{\eta}\in\{0, 1\}$) as has been studied so far (\cite{matsumoto, Zer, colleenshen,cross}). This will lead to the new concrete problems on the slope like, for instance, $\mathcal{P}=(\frac12, \frac13)$ or $\mathcal{P}'=(\frac{\pi}{5}, \sqrt{0.7})$, which in general have not been considered before. Consequently, any such setting will yield different type of motion on the slope, determined by given tractions. Then for any $\mathcal{P}_{\eta, \tilde{\eta}}$,  the specific study leading to the time-optimal paths can be developed effectively, creating the corresponding model for a potential application based on the arbitrary $(\eta, \tilde{\eta})$-navigation under the influence of gravity.  

Moreover, in case at least one of the traction coefficients is varying (this means that the cross- or along-gravity additive is variable) as in SLIPPERY or S-CROSS so far, we denote by $\mathcal{T}_{\mathcal{P}}^{\mathcal{P}'}$ (interchangeably $\mathcal{T}_{\eta,\tilde{\eta}}^{\eta',\tilde{\eta}'})$ a \textit{transition} $\mathcal{T}$ between two specific slippery slope problems $\mathcal{P}=(\eta, \tilde{\eta})$ and $\mathcal{P}'=(\eta', \tilde{\eta}')$. For clarity, this is illustrated in \cref{fig_square2}.  Hence, $\mathcal{T}_{1,0}^{0,0}$ now describes SLIPPERY linking MAT and ZNP, with $\eta\in[0, 1]$ and $\tilde{\eta}=0$, as well as $\mathcal{T}_{0,1}^{0,0}$ represents S-CROSS linking CROSS and ZNP,  with $\eta=0$ and $\tilde{\eta}\in[0, 1]$ (\cite{slippery,slipperyx}). In such a way we can collect, compare and present all the above mentioned scenarios graphically in a clear and unified manner with the bird's eye view on the problem diagram in \cref{fig_square_0}. 
More generally, we aim at covering the cases in our solution, where the traction coefficients are running through the arbitrary intervals $I_\eta, I_{\tilde{\eta}} \subseteq [0, 1]$ as well as the transitions $\mathcal{T}_{\mathcal{P}}^{\mathcal{P}'}$ that also connect the new type of problems  $\mathcal{P}_{\eta, \tilde{\eta}}$ as above mentioned\footnote{Each fixed pair $(\eta,\tilde{\eta})$ yields a specific type of motion related to $\mathcal{P}_{\eta, \tilde{\eta}}$. In turn, the equation of motion are changing during transition. There is a certain analogy to a flight of a variable-sweep wing aircraft (a swing-wing design), modifying its geometry while flying.}, e.g. $\mathcal{T}_{\frac12, \frac13}^{\frac{\pi}{5}, \sqrt{0.7}}$. 
 
To sum up in short, making use of the traction coefficients has an advantage. Gathering the preceding investigations related to various types of navigation problems on a mountain slope under various gravity effects, we can present them conveniently and compare to each other in a clear way by the scaling parameters taken as the unified coordinates 
 $(\eta, \tilde{\eta})$ as illustrated in \cref{fig_square_0}. Actually, the figure also shows the state of the art in modeling time-optimal navigation on a slope of mountain under the action of gravity studied thus far. 

 \begin{figure}[H]
\centering
\includegraphics[width=0.45\textwidth]{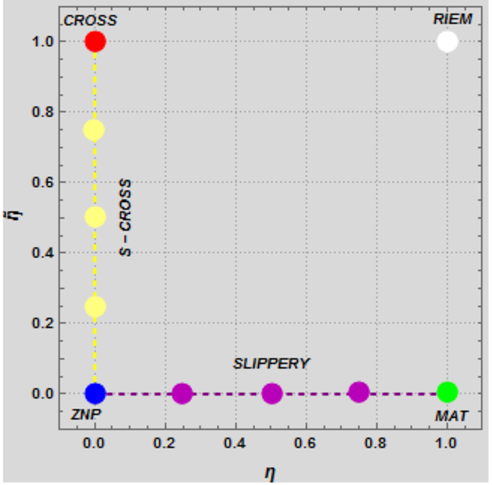}
\caption{The comparative presentation of the specific slope problems defined with the use of the traction coefficients $\eta, \tilde{\eta}$ on the problem diagram which have been studied in Riemann-Finsler geometry thus far.}
\label{fig_square_0}
\end{figure}

It is worth pointing out the meaning of  ''slippery''  and  ''non-slippery'' slope in the current context. As already emphasized, the initial Matsumoto's setting (\cite{matsumoto}) was treated naturally
as the non-slippery model in a usual sense, since the cross action of
the gravity force pushing a runner or craft on a slope to the sides (left or
right) off the $u$-track is cancelled entirely, while the longitudinal
effect is taken as it comes, i.e increasing the forward speed when going
downhill or decreasing it when climbing uphill. Roughly speaking, it is
often adopted in the interpretations of various real world applications that
the cross gravity effect is treated as somewhat ``unwanted'' or disturbing
self-motion (forward) indicated by the velocity $u$, while the along one is
fully accepted. However, there are the situations in nature so that the
approach can be exact opposite. It seems appropriate to mention the animals
that move sideways, e.g. a sidewinder rattlesnake on a desert slope as well
as the linear transverse ship's sliding motion side-to-side called sway on a
dynamic surface of the sea in marine engineering or the description of the
algebraic pedal curves and surfaces in geometry; see, e.g. \cite{cross,
slipperyx, Sabau_pedal}. In our study, both types of
gravity effects will be treated like having the same essence, without distingusihing which non-compensated gravity-additive (called a \textit{slide} in our model) is ``better'' or ``worse''. This means that we can be dragged off
forward-backward and sideways equally well on the slippery slope in the general model proposed.

Summarizing, the hitherto notion of a \textit{slippery slope (problem)} 
is taking much broader meaning in comparison to the preceding studies (\cite{slippery, slipperyx}). From
now on, each $\mathcal{P}=(\eta ,\tilde{\eta}), \eta, \tilde{\eta}\in[0, 1]$  
defines in fact a different and specific navigation problem, where at least
one type of gravity-increment (sliding) occurs. From such point of view only
the Riemannian case, where the impact of gravitational wind is cancelled
completely, i.e. $\eta =\tilde{\eta}=1$ represents the non-slippery
slope in the presence of gravity. Moreover, the current approach yields that
even the slope in the classical Matsumoto model is considered as being
slippery because of the longitudinal \textquotedblleft
slide\textquotedblright\ (i.e. the along-gravity gain) admitted, although it is construed as non-slippery at all in the usual
sense.

Remark that in order to avoid any confusion with the first slippery slope
model linking MAT and ZNP, i.e. $\eta \in \lbrack 0,1],$ $\tilde{\eta}=0$
which was introduced as a \textquotedblleft slippery
slope\textquotedblright\ in \cite{slippery} (called SLIPPERY herein) and its
natural generalization in the current study, we will slightly rename this
particular case now as a \textit{standard slippery slope}, referring to its meaning in the usual
sense and to be in agreement with our previous terminology.

\subsection{Problem formulation and main theorems}

We can now formulate the main task to which the rest of this article is
dedicated. Namely, the problem of time-minimizing navigation on a slippery
mountain slope under the action of gravity is posed in the following way:

\begin{quote}
\begin{itshape}
\noindent {Suppose a walker, craft or a vehicle has a certain constant maximum speed as measured on a horizontal plane, while gravity acts perpendicular to this plane.  Imagine now that the craft endeavours to move  on a slippery slope of a mountain under gravity, admitting a traction-dependent sliding in arbitrary (downward) direction. 
 What path should be followed by the craft to get from one point to another in the least time?}
\end{itshape}
\end{quote}

In the general context of an $n$-dimensional Riemannian manifold with $\mathbf{G}^{T}=-\bar{g}\omega ^{\sharp }$, where $\omega ^{\sharp }$ is the
gradient vector field and $\bar{g}$ is the rescaled gravitational
acceleration $g$ (see \cref{model} and \cite{slippery,Nicprw,cross,slipperyx}), we consider the active wind $\mathbf{G}_{\eta \tilde{\eta}}$ which represents the impact of gravity that is not compensated due to traction on the slippery slope, and defined by \eqref{wind_general}, with $(\eta ,\tilde{\eta})\in \mathcal{\tilde{S}}$, where $\mathcal{\tilde{S}}=[0,1]\times
\lbrack 0,1]$. Mention that it vanishes only when $\eta =\tilde{\eta}=1$, i.e. in the Riemannian case (RIEM for brevity). 
 Let us also fix $\mathcal{S=\tilde{S}}\smallsetminus \{(1,1)\}$.
The set $\mathcal{\tilde{S}}$ represents a complete problem square diagram
for our exposition (see \cref{fig_square2}). 
The key argument to answer to the above question is that in Finsler
geometry, the notion of arc length can be interpreted as time, thus making
the time-minimal paths parametrized by arc length the Finsler geodesics,
subsequently called the \textit{time geodesics} (see also \cite{matsumoto}). More precisely, the solution to the posed problem is given by the new slippery slope metric in the general case, which will be called a \textit{slippery slope metric of type} $(\eta ,\tilde{\eta})$ or an $(\eta ,\tilde{\eta})$\textit{-slope metric} for short, 
as well as the corresponding time geodesics. Our main results are
represented by the following two theorems.

\begin{theorem}
\label{Theorem1}
\textnormal{($(\eta ,\tilde{\eta})$-slope metric)} Let a 
slippery slope of a mountain be an $n$-dimensional Riemannian manifold $%
(M,h) $, $n>1$, with a cross-traction coefficient $\eta \in \lbrack 0,1]$,
an along-traction coefficient $\tilde{\eta}\in \lbrack 0,1]$ and a
gravitational wind $\mathbf{G}^{T}$ on $M$. The time-minimal paths on $%
(M,h)$ under the action of an active wind $\mathbf{G}_{\eta \tilde{\eta}}$
as in \eqref{wind_general} are the geodesics of an $(\eta ,\tilde{\eta})$-slope
metric $\tilde{F}_{\eta \tilde{\eta}}$, which satisfies 
\begin{equation}
\tilde{F}_{\eta \tilde{\eta}}\sqrt{\alpha ^{2}+2(1-\eta )\bar{g}\beta \tilde{%
F}_{\eta \tilde{\eta}}+(1-\eta )^{2}||\mathbf{G}^{T}||_{h}^{2}\tilde{F}%
_{\eta \tilde{\eta}}^{2}}=\alpha ^{2}+(2-\eta -\tilde{\eta})\bar{g}\beta 
\tilde{F}_{\eta \tilde{\eta}}+(1-\eta )(1-\tilde{\eta})||\mathbf{G}%
^{T}||_{h}^{2}\tilde{F}_{\eta \tilde{\eta}}^{2},  \label{TH_mama}
\end{equation}%
with $\alpha =\alpha (x,y),$ $\beta =\beta (x,y)$ given by \eqref{NOT},
where either $||\mathbf{G}^{T}||_{h}<\frac{1}{1-\tilde{\eta}}$ and $(\eta ,%
\tilde{\eta})\in \mathcal{D}_{1}\cup \mathcal{D}_{2}$, or $||\mathbf{G}%
^{T}||_{h}<\frac{1}{2|\eta -\tilde{\eta}|}$ and $(\eta ,\tilde{\eta})\in 
\mathcal{D}_{3}\cup \mathcal{D}_{4}$, where 
\begin{equation*}
\begin{array}{l}
\mathcal{D}_{1}=\left\{ (\eta ,\tilde{\eta})\in \mathcal{S}\text{ }|\text{ }%
\eta \geq \tilde{\eta}>2\eta -1\right\} , 
\qquad  \ \ 
\mathcal{D}_{2}=\left\{ (\eta ,\tilde{\eta})\in \mathcal{S}\text{ }|\text{ }%
\frac{3\tilde{\eta}-1}{2}<\eta <\tilde{\eta}\right\}, \\ 
\\ 
\mathcal{D}_{3}=\left\{ (\eta ,\tilde{\eta})\in \mathcal{S}\text{ }|\text{ }%
\eta \geq \frac{1}{2},\text{ }\tilde{\eta}\leq 2\eta -1\right\}, 
\quad  \ 
\mathcal{D}_{4}=\left\{ (\eta ,\tilde{\eta})\in \mathcal{S}\text{ }|\text{ }%
\tilde{\eta}\geq \frac{1}{3},\text{ }\eta \leq \frac{3\tilde{\eta}-1}{2}%
\right\} ,%
\end{array}%
\end{equation*}%
$\mathcal{S}=\bigcup\limits_{i=1}^{4}\mathcal{D}_{i}$ and $\mathcal{D}%
_{i}\cap \mathcal{D}_{j}=\varnothing ,$ for any $i\neq j,$ $i,j=1,...,4.$ No
 restriction should be imposed on $||\mathbf{G}^{T}||_{h}$ if $\eta =%
\tilde{\eta}=1.$ In particular, a slope metric of type $(0,0)$, $(1,0)$, $(0,1)$, $(1,1)$ is reduced to a Randers metric, a Matsumoto metric,  a cross slope metric and a Riemannian
metric $h$, respectively.
\end{theorem}
Moreover, an $(\eta, 0)$-slope metric is a (standard) 
 slippery slope metric $\tilde{F}_{\eta }$ and a $(0,\tilde{\eta})$-slope
metric is a slippery-cross-slope metric $\tilde{F}_{\tilde{\eta}}$, both belonging to the class of the general $(\alpha ,\beta )$-metrics in Finsler geometry, which have been introduced and discussed recently
in \cite[Thm. 1.1]{slippery} and  \cite[Thm. 1.1]{slipperyx}, respectively. 
For clarity's sake, the partition of $\mathcal{\tilde{S}}$ into the mutually disjoint subsets $\mathcal{D}_{i},i=1,...,4$ is illustrated  in \cref{partition} on further
reading. It is also worth of mentioning that the above theorem now encompasses as particular cases  the solutions to: the original  
 Matsumoto's slope-of-a-mountain problem (MAT), Zermelo's navigation problem (ZNP) on a Riemannian manifold under a gravitational wind $\mathbf{G}^{T}$ as well as CROSS. Furthermore, $\tilde{F}_{\eta \tilde{\eta}}$ provides a new Finsler metric of general $(\alpha ,\beta )$ type, having the indicatrix $I_{\tilde{F}_{\eta \tilde{\eta}}}$ (i.e. $\tilde{F}_{\eta \tilde{\eta}}=1)$
defined by
\begin{equation}
\label{CCCINDICATRIX}
\sqrt{\alpha ^{2}+2(1-\eta )\bar{g}\beta +(1-\eta )^{2}||\mathbf{G}%
^{T}||_{h}^{2}}=\alpha ^{2}+(2-\eta -\tilde{\eta})\bar{g}\beta +(1-\eta )(1-%
\tilde{\eta})||\mathbf{G}^{T}||_{h}^{2}.  
\end{equation}

The second theorem enables us to find the time geodesics that correspond to
the $(\eta ,\tilde{\eta})$-slope metric (restricted to the indicatrix $I_{\tilde{F}_{\eta \tilde{\eta}}}$), giving an answer to the research question posed above. Namely, we have obtained

\begin{theorem}
\label{Theorem2} \textnormal{(Time geodesics)} 
Let a slippery slope of a mountain be an $n$-dimensional Riemannian manifold $(M,h) $, $n>1$, with a cross-traction coefficient $\eta \in \lbrack 0,1]$,
an along-traction coefficient $\tilde{\eta}\in \lbrack 0,1]$ and a gravitational wind $\mathbf{G}^{T}$ on $M$. The time-minimal paths on $(M,h)$ under the action of an active wind $\mathbf{G}_{\eta \tilde{\eta}}$
as in \eqref{wind_general} are the time-parametrized solutions $\gamma (t)=(\gamma^{i}(t)),$ $i=1,...,n$ of the ODE system 
\begin{equation}
\ddot{\gamma}^{i}(t)+2\tilde{\mathcal{G}}_{\eta \tilde{\eta}}^{i}(\gamma (t),%
\dot{\gamma}(t))=0,  \label{GGG}
\end{equation}%
where%
\begin{eqnarray*}
\tilde{\mathcal{G}}_{\eta \tilde{\eta}}^{i}(\gamma (t),\dot{\gamma}(t)) &=&%
\mathcal{G}_{\alpha }^{i}(\gamma (t),\dot{\gamma}(t))+\left[ \tilde{\Theta}%
(r_{00}+2\alpha ^{2}\tilde{R}r)+\alpha \tilde{\Omega}r_{0}\right] \frac{\dot{%
\gamma}^{i}(t)}{\alpha } \\
&&-\left[ \tilde{\Psi}(r_{00}+2\alpha ^{2}\tilde{R}r)+\alpha \tilde{\Pi}r_{0}%
\right] \frac{w^{i}}{\bar{g}}-\tilde{R}w_{\text{ }|j}^{i}\frac{\alpha
^{2}w^{j}}{\bar{g}^{2}},
\end{eqnarray*}%
with%
\begin{equation}
\begin{array}{l}
\mathcal{G}_{\alpha }^{i}(\gamma (t),\dot{\gamma}(t))=\frac{1}{4}%
h^{im}\left( 2\frac{\partial h_{jm}}{\partial x^{k}}-\frac{\partial h_{jk}}{%
\partial x^{m}}\right) \dot{\gamma}^{j}(t)\dot{\gamma}^{k}(t),\qquad \mathit{%
\tilde{\Psi}}=\frac{\bar{g}^{2}\alpha ^{2}}{2\tilde{E}}[\alpha ^{4}\tilde{A}%
^{2}\tilde{B}+(\tilde{\eta}-\eta )^{2}], \\ 
~ \\ 
r_{00}=-\frac{1}{\bar{g}}w_{j|k}\dot{\gamma}^{j}(t)\dot{\gamma}%
^{k}(t),\qquad r_{0}=\frac{1}{\bar{g}^{2}}w_{j|k}\dot{\gamma}%
^{j}(t)w^{k},\qquad r=-\frac{1}{\bar{g}^{3}}w_{j|k}w^{j}w^{k}, \\ 
~ \\ 
\tilde{R}=\frac{(1-\eta )\bar{g}^{2}}{2\alpha ^{4}\tilde{B}}[\left( 1-\tilde{%
\eta}\right) \alpha ^{2}\tilde{B}-(\tilde{\eta}-\eta )],\qquad \mathit{%
\tilde{\Theta}}=\frac{\bar{g}\alpha }{2\tilde{E}}[\alpha ^{6}\tilde{A}\tilde{%
B}^{2}-(\tilde{\eta}-\eta )^{2}\bar{g}\beta ], \\ 
~ \\ 
\mathit{\tilde{\Omega}}=\frac{(1-\eta )\bar{g}^{2}}{\alpha ^{2}\tilde{B}%
\tilde{E}}\{[\left( 1-\tilde{\eta}\right) \alpha ^{2}\tilde{B}-(\tilde{\eta}%
-\eta )][\alpha ^{6}\tilde{B}^{3}+(\tilde{\eta}-\eta )^{2}||\mathbf{G}%
^{T}||_{h}^{2}]-(\tilde{\eta}-\eta )^{2}\alpha ^{2}(\bar{g}\beta \tilde{B}+||%
\mathbf{G}^{T}||_{h}^{2}\tilde{A})\}, \\ 
~ \\ 
\mathit{\tilde{\Pi}}=\frac{(1-\eta )\bar{g}^{3}}{2\alpha ^{3}\tilde{B}\tilde{%
E}}\{[\left( 1-\tilde{\eta}\right) \alpha ^{2}\tilde{B}-(\tilde{\eta}-\eta
)][2\alpha ^{6}\tilde{A}\tilde{B}^{2}-(\tilde{\eta}-\eta )^{2}\bar{g}\beta
]+(\tilde{\eta}-\eta )^{2}\alpha ^{2}\tilde{B}[2\alpha ^{2}+(1-\eta )\bar{g}%
\beta ]\}, \\ 
~ \\ 
\tilde{A}=-\frac{1}{\alpha ^{2}}\{(1-\eta )\left[ 1-(2-\eta -\tilde{\eta}%
)\left( 1-\tilde{\eta}\right) ||\mathbf{G}^{T}||_{h}^{2}\right] -(2-\eta -%
\tilde{\eta})^{2}\bar{g}\beta -(2-\eta -\tilde{\eta})\alpha ^{2}\}, \\ 
~ \\ 
\tilde{B}=-\frac{1}{\alpha ^{2}}\{[1-2(1-\eta )(1-\tilde{\eta})||\mathbf{G}%
^{T}||_{h}^{2}]-2(2-\eta -\tilde{\eta})\bar{g}\beta -2\alpha ^{2}\}, \\ 
~ \\ 
\tilde{C}=\frac{1}{\alpha }\left( \alpha ^{2}\tilde{B}+\bar{g}\beta \tilde{A}%
\right) ,\qquad \tilde{E}=\alpha ^{6}\tilde{B}\tilde{C}^{2}+(||\mathbf{G}%
^{T}||_{h}^{2}\alpha ^{2}-\bar{g}^{2}\beta ^{2})[\alpha ^{4}\tilde{A}^{2}%
\tilde{B}+(\tilde{\eta}-\eta )^{2}]%
\end{array}
\label{geo_tilde}
\end{equation}%
and $\alpha =\alpha (\gamma (t),\dot{\gamma}(t)),$ $\beta =\beta (\gamma (t),%
\dot{\gamma}(t))$, and $w^{i}$ denoting the components of $\mathbf{G}^{T}$. 
Moreover, along the time-minimal paths $\gamma (t)$, equation \eqref{CCCINDICATRIX} holds. 
\end{theorem}

\ 

The paper is organized in the following way. In \cref{Sec_2} we recall
some general notions and results regarding Riemann-Finsler geometry that
will serve as the necessary background in the proofs of the main results. In
our exposition, we keep in mind \cref{fig_square2}, which represents
the \textquotedblright map\textquotedblright\ of all problems $\mathcal{P}%
_{\eta ,\tilde{\eta}},$ $(\eta ,\tilde{\eta})\in \mathcal{\tilde{S}}$.
Taking into account both components  (transverse and longitudinal) of the gravitational wind $\mathbf{G}^{T}$ with respect to the direction of motion $u$, which are controlled by the traction coefficients $(\eta ,\tilde{\eta}%
)\in \mathcal{\tilde{S}}$, we significantly extend Matsumoto's slope-of-a-mountain problem 
$\mathcal{P}_{MAT}$ (or $\mathcal{P}_{1,0}),$ covering the whole square $\mathcal{%
\tilde{S}}$ and including the previous studies in this regard  \cite{slippery,cross,slipperyx}
(for clarity, in \cref{fig_square2}, $\mathcal{T}_{1,0}^{0,0}$ describes SLIPPERY, CROSS is defined by $\mathcal{P}_{0,1}$ and $\mathcal{T}_{0,1}^{0,0}$ stands for S-CROSS). Our
concern in \cref{model} is focused on the description of the general model
of a slippery mountain slope under gravity (precisely, \cref{Sec_3.2}),
starting with some special navigation problems (the so-called  reduced ZNP, reduced MAT and reduced CROSS) which are achieved by the transition
from the initial Riemannian background $\mathcal{P}_{RIEM}$ (or $\mathcal{P}%
_{1,1})$ to the Zermelo navigation problem (ZNP) under a weak wind $%
(1-\eta )\mathbf{G}^{T}$, Matsumoto slope-of-a-mountain problem (MAT)
and the cross slope problem (CROSS), respectively. Then, in \cref{Sec_4}, we
perform the proof of \cref{Theorem1}, dividing it into two steps including a
sequence of cases and lemmas, which gradually establish all statements of \cref{Theorem1}. The first one refers to the direction-dependent deformation of the background Riemannian metric $h$ by
the vector field $(\eta -\tilde{\eta})\mathbf{G}_{MAT}$, $\eta \neq \tilde{%
\eta}$, describing the effect of both traction coefficients.  
Precesily, we state that only under the assumption $\left. |\eta -%
\tilde{\eta}|\text{ }||\mathbf{G}_{MAT}||_{h}<1,\right. $ the
anisotropic deformation of $h$ by $(\eta -\tilde{\eta})\mathbf{G}%
_{MAT}$ provides a Finsler metric $F(x,y)=\frac{\alpha ^{2}}{\alpha -(\eta -%
\tilde{\eta})\bar{g}\beta }$ of Matsumoto type which is strongly convex when $||\mathbf{G}^{T}||_{h}<\frac{1}{2|\eta -\tilde{\eta}|}$ (%
\cref{Lema1,Lema2}). 
The outcome of the first step (i.e. the Matsumoto type metric $F$ for $\eta \neq \tilde{\eta})$ is the main tool for the next step from the perspective of the Zermelo navigation \cite{CJS,SH}.  By applying %
\cref{Prop3}, the second step develops a rigid translation of the indicatrix
of $F$ (provided in the first step) or $h$ if $\eta =\tilde{\eta}$ by the
rescaled gravitational wind $(1-\eta )\mathbf{G}^{T},$ restricted to $%
F(x,-(1-\eta )\mathbf{G}^{T})<1$. This leads to classical Finsler metrics,
namely the $(\eta ,\tilde{\eta})$-slope metrics, besides the necessary and
sufficient conditions for their strong convexity (\cref{Lema3}). In \cref{Sec_5}, we prove \cref{Theorem2},
deriving the spray coefficients related to the $(\eta ,\tilde{\eta})$-slope
metric $\tilde{F}_{\eta \tilde{\eta}}.$ This is based on an implicit
expression of the $(\eta ,\tilde{\eta})$-slope metrics, which belong to Finsler metrics of general $(\alpha ,\beta )$ type (\cref{PropXX}) and some technical
results formulated in \cref{Lema4,Lema5,Prop5}, including a key argument
that any time geodesic of an $(\eta ,\tilde{\eta})$-slope metric is unitary
with respect to $\tilde{F}_{\eta \tilde{\eta}}$ because it is a trajectory
in Zermelo's navigation developed in the second step. Therefore,  the 
time-minimal paths provided by \cref{Theorem2} represent the geodesics
of $\tilde{F}_{\eta \tilde{\eta}}$ restricted to the indicatrix $I_{\tilde{F}%
_{\eta \tilde{\eta}}}$. Finally, in \cref{Sec_Examples}, we conduct a $2$%
-dimensional study with a few examples, analyzing and comparing the behavior
of time geodesics and the evolution of time fronts under various gravity
effects determined by the cross- and along-traction 
as well as the gravitational wind force on the slippery slopes.

\section{Preliminaries}

\label{Sec_2}In this section, we briefly recall the notions and general facts from Riemann-Finsler geometry that are needed for presenting and proving our aforementioned results; for more details, see, for example, \cite{chern_shen,B-Miron,colleenshen,SH,Yu,Kristaly,Y-Sabau,CJS,Musznay}.

\subsection{Finsler manifolds}

Let $(M,h)$ be a Riemannian manifold, where $M$ is an $n$-dimensional $C^{\infty
} $-manifold, $n>1,$ and $h$ is a Riemannian metric on $M$. Let $T_{x}M$ be the
tangent space at $x\in M$ and $(x^{i}),$ $i=1,...,n$ be the local
coordinate system on a local chart in $x\in M$. The set $\left\{ \frac{%
\partial }{\partial x^{i}}\right\} ,$ $i=1,...,n$ denotes the natural basis
for the tangent bundle $TM=\underset{x\in M}{\cup }T_{x}M$ which is itself a 
$C^{\infty }$-manifold. Thus, for every $y\in T_{x}M$, one has $y=y^{i}\frac{%
\partial }{\partial x^{i}}$ and the coordinates on a local chart in $%
(x,y)\in TM$ are denoted by $(x^{i},y^{i}),$ $i=1,...,n$.

A natural generalization of a Riemannian metric is a \textit{Finsler metric}. Specifically, the pair $(M,F)$ is a Finsler manifold if $F:TM\rightarrow
\lbrack 0,\infty )$ is a continuous function with the following properties:

i) $F$ is a $C^{\infty }$-function on the slit tangent bundle $TM_{0}=TM\backslash \{0\}$;

ii) $F$ is positively homogeneous of degree one with respect to $y$, i.e. $%
F(x,{c}y)={c}F(x,y)$, for all ${c}>0$;

iii) the Hessian $g_{ij}(x,y)=\frac{1}{2}\frac{\partial ^{2}F^{2}}{\partial
y^{i}\partial y^{j}}$ is positive definite for all $(x,y)\in TM_{0}.$

\noindent Denoting by $I_{F}=\left\{ (x,y)\in TM\text{ }|\text{ }%
F(x,y)=1\right\} $ the indicatrix of $F$, one can remark that the property
iii) refers to the fact that $I_{F}$ is strongly convex. In particular, the
Finsler metric $F$ is a Riemannian metric if and only if $g_{ij}(x,y)$ does
not depend on $y,$ i.e. $g_{ij}(x,y)=g_{ij}(x).$

Let $\mathcal{A}$ be a conic open subset of $TM_{0}$. According to \cite%
{CJS,JS,JPS} this means that for each $x\in M,$ $\mathcal{A}_{x}=\mathcal{A\
\cap \ }T_{x}M$ is a conic subset, i.e. if $y\in \mathcal{A}_{x},$ then $%
\lambda y$ $\in \mathcal{A}_{x}$ for every $\lambda >0$. In particular, a 
\textit{conic Finsler metric} is a Finsler metric on $\mathcal{A}$, i.e. $F:%
\mathcal{A}\rightarrow \lbrack 0,\infty )$ is a continuous function
satisfying i), ii) and iii) $\ $for all $(x,y)\in \mathcal{A}$ (see \cite%
{CJS,JS}).

A smooth vector field on $TM_{0},$ locally expressed by $S=y^{i}\frac{%
\partial }{\partial x^{i}}-2\mathcal{G}^{i}\frac{\partial }{\partial y^{i}},$
is called a \textit{spray} on $M$. The functions $\mathcal{G}^{i}=\mathcal{G}%
^{i}(x,y),$ $i=1,...,n$ are positively\ homogeneous of degree two with
respect to $y,$ i.e. $\mathcal{G}^{i}(x,{c}y)={c}^{2}\mathcal{G}^{i}(x,y)$,
for all ${c}>0,$ and they are called the \textit{spray coefficients} \cite%
{chern_shen}. In the case where the spray is induced by a Finsler metric ${F=%
}\sqrt{g_{ij}(x,y)y^{i}y^{j}}$, the spray coefficients are given by 
\begin{equation}
\mathcal{G}^{i}(x,y)=\frac{1}{4}g^{il}\big\{\lbrack {F}%
^{2}]_{x^{k}y^{l}}y^{k}-[{F}^{2}]_{x^{l}}\big\}=\frac{1}{4}g^{il}\left( 2%
\frac{\partial g_{jl}}{\partial x^{k}}-\frac{\partial g_{jk}}{\partial x^{l}}%
\right) y^{j}y^{k},  \label{S1}
\end{equation}%
$(g^{il})$ being the inverse matrix of $(g_{il})$.

Let us consider a regular piecewise $C^{\infty }$-curve on $M,$ $\gamma
:[0,1]\rightarrow M$, $\gamma (t)=(\gamma ^{i}(t)),$ $i=1,...,n,$ where the
velocity vector of $\gamma $ is denoted by $\dot{\gamma}(t)=\frac{d\gamma 
}{dt}.$ The curve $\gamma $ is called $F$\textit{-geodesic} if $\dot{%
\gamma}(t)$ is parallel along the curve, i.e. in the local coordinates, $%
\gamma ^{i}(t),$ $i=1,...,n$ are the solutions of the ODE system%
\begin{equation}
\ddot{\gamma}^{i}(t)+2\mathcal{G}^{i}(\gamma (t),\dot{\gamma}(t))=0.
\label{geo1}
\end{equation}

It is worthwhile to mention that Zermelo's navigation, apart from
the fact that it is a classical optimal control problem, provides a technique
to construct a new Finsler metric by perturbing a given Finsler metric (the
so-called background metric) by a vector field $W,$ i.e. a time-independent
wind on a manifold $M$, under some constraints. In particular, by
considering that the background metric is a Riemannian one, denoted by $h,$
the Randers metric solves Zermelo's problem of navigation in the case of
weak wind $W$, i.e. $||W||_{h}<1$ \cite{colleenshen,chern_shen}. When $W$
is a critical wind, i.e. $||W||_{h}=1,$ the problem is solved by the Kropina
metric \cite{Y-Sabau}. In this respect, we mention the following result 
(see \cite[Lemma 3.1]{SH}, \cite[Proposition 2.14]{CJS}, \cite[Lemma 1.4.1]%
{chern_shen}).

\begin{proposition}
\label{Prop3} Let $(M,F)$ be a Finsler manifold and $W$ a vector field on $M$
such that $F(x,-W)<1$. Then the solution of the Zermelo's navigation problem
with the navigation data $(F,W)$ is a Finsler metric $\tilde{F}$ obtained by
solving the equation%
\begin{equation}
F(x,y-\tilde{F}(x,y)W)=\tilde{F}(x,y),\text{ }  \label{MAIN}
\end{equation}%
for any nonzero $y\in T_{x}M$, $x\in M.$
\end{proposition}

\noindent Since the indicatrix $I_{F}$ is strongly convex and it is assumed
that $F(x,-W)<1$, \eqref{MAIN} admits a unique positive solution $\tilde{%
F}$ for any nonzero $y\in T_{x}M$ \cite{CJS,SH}. Another key remark
regarding the inequality $F(x,-W)<1$ is that it assures the fact that $%
\tilde{F}$ is a Finsler metric, having the indicatrix $I_{\tilde{F}}=\left\{
(x,y)\in TM\text{ \ }|\text{ }\tilde{F}(x,y)=1\right\} $ strongly convex, as
well as, for any $x\in M$, $y=0$ belongs to the region bounded by $I_{\tilde{%
F}}$; for more details, see \cite{CJS}. Additionally, any regular
piecewise $C^{\infty }$-curve $\gamma :[0,1]\rightarrow M$, parametrized by
time, that represents a trajectory in Zermelo's navigation problem has unit 
$\tilde{F}$-length, i.e. $\tilde{F}(\gamma (t),\dot{\gamma}(t))=1,$ where $%
\dot{\gamma}(t)$ is the velocity vector \cite[Lemma 1.4.1]{chern_shen}.

\subsection{General $(\protect\alpha ,\protect\beta )$-metrics}

Various examples of Finsler manifolds can be found in the literature and a few of
them are outlined in the sequel. Let $\alpha ^{2}=a_{ij}(x)y^{i}y^{j}$ be
a quadratic form, where $a_{ij}(x)$ is a Riemannian metric on $M$ and $\beta
=b_{i}(x)dx^{i}$ be a differential $1$-form on $M$, also expressed as $\beta
=b_{i}y^{i}$. The pair $(M,F)$ is called Finsler manifold with \textit{%
general} $(\alpha ,\beta )$-\textit{metric }if the Finsler metric $F$ can be
read as $F=\alpha \phi (b^{2},s)$, where $\phi (b^{2},s)$ is a positive $%
C^{\infty }$-function in the variables $b^{2}=||\beta ||_{\alpha
}^{2}=a^{ij}b_{i}b_{j}$ and $s=\frac{\beta }{\alpha }$, with $|s|\leq
b<b_{0} $ and $0<b_{0}\leq \infty $; for more details, see \cite{Yu}. 
The examples of general $(\alpha ,\beta )$-metrics are provided by the slippery
slope and slippery-cross-slope metrics, which have been presented recently in \cite%
{slippery,cross,slipperyx}. In the case where $\phi $ depends only on the
variable $s$, the function $F=\alpha \phi (s)$ is known as $(\alpha ,\beta )$%
-\textit{metric}. Such an example is the Randers metric $F=\alpha +\beta ,$
with $\phi (s)=1+s$, which solves Zermelo's navigation problem under the
influence of a weak wind, i.e. $|s|\leq b<1$ \cite{chern_shen}. Another
example is provided by the Matsumoto metric $F=\frac{\alpha ^{2}}{\alpha
-\beta },$ with $\phi (s)={\frac{1}{1-s}}$ and $|s|\leq b<{\frac{1}{2}},$
which carries out the solution to Matsumoto's slope-of-a-mountain problem 
\cite{matsumoto}.

From the theory of the general $(\alpha, \beta )$-metrics we recall only
a few key results for our arguments.

\begin{proposition}
\label{Prop1} \cite{Yu} Let $M$ be an $n$-dimensional manifold. $F=\alpha
\phi (b^{2},s)$ is a Finsler metric for any Riemannian metric $\alpha $ and $%
1$-form $\beta ,$ with $||\beta ||_{\alpha }<b_{0}$ if and only if $\phi
=\phi (b^{2},s)$ is a positive $C^{\infty }$-function satisfying%
\begin{equation*}
\phi -s\phi _{2}>0,\text{ \ \ }\phi -s\phi _{2}+(b^{2}-s^{2})\phi _{22}>0,
\end{equation*}%
when $n\geq 3$ or%
\begin{equation*}
\phi -s\phi _{2}+(b^{2}-s^{2})\phi _{22}>0,
\end{equation*}%
when $n=2$, where $s=\frac{\beta }{\alpha }$ and $b=||\beta ||_{\alpha }$
satisfy $|s|\leq b<b_{0}$.
\end{proposition}

\noindent We notice that $\phi _{1}$ and $\phi _{2}$ denote the derivatives
of the function $\phi $ with respect to the first variable $b^{2}$ and the
second variable $s$, respectively. Similarly, $\phi _{12}$ and $\phi _{22}$
denote the derivatives of $\phi _{1}$ and $\phi _{2}$ with respect to $s.$
When $\phi $ is a function only of variable $s$, the derivatives $\phi _{2}$
and $\phi _{22}$ are simply denoted by $\phi ^{\prime }$ and $\phi ^{\prime
\prime }$, respectively.

To conclude the presentation of the desired results, we also need to recall
the following notation 
\begin{equation}
\begin{array}{l}
r_{ij}=\frac{1}{2}(b_{i|j}+b_{j|i}),\quad r_{i}=b^{j}r_{ji},\quad
r^{i}=a^{ij}r_{j},\quad r_{00}=r_{ij}y^{i}y^{j},\quad r_{0}=r_{i}y^{i},\quad
r=b^{i}r_{i},\quad \\ 
~ \\ 
s_{ij}=\frac{1}{2}(b_{i|j}-b_{j|i}),\quad s_{i}=b^{j}s_{ji},\quad
s^{i}=a^{ij}s_{j},\quad s_{0}^{i}=a^{ij}s_{jk}y^{k},\quad s_{0}=s_{i}y^{i},%
\end{array}
\label{rs}
\end{equation}%
with $b^{j}=a^{ji}b_{i}$, $b_{i|j}=\frac{\partial b_{i}}{\partial x^{j}}%
-\Gamma _{ij}^{k}b_{k}$ and $\Gamma _{ij}^{k}=\frac{1}{2}a^{km}\left( \frac{%
\partial a_{jm}}{\partial x^{i}}+\frac{\partial a_{im}}{\partial x^{j}}-%
\frac{\partial a_{ij}}{\partial x^{m}}\right) $ being the Christoffel
symbols of the Riemannian metric $a_{ij}$. We point out that the
differential $1$-form $\beta $ is closed if and only if $s_{ij}=0$ (see 
\cite{chern_shen}).

\begin{proposition}
\label{Prop2} \cite{Yu} For a general $(\alpha ,\beta )$\textit{-}metric $%
F=\alpha \phi (b^{2},s)$, its spray coefficients $\mathcal{G}^{i}$ are
related to the spray coefficients $\mathcal{G}_{\alpha }^{i}$ of $\alpha $ by%
\begin{eqnarray*}
\mathcal{G}^{i} &=&\mathcal{G}_{\alpha }^{i}+\alpha Qs_{0}^{i}+\left[
\Theta (-2\alpha Qs_{0}+r_{00}+2\alpha ^{2}Rr)+\alpha \Omega
(r_{0}+s_{0})\right] \frac{y^{i}}{\alpha } \\
&&+\left[ \Psi (-2\alpha Qs_{0}+r_{00}+2\alpha ^{2}Rr)+\alpha \Pi
(r_{0}+s_{0})\right] b^{i}-\alpha ^{2}R(r^{i}+s^{i}),
\end{eqnarray*}%
where%
\begin{eqnarray*}
Q &=&\frac{\phi _{2}}{\phi -s\phi _{2}},\qquad \qquad \qquad \qquad \text{\
\ \ \ }\Theta  \ = \ \frac{(\phi -s\phi _{2})\phi _{2}-s\phi \phi _{22}}{%
2\phi \lbrack \phi -s\phi _{2}+(b^{2}-s^{2})\phi _{22}]}, \\
\Psi &=&\frac{\phi _{22}}{2[\phi -s\phi _{2}+(b^{2}-s^{2})\phi _{22}]},\quad
\ \ \ \Pi \ =\  \frac{(\phi -s\phi _{2})\phi _{12}-s\phi _{1}\phi _{22}}{%
(\phi -s\phi _{2})[\phi -s\phi _{2}+(b^{2}-s^{2})\phi _{22}]}, \\
\Omega &=&\frac{2\phi _{1}}{\phi }-\frac{s\phi +(b^{2}-s^{2})\phi _{2}}{\phi 
}\Pi ,\qquad R\ = \ \frac{\phi _{1}}{\phi -s\phi _{2}}.
\end{eqnarray*}
\end{proposition}

\section{Model of a slippery mountain slope under gravity}

\label{model}
Let $(M,h)$ be an $n$-dimensional Riemannian manifold, $n>1,$ which
represents a model for a slippery slope of a mountain. Let $\omega
^{\sharp }=h^{ji}\frac{\partial p}{\partial x^{j}}\frac{\partial }{\partial
x^{i}}$ be the gradient vector field of $p$, where $p:M\rightarrow \mathbb{R}
$ is a $C^{\infty }$-function on $M.$ Making use of $\omega ^{\sharp }$, we
have defined the gravitational wind $\mathbf{G}^{T}=-\bar{g}\omega ^{\sharp }
$, where $\bar{g}$ is the rescaled magnitude of the acceleration of gravity $%
g$, i.e. $\bar{g}=\lambda g,$ $\lambda >0$ \cite%
{slippery,Nicprw,cross,slipperyx}. We note that if the slope of a 
mountain $(M,h)$ is a surface given by $z=f(x^{1},x^{2}),$ where $f$ is a $%
C^{\infty}$-function on $M$, then $p$ is the height function
of the surface, i.e. $p=f(x^{1},x^{2})$.  Based on scaling, we assume throughout
this section that we work with the self-velocity $u$ of a moving craft on
the slope with $||u||_{h}=\sqrt{h(u,u)}=1,$ as one often sees in the
literature on the Zermelo navigation (e.g. \cite{colleenshen}). Along
this section we also refer  to a $2$-dimensional model for the slope, more
precisely, to the inclined plane for a better view of the study, although
the time-optimal navigation problems described in this work are valid for an arbitrary
dimension.

\subsection{Some special cases}

Looking at the known cases visualized graphically in \cref{fig_square_0}, we
first propose some new scenarios. Namely, we are going to consider three
special transitions with varying tractions, linking the Riemannian set-up with the
Matsumoto, Zermelo and cross slope landscapes. For clarity's sake, see improved %
\cref{fig_square_1} in this regard. As we can observe right below, it is
possible to obtain the explicit form of the Finsler metrics in the first
two problems, which are respectively of the Randers and Matsumoto type,
including the parameter $\eta $ or $\tilde{\eta}$ in their expressions in
addition.

We begin by referring to the Zermelo navigation problem, where the purely
geometric solution is given by a Finsler metric of Randers type, under the
action of a space-dependent weak vector field $W$ on a Riemannian manifold $%
(M,h)$ \cite{SH,chern_shen,colleenshen,Zer}.

\subsubsection{Reduced ZNP}

\label{sec_RZNP}

Let us take into account the setting $\eta =\tilde{\eta}\in[0,1].$ In the sequel, this
scenario is called the \textit{reduced Zermelo navigation problem} (R-ZNP
for brevity) and it is marked illustratively by the dashed blue diagonal in %
\cref{fig_square_1}. Since the traction parameters run through the full
range, we can connect the Riemannian and Zermelo cases directly by the
transition $\mathcal{T}_{1,1}^{0,0}$ along the diagonal $\eta =\tilde{\eta%
}$. More precisely, the resultant velocity in this case is defined as
follows 
\begin{equation}
v_{_{R-ZNP}}=u+(1-\eta )\mathbf{G}^{T},  \label{vr-znp}
\end{equation}%
for any $\eta \in \lbrack 0,1]$. In other words, both cross and
effective winds are compensated respectively by the cross- and
along-traction coefficients at the same time and relatively equally, i.e. $P\in{OC}$ in \cref{fig_slope_general}.  

Now we are in position to apply the Zermelo's navigation technique as in 
\cite{chern_shen, colleenshen}, with the navigation data $(h,W)$, namely, deforming the Riemannian metric $h$ by a weak
vector field $W=(1-\eta )\mathbf{G}^{T}$, i.e. $||W||_{h}=(1-\eta )||\mathbf{G}^{T}||_{h}<1$, which stands for the strong convexity condition in this case. Namely, by using the assumption on the self-speed $%
||u||_{h}=1$ and \eqref{vr-znp}, we get $||v_{_{R-ZNP}}||_{h}^{2}-2(1-\eta )h(v_{_{R-ZNP}},\mathbf{G}^{T})-1+(1-\eta )^{2}||\mathbf{G}%
^{T}||_{h}^{2}=0$. This leads to the Finsler metric of Randers type $\tilde{F%
}_{_{R-ZNP}}$ including either of the traction coefficients as a parameter. The result is 
\begin{equation}
\tilde{F}_{_{R-ZNP}}(x,y)=\frac{\sqrt{[(1-\eta )h(y,\mathbf{G}^{T})]^{2}+\lambda _{\eta }||y||_{h}^{2}}}{\lambda _{\eta }}-\frac{(1-\eta
)h(y,\mathbf{G}^{T})}{\lambda _{\eta }},  \label{metric_r-znp}
\end{equation}
with $\lambda _{\eta }=1-(1-\eta )^{2}||\mathbf{G}^{T}||_{h}^{2},$ for any  $%
(x,y)\in TM$. In particular, if $\eta =0$, i.e. $P=C$ in %
\cref{fig_slope_general}, then the last equation yields the standard
Randers metric, which represents the solution to ZNP under the action of a
full gravitational wind $\mathbf{G}^{T}$. On the other hand, for $\eta =1$
we get the non-slippery Riemannian case, i.e. $P=O$ in %
\cref{fig_slope_general}. 

We notice that for every $\eta \in \lbrack 0,1]$ (along the diagonal $\eta =\tilde{\eta}$), the indicatrices of the Randers type metrics $\tilde{F}%
_{_{R-ZNP}}$ are the ``clonated'' ellipsoids because by Zermelo's navigation, 
the Riemannian $h$-circle (ellipsoid) is only rigidly translated by $(1-\eta
)\mathbf{G}^{T}$, under the restriction $(1-\eta )||\mathbf{G}^{T}||_{h}<1$; see, e.g. \cite{chern_shen}. In particular, there is not any anisotropic
deformation of the indicatrices and only rigid translation is applied, while
transiting between two arbitrary problems in R-ZNP.

It is useful to consider some rather simple examples to gain
some intuition. For instance, one can compare S-CROSS with $\tilde{\eta}=0$
and $||\mathbf{G}^{T}||_{h}=0.5$ discussed on the inclined plane in \cite%
{slipperyx} to R-ZNP with $\eta =\tilde{\eta}=0.5$ and $||\mathbf{G}%
^{T}||_{h}=1$. Clearly, the former means the classical ZNP, so the
gravitational wind blows in full, however its force has been halved. Both
scenarios yield the same outcome, representing R-ZNP in the general context
as shown in \cref{fig_square_1}. For the sake of clarity, see, e.g. the
corresponding indicatrices' behaviors (dashed blue) in dimension 2 on the
same planar slope in \cite[Figure 7 (top left)]{slipperyx} concerning S-CROSS and \cref{fig_plane_indi2} (top left) in the general model (\cref{Sec_6.1}). By analogy, the same
effect will be obtained if $\mathbf{G}^{T}$-forces are almost doubled in
both problems, i.e. 0.99 in the former and 1.99 in the latter, as also
illustrated in the above mentioned figures (solid blue). Roughly speaking,
what is reduced in gravitational wind force in ZNP can be regained by traction(s) in R-ZNP, and the other way around in order to keep the
same initial situation.

\begin{figure}[h!]
\centering
~\includegraphics[width=0.49\textwidth]{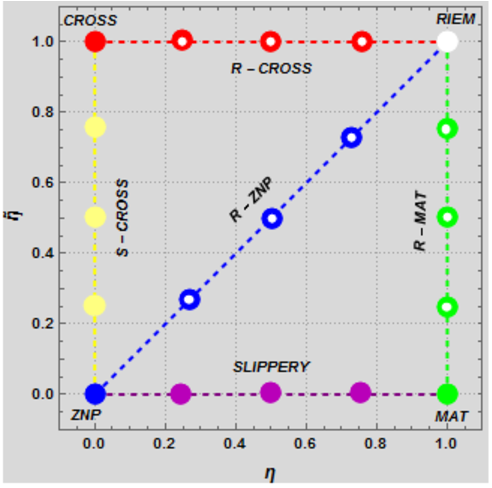}
\caption{The extended problem diagram (square) including the new cases: R-ZNP,
R-MAT and R-CROSS; cf. \cref{fig_square_0}. }
\label{fig_square_1}
\end{figure}

\subsubsection{Reduced MAT}

Now we set $\eta =1$ and $\tilde{\eta}\in \lbrack 0,1]$, calling
this scenario the \textit{reduced Matsumoto slope-of-a-mountain problem}
(R-MAT for short). The along-traction coefficient $\tilde{\eta}$ runs
through the entire range, so RIEM and MAT can be linked directly by the
transition $\mathcal{T}_{1,1}^{1,0}$ which is also included in the slippery slope
model. This situation is indicated by the dashed green segment in %
\cref{fig_square_1}. Thus, the equation of motion for this case reads 
\begin{equation}
v_{_{R-MAT}}=u+(1-\tilde{\eta})\mathbf{G}_{MAT},  \label{vr-mat}
\end{equation}%
for any $\tilde{\eta}\in \lbrack 0,1],$ where $\mathbf{G}_{MAT}$ is the
orthogonal projection of  $\mathbf{G}^{T}$ on $u$. Moreover, it follows
immediately that $P\in{OA}$ in \cref{fig_slope_general}, so the
cross wind is vanished, whilst the effective wind is scaled by the
along-traction coefficient $\tilde{\eta}$. Since the velocities 
$v_{_{R-MAT}}$ and $u$ are always collinear in this case, R-MAT is based on a direction-dependent deformation of the background Riemannian metric $h$ by the vector field $(1-\tilde{\eta})\mathbf{G}_{MAT}.$ By applying the navigation technique as in \cite{chern_shen, colleenshen}, where $||v_{_{R-MAT}}||_{h}=1\pm||(1-\tilde{\eta})\mathbf{G}_{MAT}||_{h}$ (\textquotedblleft +\textquotedblright\ for downhill and
\textquotedblleft -\textquotedblright\ for uphill motion), as well as \cite[Step I]{slipperyx}, the resultant Finsler metric is
obtained explicitly. More precisely, it is the $(\alpha ,\beta )$-metric of Matsumoto type including the along-traction coefficient $\tilde{\eta}\in\lbrack 0,1]$ as a parameter, and denoted by $\tilde{F}_{_{R-MAT}}$. We thus get 
\begin{equation}
\tilde{F}_{_{R-MAT}}(x,y)=\frac{||y||_{h}^{2}}{||y||_{h}+(1-\tilde{\eta})h(y,%
\mathbf{G}^{T})},  \label{Fr-mat}
\end{equation}%
for any $(x,y)\in TM_{0}$, under the strong convexity restriction $(1-\tilde{\eta})||\mathbf{G%
}^{T}||_{h}<\frac{1}{2}.$ In particular, if $\tilde{\eta}=0$, then the last
equation leads to the standard Matsumoto metric (\cite{matsumoto,S-Sabau}), which
stands for the solution to MAT, i.e. $P=A$ in \cref{fig_slope_general}. On
the other edge, for $\tilde{\eta}=1$ we get the Riemannian case, i.e. $%
P=O$ in \cref{fig_slope_general} because the impact of the gravitational
wind is compensated completely during such kind of motion on the slope.
To put it differently, the dead wind coincides with $\mathbf{G}^{T}$ whilst
the active wind is vanished. Therefore, the Finsler metric $\tilde{F}%
_{_{R-MAT}}$ is then simplified to $||y||_{h}$. In contrast to R-ZNP, there is
not any rigid translation of the indicatrix of $\tilde{F}_{_{R-MAT}}$ and
only anisotropic deformation is applied, while transiting between two
arbitrary problems included in R-MAT.  Moreover, analogously to R-ZNP, a change of the force of $\mathbf{G}^{T}$  in MAT (satisfying the strong convexity condition $||\mathbf{G}^{T}||_{h}<\frac{1}{2(1-\tilde{\eta})}$) can be compensated by
the respective change of along-traction coefficient $\tilde{\eta}$ in R-MAT under $\mathbf{G}^{T}$, and vice versa (under the strong convexity conditions) in order to preserve the same initial set-up.

\subsubsection{Reduced CROSS}

As the last special problem mentioned in this subsection we
consider the setting $\tilde{\eta}=1$ and $\eta \in \lbrack 0,1]$. By
analogy to the cases described above, this scenario is named the \textit{%
reduced cross slope problem} (R-CROSS for short) and indicated by the dashed red segment in \cref{fig_square_1}.
As the cross-traction coefficient $\eta $ runs through the full range,
RIEM and CROSS can be linked now by the transition $\mathcal{T}_{1,1}^{0,1}$, becoming the
particular and edge cases in the current set-up. It follows from
the above that the related equation of motion is\footnote{%
For brevity, we write $\mathbf{G}_{MAT}^{\perp}$ for $\text{Proj}_{u^{\perp }}\mathbf{G}^{T}$,  where $\mathbf{G}_{MAT}^{\perp}=-\mathbf{G}_{MAT}+\mathbf{G}^{T}$, i.e. $\overrightarrow{OA^{\prime }}$ in \cref{fig_slope_general} and recall that  $\mathbf{G}_{MAT}$ stands for $\text{Proj}_{u}\mathbf{G}^{T}$, i.e. $\overrightarrow{OA}$ in \cref{fig_slope_general}. 
}   
\begin{equation}
v_{_{R-CROSS}}=u+(1-\eta )\mathbf{G}_{MAT}^{\perp}, 
\label{vr-cross}
\end{equation}%
for any $\eta \in \lbrack 0,1]$. Hence, this yields $P\in {OA^{\prime }}$ in \cref{fig_slope_general} and the effective wind is zeroed
while the cross wind is varying, depending on cross-traction on the
slippery mountain slope. In particular, if $\eta =0$, i.e. $P=A^{\prime }$
in \cref{fig_square_1}, then the solution is given by a cross slope metric (%
\cite{cross}), which belongs to the class of general $(\alpha ,\beta )$%
-metrics in Finsler geometry. On the other end, if $\eta =1$, i.e. $P=O$,
then we are led to the Riemannian metric $h$.

Having solved two previous problems explicitly, one may expect that the similar ease of investigation will be in
the third analogous scenario. Unfortunately, the solution is much more
complicated now. In contrast to R-MAT and R-ZNP, this time we do not get a
``simple'' explicit form of the Finsler metric. 
As shown on further reading, it is signficantly nontrivial and could be
studied individually\footnote{The main proof concerning R-CROSS studied individually would
follow the analogous way as in \cite{cross}, however with the scaling factor 
$(1-\eta)$ for the vector field $\mathbf{G}_{MAT}^{\perp}$, included from the
beginning in the related equations of motion. Similarly, the corresponding
scaling factor in R-MAT is $(1-\tilde{\eta})$, however referring to the
vector field $\mathbf{G}_{MAT}$, as well as $(1-\eta)$ with reference to the
gravitational wind $\mathbf{G}^{T}$ in R-ZNP.}, however the corresponding
solution can be extracted as the particular case from the general result in
\cref{Sec_4}, having the strong convexity condition $||\mathbf{G}^{T}||_h<\frac{1}{2(1-\eta)}$ in this case. Actually, the computational difficulty could have been
expected in advance, since the detailed solution to CROSS has already been
analyzed in \cite{cross} and it stands for the edge case in the current
setting. We decided, however, to mention this new type of the slippery slope problem
here so that the problem diagram form a square clearly after including the last
missing side as shown in \cref{fig_square_1}.

Unlike R-ZNP and R-MAT, the evolution of the indicatrix is based on both
anisotropic deformation and rigid translation combined together, while
transiting between two arbitrary problems included in R-CROSS. Moreover, analogously to the preceding special scenarios, R-CROSS under $\mathbf{G}%
^{T} $ and with the cross-traction coefficient $\eta$ can coincide with CROSS in the presence of weaker gravitational wind, i.e. $(1-\eta)\mathbf{G}^{T}$,  and the other way around; see also \cite[Figure 2 (left)]{cross} in this
regard.

Loosely speaking, ZNP, MAT and CROSS under weaker gravitational winds
correspond respectively to the certain R-ZNP, R-MAT and R-CROSS under
stronger gravitational winds. To complete, all three cases described above
start (or end) at the Riemannian vertex $(1,1)$, including the transition
along the diagonal of the problem square diagram. Likewise, one can also consider
another transition along the second diagonal of $\mathcal{\tilde{S}}$, i.e. $\mathcal{T}_{1,0}^{0,1}$ linking MAT and CROSS. In this case, we have the
relation $\tilde{\eta}=1-\eta $, $\eta \in \lbrack 0,1]$, and so $P\in {AA^{\prime }}$ (\cref{fig_slope_general}), where one traction
coefficient is a non-identity\footnote{Unlike R-ZNP, where $\tilde{\eta}=\eta $.} (linear) function of another. As in R-CROSS, the corresponding resultant metric cannot be explicitly obtained in a simple
form.

\subsection{General case}

\label{Sec_3.2}

Following the above presented reasoning in a more general context we 
ask whether the Riemannan case can also be linked with a problem $%
\mathcal{P}$ defined by a pair $(\eta, \tilde{\eta})$, indicating the
interior of the square diagram $\mathcal{\tilde{S}}$, and not just its
boundary $\partial\tilde{\mathcal{S}}$ (i.e. R-CROSS, SLIPPERY, R-MAT, R-CROSS) or one diagonal, i.e. $\tilde{\eta}=\eta $ (R-ZNP) as untill now. Moreover, we can actually look at the model even more generally, connecting
two arbitrary problems $\mathcal{P}=(\eta, \tilde{\eta})$ and $\mathcal{P}%
^{\prime }=(\eta^{\prime }, \tilde{\eta}^{\prime })$ of the whole diagram,
where $\eta, \eta^{\prime }, \tilde{\eta}, \tilde{\eta}^{\prime }\in[0, 1]$.
Thus, we are going to enter the interior of $\mathcal{\tilde{S}}$ created by four
cornered cases, i.e. RIEM, MAT, ZNP and CROSS (respectively, $O$, $A$, $C$
and $A^{\prime }$ in \cref{fig_slope_general}), and ultimately to cover its
whole area. From such point of view each $\mathcal{P}$ with the traction
coefficients being fixed uniquely defines different and specific navigation
problem on the slope, in which the corresponding equations of motion depend
on both traction coefficients. Therefore, a pair $(\eta, \tilde{\eta})$ determines the
range of compensation of the gravity effects (transverse and longitudinal)
during motion on the slippery slope, in particular, the behavior of time-optimal trajectories. In consequence, we can state that there exist in fact infinitely many slippery slope navigation problems, where the classical Matsumoto's slope-of-a-mountain and Zermelo's navigation under gravitational wind stand for natural but also very particular cases now, among many others.

Furthermore, it will be possible to create the direct links between two
arbitrary problems via the general solution, i.e. the transitions $\mathcal{T%
}_{\mathcal{P}}^{\mathcal{P}^{\prime }}$, where the traction coefficients
are not fixed but varying as in SLIPPERY and S-CROSS thus far. In other
words, one can set up the ranges of the parameters, i.e. $\eta\in[\eta_1,
\eta^{\prime }_1]\subseteq[0, 1], \ \tilde{\eta}\in[\tilde{\eta}_1, \tilde{\eta%
}^{\prime }_1]\subseteq[0, 1]$ and fix the relation $\tilde{\eta}=f(\eta)$,
e.g. $\tilde{\eta}=(\eta-\eta_1)(\tilde{\eta}^{\prime }_1- \tilde{\eta}_1)/(\eta^{\prime
}_1-\eta_1)+ \tilde{\eta}_1$. This is visualized graphically by
a straight line segment (black) along which $\mathcal{P}$ is moving
smoothly, connecting two specific problems $\mathcal{P}_1=(\eta_1, \tilde{\eta}_1), \ \mathcal{P}^{\prime }_1=(\eta^{\prime }_1, \tilde{\eta}^{\prime
}_1)\in \mathcal{\tilde{S}}$ in \cref{fig_square2}. 
Taking a step further, we remark that $\mathcal{P}_1$ and $\mathcal{P}^{\prime }_1$ could also be joined  in the nonlinear ways.

\begin{figure}[h!]
\centering
~\includegraphics[width=0.485\textwidth]{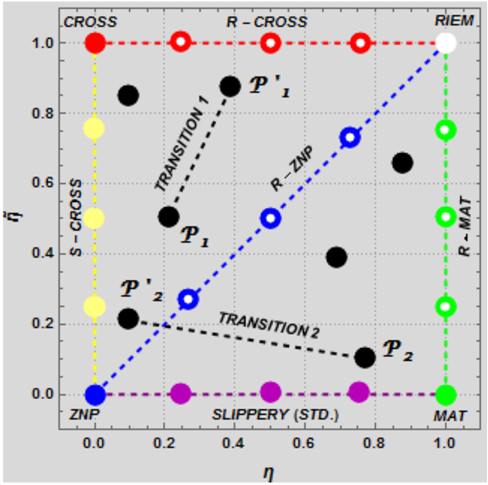}
\caption{The complete problem square diagram $\mathcal{\tilde{S}}%
=[0,1]\times \lbrack 0,1]$ including all navigation problems on the slippery slope under
gravity with fixed (i.e.  a specific point $\mathcal{P}$) and varying (i.e.
a transition $\mathcal{T}_{\mathcal{P}}^{\mathcal{P}^{\prime }}$) cross- and
along-traction coefficients, where $\mathcal{P}=(\protect\eta, \tilde{%
\protect\eta}), \mathcal{P}^{\prime }=(\protect\eta^{\prime }, \tilde{%
\protect\eta}^{\prime })\in\mathcal{\tilde{S}}$. }
\label{fig_square2}
\end{figure}

\begin{figure}[h]
\centering
~\includegraphics[width=0.70\textwidth]{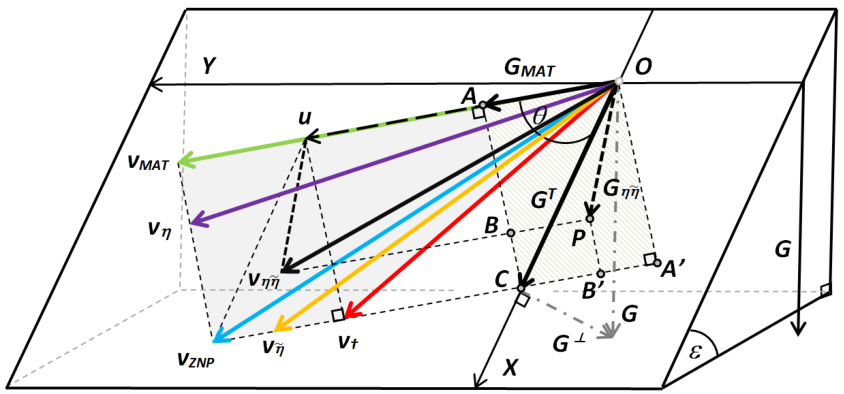}
~\includegraphics[width=0.30\textwidth]{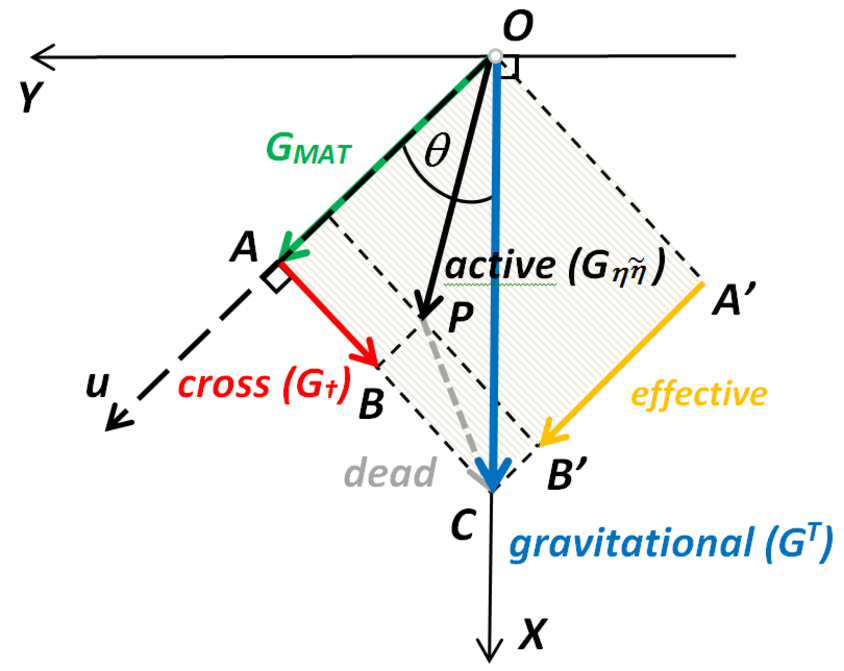}
\caption{Left: A model of a slippery $(\protect\eta, \tilde{\protect\eta})$%
-slope as an inclined plane $M$ of the slope angle $\protect\varepsilon$ in $%
\mathbb{R}^{3}$, under the gravity field $\mathbf{G}=\mathbf{G}^{T}+\mathbf{G%
}^{\perp}$, acting perpendicularly to the horizontal plane (the base of the
slope). The gravitational wind $\mathbf{G}^{T}$ ``blows'' tangentially to $M$
in the steepest downhill direction $X$, and $\mathbf{G}^{\perp}$ is the
component of gravity normal to the slope $M$; $OX\perp OY$, $O\in M$. The
resultant velocity is represented by $v_{\protect\eta\tilde{\protect\eta}}$
(black) and its particular (cornered) cases are: $v_{ZNP}$ (blue), $v_{MAT}$
(green), $v_{\dag}$ (CROSS, red) and $v_{RIEM}=u$ (dashed black). For
comparison, $v_{\protect\eta}$ (purple) and $v_{\tilde{\protect\eta}}$
(yellow) refer to the exemplary SLIPPERY and S-CROSS, respectively. Right: The
decompositions of a gravitational wind $\mathbf{G}^{T}=\protect%
\overrightarrow{OC}$ (blue) on the slippery slope $M$, where $OA\perp
OA^{\prime }$. $\mathbf{G}^{T}$ is a vector sum of an active wind $\mathbf{G}%
_{\protect\eta\tilde{\protect\eta}}=\protect\overrightarrow{OP}$ (black) and
a dead wind $\protect\overrightarrow{PC}$ (dashed grey). The lateral ($%
\mathbf{G}_{\dag}=\protect\overrightarrow{AB}$, a cross wind, red) and
longitudinal ($\protect\overrightarrow{A^{\prime }B^{\prime }}$, an
effective wind, yellow) components of the active wind $\mathbf{G}_{\protect%
\eta\tilde{\protect\eta}}=\protect\overrightarrow{OP}$ w.r.t. $u$ depend in
particular on the cross-traction ($\protect\eta$) and along-traction ($%
\tilde{\protect\eta}$) coefficients, respectively, where $\protect%
\overrightarrow{AB}\perp \protect\overrightarrow{A^{\prime }B^{\prime }}$.
The direction $\protect\theta$ of motion that is not influenced by gravity is 
indicated by the Riemannian self-velocity $u$ (the control vector, dashed
black) and measured clockwise from $OX$, where $||u||_h=1$. }
\label{fig_slope_general}
\end{figure}

Considering equations \eqref{vr-znp}, \eqref{vr-mat} and \eqref{vr-cross}, the general equation of motion is formulated as follows 
\begin{equation}
v_{\eta \tilde{\eta}}=u+\mathbf{G}_{\eta \tilde{\eta}},
\label{eqs_motion_general}
\end{equation}

\ 

\noindent where the active wind $\mathbf{G}_{\eta \tilde{\eta}}$ on the $(\eta ,\tilde{\eta%
})$-slope is defined by the following linear combination
\begin{equation}
\mathbf{G}_{\eta \tilde{\eta}}=(1-\eta )\mathbf{G}_{MAT}^{\perp }+(1-\tilde{%
\eta})\mathbf{G}_{MAT},\quad (\eta ,\tilde{\eta})\in \mathcal{\tilde{S}}
\label{bbbb}
\end{equation}%
which is equivalent to 
\begin{equation}
\mathbf{G}_{\eta \tilde{\eta}}=(\eta -\tilde{\eta})\mathbf{G}_{MAT}+(1-\eta )%
\mathbf{G}^{T}.  \label{wind_general}
\end{equation}%

\ 

\noindent The formula \eqref{eqs_motion_general} implies in particular the equations of motion from the preceding studies of time-optimal navigation on the mountain slopes under gravity\footnote{To be precise, regarding the wind, the equations of motion include the velocity induced on the walker or craft by the gravitational force on the slippery slope.}, namely, $v_{\eta }=u+\mathbf{G}_{\eta }=u+\eta 
\mathbf{G}_{MAT}+(1-\eta )\mathbf{G}^{T}$ as in \cite{slippery} (SLIPPERY), $%
v_{\dag }=v_{01}=u+\mathbf{G}_{\dag }=u-\mathbf{G}_{MAT}+\mathbf{G}^{T}$ as
in \cite{cross} (CROSS), $v_{\tilde{\eta}}=u+\mathbf{G}_{\tilde{\eta}}= 
u-\tilde{\eta}\mathbf{G}_{MAT}+\mathbf{G}^{T}$ as in \cite{slipperyx}
(S-CROSS), $v_{MAT}=v_{10}=u+\mathbf{G}_{MAT}$ as in \cite{matsumoto}, 
$v_{ZNP}=v_{00}=u+\mathbf{G}^{T}$ as in \cite{chern_shen, colleenshen} 
and obviously, $v_{RIEM}=v_{11}=u$. On the other hand, for instance, the
relation $v_{\eta \tilde{\eta}}=u+\mathbf{G}_{\eta \tilde{\eta}}$, where $%
\mathbf{G}_{\eta \tilde{\eta}}=(1-\eta )\mathbf{G}_{MAT}^{\perp }+\eta 
\mathbf{G}_{MAT}
=(2\eta -1)\mathbf{G}_{MAT}+(1-\eta )\mathbf{G}^{T}$ 
linking MAT and CROSS has not been covered so far. Now however, the corresponding time-minimizing solution will come as
the particular case of the general result presented in the sequel. 

Furthemore, substituting a dead wind, i.e. the part of a gravitational wind $\mathbf{G}^{T}$ that is reduced due to traction, being the measure of compensation of gravity effect on $(\eta, \tilde{\eta})$-slope, for the corresponding active wind in the equations of motion related to an arbitrary slippery slope problem, i.e. $v_{\eta\tilde{\eta}}=u-\mathbf{G}%
_{\eta\tilde{\eta} }+\mathbf{G}^{T}$ yields a symmetric scenario to the
original one with respect to the center $(\frac12, \frac12)$ of the problem
square diagram $\mathcal{\tilde{S}}$ (\cref{fig_square2}). Namely, fixing $%
(\eta_0, \tilde{\eta}_0)\in\mathcal{\tilde{S}}$, we have a conversion 
$\mathcal{P}_{\eta_0, \tilde{\eta}_0}\rightleftarrows \mathcal{P}^{\prime
}_{1-\eta_0, 1-\tilde{\eta}_0}$, 
e.g. ZNP $\rightleftarrows$ RIEM, MAT $\rightleftarrows$ CROSS, $\eta_0$%
-SLIPPERY $\rightleftarrows$ (1-$\eta_0$)-R-CROSS, $\tilde{\eta_0}$-S-CROSS $%
\rightleftarrows$ ($1-\tilde{\eta_0}$)-R-MAT. Thus, the active (dead) wind in 
$\mathcal{P}$ stands respectively for the dead (active) wind in $\mathcal{P}^{\prime }$. 
For clarity's sake, see also \cref{fig_slope_general}  in this regard, where the dead wind $\overrightarrow{PC}$ reads

\begin{equation}
\overrightarrow{PC}=\tilde{\eta}\mathbf{G}_{MAT}+\eta\mathbf{G}%
_{MAT}^{\perp}=(\tilde{\eta}-\eta)\mathbf{G}_{MAT}+\eta\mathbf{G}^{T},
\qquad (\eta, \tilde{\eta})\in \mathcal{\tilde{S}}.
\end{equation}

\ 

\noindent It may be worth mentioning that the dead wind in each of the four cornered cases of $\mathcal{\tilde{S}}$, i.e. MAT, RIEM, CROSS and ZNP is equal to $\mathbf{G}_{MAT}^{\perp}$, $\mathbf{G}^{T}$, $\mathbf{G}_{MAT}$ and $\overrightarrow{0}$, respectively.

\section{Proof of Theorem 1.1}
\label{Sec_4}

The goal of this section is to prove \cref{Theorem1}. Some preparations are necessary. Let us consider the $n$-dimensional
Riemannian manifold $(M,h)$, $n>1,$ which represents here a model for a
slippery slope of a mountain. Let $\mathbf{G}^{T}=-\bar{g}\omega ^{\sharp }=-%
\bar{g}h^{ji}\frac{\partial p}{\partial x^{j}}\frac{\partial }{\partial x^{i}%
}$ be the gravitational wind and let $u$ be the self-velocity of a moving craft
on the slope, assuming throughout this section that $||u||_{h}=1$. For now,
let us consider the active wind $\mathbf{G}_{\eta \tilde{\eta}}$ expressed
by \eqref{bbbb}, with $(\eta ,\tilde{\eta})\in \mathcal{\tilde{S}}$,
where $\mathcal{\tilde{S}}=[0,1]\times \lbrack 0,1]$, which vanishes only
when $\eta =\tilde{\eta}=1.$

We pose the navigation problem $\mathcal{P}_{\eta ,\tilde{\eta}}$
on $(M,h)$ under the action of the active wind $\mathbf{G}_{\eta \tilde{\eta}%
}$, where the resultant velocity is $v_{\eta \tilde{\eta}}=u+\mathbf{G}_{\eta 
\tilde{\eta}},$ for any $(\eta ,\tilde{\eta})\in \mathcal{\tilde{S}}.$
Apparently, it looks like a standard Zermelo navigation problem, where the solution is given by a Finsler metric of Randers type if the wind is weak (\cite {colleenshen,chern_shen,Y-Sabau,CJS}). In reality, our navigation problem is quite complicated due to the active wind $\mathbf{G}_{\eta \tilde{\eta}}$.  The key observation is that our ingredient $\mathbf{G}_{\eta \tilde{\eta}},$
written as $\mathbf{G}_{\eta \tilde{\eta}}=(\eta -\tilde{\eta})\mathbf{G}%
_{MAT}+(1-\eta )\mathbf{G}^{T},$ for any $(\eta ,\tilde{\eta})\in \mathcal{%
\tilde{S}}$, is not a priori known because only the gravitational wind $%
\mathbf{G}^{T}$ is given, and the vector $\mathbf{G}_{MAT}$ being the
orthogonal projection of $\mathbf{G}^{T}$ onto the self-velocity $u,$ depends
on the direction of $u.$ Moreover, using \eqref{bbbb}, it follows
immediately that $||\mathbf{G}_{\eta \tilde{\eta}}||_{h}\leq ||\mathbf{G}%
^{T}||_{h},$ for any $(\eta ,\tilde{\eta})\in \mathcal{S}$, where $\mathcal{%
S=\tilde{S}}\smallsetminus \{(1,1)\}.$ Thus, it is appropriate for us to
split the proof of \cref{Theorem1} into two steps including a sequence of
cases and lemmas, which enable us to describe the $(\eta ,\tilde{\eta})$%
-slope metrics besides the necessary and sufficient conditions for their
strong convexity, expressed exclusively with respect to the force of the
gravitational wind $\mathbf{G}^{T}$, for any $(\eta ,\tilde{\eta})\in 
\mathcal{S}$. The first step describes a direction-dependent deformation,
more precisely, the deformation of the background Riemannian metric $h$ by
the vector field $(\eta -\tilde{\eta})\mathbf{G}_{MAT}$. The second step
develops the classical Zermelo navigation, where the indicatrix of
the resulting Finsler metric $F$ of Matsumoto type, provided by the first
step, is rigidly translated by the gravitational vector field $(1-\eta )%
\mathbf{G}^{T}$, under the condition $F(x,-(1-\eta )\mathbf{G}^{T})<1$ which
practically secures that a craft on the slippery mountainside can go forward
in any direction (see \cite{SH,CJS}).

\paragraph{Step I.}

We state that under the assumption that $\left. |\eta -\tilde{\eta}|\text{ }
||\mathbf{G}_{MAT}||_{h}<1,\right. $ the direction-dependent deformation of
the Riemannian metric $h$ by $(\eta -\tilde{\eta})\mathbf{G}_{MAT}$ leads to
a Finsler metric if and only if $||\mathbf{G}^{T}||_{h}<\frac{1}{2|\eta -%
\tilde{\eta}|},$ for any $(\eta ,\tilde{\eta})\in \mathcal{\tilde{S}}%
\smallsetminus \mathcal{L}$, where $\mathcal{L}=\left\{ (\eta ,\tilde{\eta}%
)\in \mathcal{\tilde{S}}\text{ }|\text{ }\eta =\tilde{\eta}\right\}$. Moreover, when $\left. |\eta -\tilde{\eta}|\text{ }||\mathbf{G}%
_{MAT}||_{h}\geq 1\right. $ at some directions, this deformation cannot
afford a Finsler metric.

To this end, we describe the deformation of $h$ by the vector field $(\eta -%
\tilde{\eta})\mathbf{G}_{MAT}$ in terms of the resultant velocity $v=u+(\eta
-\tilde{\eta})\mathbf{G}_{MAT}$, for any $(\eta ,\tilde{\eta})\in \mathcal{%
\tilde{S}}\smallsetminus \mathcal{L}.$ Evidently, if $\eta =\tilde{\eta}$, 
then $v=u$. Furthermore, we need to study the following cases: 1. $\left. |\eta -%
\tilde{\eta}|\text{ }||\mathbf{G}_{MAT}||_{h}<1,\right. $ 2. $\left. |\eta -%
\tilde{\eta}|\text{ }||\mathbf{G}_{MAT}||_{h}=1\right. $ and 3. $\left.
|\eta -\tilde{\eta}|\text{ }||\mathbf{G}_{MAT}||_{h}>1,\right.$ separately.

Let us fix the notation. The desired direction  of self-motion is represented by the angle $\theta$ between $\mathbf{G}^{T}$ and $u$ (clockwise). Since $\mathbf{G}%
_{MAT}=$Proj$_{u}\mathbf{G}^{T}$, the vectors $v,$ $u$ and $\mathbf{G}_{MAT}$
are collinear and once we have denoted by $\bar{\theta}$ the angle between $%
u $ and $\mathbf{G}_{MAT},$ it results that it can only be $0$ or $\pi $. We
notice that when $\theta $ is $\frac{\pi }{2}$ or $\frac{3\pi }{2}$, the
angle $\bar{\theta}$ \ is not determined, i.e. $u$ and $\mathbf{G}^{T}$ are
orthogonal, and $\mathbf{G}_{MAT}$ vanishes.

\noindent

\paragraph{Case 1. $\left. |\protect\eta -\tilde{\protect\eta}|\text{ }||%
\mathbf{G}_{MAT}||_{h}<1\right. $}

Since we have assumed that $\left. |\eta -\tilde{\eta}|\text{ }||\mathbf{G}%
_{MAT}||_{h}<1,\right. $ it is certainly true that the angle between $%
\mathbf{G}^{T}$ and $v$ is also $\theta$ (the vectors $u$ and $v$ point in
the same direction). It is worthwhile to emphasize that due to the assertion 
$\left. |\eta -\tilde{\eta}|\text{ }||\mathbf{G}_{MAT}||_{h}<1\right. $
there is not any direction where the resultant vector $v$ vanishes. Now we
focus on $\bar{\theta},$ namely:

\medskip \noindent i) When $\bar{\theta}=0$ (going downhill), we have $%
\theta \in \lbrack 0,\frac{\pi }{2})\cup (\frac{3\pi }{2},2\pi )$ and the
angle between $\mathbf{G}^{T}$ and $\mathbf{G}_{MAT}$ is either $\theta $ or 
$2\pi -\theta $. Then, we clearly obtain that $||\mathbf{G}_{MAT}||_{h}=||%
\mathbf{G}^{T}||_{h}\cos \theta $ and 
\begin{equation*}
h(v,\mathbf{G}_{MAT})=||v||_{h}||\mathbf{G}_{MAT}||_{h}=||v||_{h}||\mathbf{G}%
^{T}||_{h}\cos \theta =h(v,\mathbf{G}^{T}).
\end{equation*}%
Also, for any $(\eta ,\tilde{\eta})\in \mathcal{\tilde{S}}\smallsetminus 
\mathcal{L}$ it follows that $\left. |\eta -\tilde{\eta}|\text{ }||\mathbf{G}%
^{T}||_{h}\cos \theta <1\right. $ and $\frac{h(v,\mathbf{G}^{T})}{||v||_{h}}<%
\frac{1}{|\eta -\tilde{\eta}|}.$

\medskip \noindent ii) When $\bar{\theta}=\pi $ (going uphill), it turns out
that $\theta \in (\frac{\pi }{2},\frac{3\pi }{2})$ and the angle between $%
\mathbf{G}^{T}$ and $\mathbf{G}_{MAT}$ is $|\pi -\theta |$. Thus, one easily
obtains that $||\mathbf{G}_{MAT}||_{h}=-||\mathbf{G}^{T}||_{h}\cos \theta $
and%
\begin{equation*}
h(v,\mathbf{G}_{MAT})=-||v||_{h}||\mathbf{G}_{MAT}||_{h}=||v||_{h}||\mathbf{G%
}^{T}||_{h}\cos \theta =h(v,\mathbf{G}^{T}).
\end{equation*}%
Moreover, for any $(\eta ,\tilde{\eta})\in \mathcal{\tilde{S}}%
\smallsetminus \mathcal{L}$ we have that $\left. -|\eta -\tilde{\eta}|\text{ 
}||\mathbf{G}^{T}||_{h}\cos \theta <1\right. $ and $-\frac{h(v,\mathbf{G}%
^{T})}{||v||_{h}}<\frac{1}{|\eta -\tilde{\eta}|}.$

To sum up, by both of the above sub-cases and noting that $v=u$ when $\theta
\in \{\frac{\pi }{2},\frac{3\pi }{2}\}$, we get 
\begin{equation}
h(v,\mathbf{G}_{MAT})=h(v,\mathbf{G}^{T})=||v||_{h}||\mathbf{G}%
^{T}||_{h}\cos \theta ,\text{ for any }\theta \in \lbrack 0,2\pi ).
\label{I.1}
\end{equation}%
In addition, we can write the inequality $\left. |\eta -\tilde{\eta}|\text{ }||\mathbf{G}%
_{MAT}||_{h}<1\right. $ as follows 
\begin{equation}
|\eta -\tilde{\eta}|\text{ }||\mathbf{G}^{T}||_{h}|\cos \theta |<1\,\ \ 
\text{\ \ or \ \ }\frac{|h(v,\mathbf{G}^{T})|}{||v||_{h}}<\frac{1}{|\eta -%
\tilde{\eta}|},\text{ }  \label{I.2}
\end{equation}%
for any $\theta \in \lbrack 0,2\pi )$ and $(\eta ,\tilde{\eta})\in \mathcal{%
\tilde{S}}\smallsetminus \mathcal{L}.$ Now, using \eqref{I.1}, we proceed by
straightforward computation starting with $1=||u||_{h}=||v-(\eta -\tilde{\eta%
})\mathbf{G}_{MAT}||_{h}.$ This leads to the equation 
\begin{equation*}
||v||_{h}^{2}-2(\eta -\tilde{\eta})||v||_{h}||\mathbf{G}^{T}||_{h}\cos
\theta -[1-(\eta -\tilde{\eta})^{2}||\mathbf{G}^{T}||_{h}^{2}\cos ^{2}\theta
]=0,
\end{equation*}%
which, due to the first inequality in \eqref{I.2}, admits the unique
positive root 
\begin{equation}
||v||_{h}=1+(\eta -\tilde{\eta})||\mathbf{G}^{T}||_{h}\cos \theta ,\text{ }
\label{I.3}
\end{equation}%
\bigskip for any $\theta \in \lbrack 0,2\pi )$ and $(\eta ,\tilde{\eta})\in 
\mathcal{\tilde{S}}\smallsetminus \mathcal{L}.$

If we introduce the notation $g_{1}(x,v)=||v||_{h}^{2}-||v||_{h}-(\eta -%
\tilde{\eta})h(v,\mathbf{G}^{T})$ and use \eqref{I.1}, then equation %
\eqref{I.3} can be written into its equivalent form $g_{1}(x,v)=0$. Thus,
based on Okubo's method (\cite{matsumoto}), we get the function%
\begin{equation*}
F(x,v)=\frac{||v||_{h}^{2}}{||v||_{h}+(\eta -\tilde{\eta})h(v,\mathbf{G}^{T})%
}
\end{equation*}%
as the solution of the equation $g_{1}(x,\frac{v}{F})=0.$ Moreover, we can
extend $F(x,v)$ to an arbitrary nonzero vector $y\in T_{x}M,$ for any $x\in
M$ because any nonzero $y$ can be expressed as $y=cv,$ $c>0,$ and $F(x,v)=1$%
. Namely, it turns out the following positive homogeneous $C^{\infty }$%
-function on $TM_{0}$ 
\begin{equation}
F(x,y)=\frac{||y||_{h}^{2}}{||y||_{h}+(\eta -\tilde{\eta})h(y,\mathbf{G}^{T})%
}\text{ },\text{ \ for any \ }(\eta ,\tilde{\eta})\in \mathcal{\tilde{S}}%
\smallsetminus \mathcal{L}.  \label{eta_mat}
\end{equation}

There is still a certain amount of properties which is arising regarding
function $F(x,y)$ obtained in \eqref{eta_mat}. Before all else, we claim
that our assertion $\left. |\eta -\tilde{\eta}|\text{ }||\mathbf{G}%
_{MAT}||_{h}<1\right. $ is a necessary and sufficient condition for $F(x,y)$
to be positive on all $TM_{0}.$ In order to prove this let us observe that
the positivity of \eqref{eta_mat} on $TM_{0}$ means that 
\begin{equation}
||y||_{h}+(\eta -\tilde{\eta})h(y,\mathbf{G}^{T})>0,  \label{ineq}
\end{equation}%
for all nonzero $y$ and any $(\eta ,\tilde{\eta})\in \mathcal{\tilde{S}}%
\smallsetminus \mathcal{L}.$ If the positivity is achieved on $TM_{0},$ we
can replace $y$ with $\pm \mathbf{G}^{T}\neq 0$ in \eqref{ineq} and thus,
it follows that $\left. |\eta -\tilde{\eta}|\text{ }||\mathbf{G}%
^{T}||_{h}<1.\right. $ Having the inequality $||\mathbf{G}_{MAT}||_{h}\leq ||%
\mathbf{G}^{T}||_{h}$ in any direction (note that $||\mathbf{G}%
_{MAT}||_{h}=||\mathbf{G}^{T}||_{h}|\cos \theta |,$ for any $\theta \in
\lbrack 0,2\pi )$), it yields that $\left. |\eta -\tilde{\eta}|\text{ }||%
\mathbf{G}_{MAT}||_{h}<1\right. $ on all $TM_{0}.$ Conversely, if $\left.
|\eta -\tilde{\eta}|\text{ }||\mathbf{G}_{MAT}||_{h}<1\right. $ on all $%
TM_{0}$ (the case $\mathbf{G}_{MAT}=0$ is also included), we have $%
\frac{|h(y,\mathbf{G}^{T})|}{||y||_{h}}<\frac{1}{|\eta -\tilde{\eta}|}$ for
any nonzero $y,$ which gives~\eqref{ineq}. Thus, the claim that $F(x,y)$ is
positive on $TM_{0}$ is proved.

From now on, we use the notation as in \cite{slippery,cross,slipperyx}, that is 
\begin{equation}
\alpha ^{2}=||y||_{h}^{2}=h_{ij}y^{i}y^{j}\text{ \ \ \ \ and \ \ }\beta =-%
\frac{1}{\bar{g}}h(y,\mathbf{G}^{T})=h(y,\omega ^{\sharp })=b_{i}y^{i},
\label{NOT}
\end{equation}%
$\alpha =\alpha (x,y),$ $\beta =\beta (x,y)$ and $||\beta ||_{h}=||\omega
^{\sharp }||_{h}.$ We notice that the differential $1$-form $\beta $ is
closed, i.e. $s_{ij}=0$, because it includes the gravitational wind $\mathbf{%
G}^{T}$ which is a scaled gradient vector field, i.e. $\mathbf{G}^{T}=-\bar{%
g}\omega ^{\sharp }=-\bar{g}h^{ji}\frac{\partial p}{\partial x^{j}}\frac{%
\partial }{\partial x^{i}}$; for more details, see \cite[Lemma 4.3]{slippery}.

With the notation \eqref{NOT}, we can express the function \eqref{eta_mat} as  
\begin{equation}
F(x,y)=\frac{\alpha ^{2}}{\alpha -(\eta -\tilde{\eta})\bar{g}\beta },\text{
for any \ }(\eta ,\tilde{\eta})\in \mathcal{\tilde{S}}\smallsetminus 
\mathcal{L},  \label{Matsumoto_eta}
\end{equation}%
which shows that it is of Matsumoto type having the explicit indicatrix%
\begin{equation*}
I_{F}=\left\{ (x,y)\in TM_{0}\text{ }|\text{ }\alpha ^{2}[\alpha -(\eta -%
\tilde{\eta})\bar{g}\beta ]^{-1}=1\right\} \subset TM.
\end{equation*}%
Since $y=0$ does not lie in the closure of the indicatrix $I_{F}$, we can
extend $F(x,y)$ continuously to all $TM,$ i.e. $F(x,0)=0$ for any $x\in M$
(see \cite{CJS}).

The function \eqref{Matsumoto_eta} seems to be a promising
Finsler metric. In order to make sure of this, we are going to establish the
necessary and sufficient conditions for the strong convexity of the
indicatrix $I_{F},$ for any $(\eta ,\tilde{\eta})\in \mathcal{\tilde{S}}%
\smallsetminus \mathcal{L}$. We can write $F(x,y)=\alpha \phi (s),$ where $%
\phi (s)=\frac{1}{1-(\eta -\tilde{\eta})\bar{g}s}\,$\ with $s=\frac{\beta }{%
\alpha }$, and the second inequality in \eqref{I.2} is actually $|s|<\frac{1%
}{|\eta -\tilde{\eta}|\bar{g}}$, for arbitrary nonzero $y\in T_{x}M$ and $%
x\in M$. Thus, for every $(\eta ,\tilde{\eta})\in \mathcal{\tilde{S}}%
\smallsetminus \mathcal{L}$ it follows that $\phi $ is a positive $C^{\infty
}$-function on the open interval $\mathcal{I}=\left( -(|\eta -\tilde{\eta}|%
\bar{g})^{-1},(|\eta -\tilde{\eta}|\bar{g})^{-1}\right) $. In the sequel, we
collect the desired properties for $\phi (s),$ with $|s|<\frac{1}{|\eta -%
\tilde{\eta}|\bar{g}}$, and we control the force of the gravitational wind $%
\mathbf{G}^{T}$ via the variable $s.$

\begin{lemma}
\label{Lema1} Let $\phi $ be the function given by $\phi (s)=\frac{1}{%
1-(\eta -\tilde{\eta})\bar{g}s}$ with $s\in \mathcal{I}$. For any $(\eta ,%
\tilde{\eta})\in \mathcal{\tilde{S}}\smallsetminus \mathcal{L}$, the
following statements are equivalent:

\begin{itemize}
\item[i)] $\phi (s)-s\phi ^{\prime }(s)+(b^{2}-s^{2})\phi ^{\prime \prime
}(s)>0$, where $b=||\omega ^{\sharp }||_{h};$

\item[ii)] $|s|\leq b<b_{0}$, where $b_{0}=\frac{1}{2|\eta -\tilde{\eta}|%
\bar{g}}$;

\item[iii)] $||\mathbf{G}^{T}||_{h}<\frac{1}{2|\eta -\tilde{\eta}|}.$
\end{itemize}
\end{lemma}

\begin{proof}
By using the Cauchy-Schwarz inequality $|h(y,\omega ^{\sharp })|\leq
||y||_{h}||\omega ^{\sharp }||_{h}$ it follows that $|s|\leq ||\omega
^{\sharp }||_{h}=b,$ for any nonzero $y\in T_{x}M$ and $x\in M.$ Let us now
focus on $(b^{2}-s^{2})\phi ^{\prime \prime }(s)$. Due to $|s|<\frac{1}{%
|\eta -\tilde{\eta}|\bar{g}}$, one has $(b^{2}-s^{2})\phi ^{\prime
\prime }(s)=(b^{2}-s^{2})\frac{2(\eta -\tilde{\eta})^{2}\bar{g}^{2}}{%
[1-(\eta -\tilde{\eta})\bar{g}s]^{3}}\geq 0.$ Thus, the minimum value of $%
(b^{2}-s^{2})\phi ^{\prime \prime }(s)$ is $0$ and it is achieved when $%
|s|=b,$ for any $(\eta ,\tilde{\eta})\in \mathcal{\tilde{S}}\smallsetminus 
\mathcal{L}$. Moreover, a simple computation leads to%
\begin{equation}
\phi (s)-s\phi ^{\prime }(s)+(b^{2}-s^{2})\phi ^{\prime \prime }(s)=\frac{%
[1-(\eta -\tilde{\eta})\bar{g}s][1-2(\eta -\tilde{\eta})\bar{g}%
s]+2(b^{2}-s^{2})(\eta -\tilde{\eta})^{2}\bar{g}^{2}}{[1-(\eta -\tilde{\eta})%
\bar{g}s]^{3}}.  \label{IEQ}
\end{equation}

To prove i) $\Rightarrow $ ii), we assume that $\phi (s)-s\phi ^{\prime
}(s)+(b^{2}-s^{2})\phi ^{\prime \prime }(s)>0.$ Let us take $s=\pm b$ in %
\eqref{IEQ}. It follows that $1\mp 2(\eta -\tilde{\eta})\bar{g}b>0$, and
then $b<\frac{1}{2|\eta -\tilde{\eta}|\bar{g}}.$ Therefore, $|s|\leq b<\frac{%
1}{2|\eta -\tilde{\eta}|\bar{g}}$ which is precisely the required ii).
Conversely, let us assume that $|s|\leq b<\frac{1}{2|\eta -\tilde{\eta}|\bar{%
g}}$. Due to \eqref{IEQ}, we get%
\begin{equation*}
\phi (s)-s\phi ^{\prime }(s)+(b^{2}-s^{2})\phi ^{\prime \prime }(s)\geq 
\frac{\lbrack 1-(\eta -\tilde{\eta})\bar{g}s][1-2(\eta -\tilde{\eta})\bar{g}%
s]}{[1-(\eta -\tilde{\eta})\bar{g}s]^{3}}=\frac{1-2(\eta -\tilde{\eta})\bar{g%
}s}{[1-(\eta -\tilde{\eta})\bar{g}s]^{2}}~>~0,
\end{equation*}%
for any $(\eta ,\tilde{\eta})\in \mathcal{\tilde{S}}\smallsetminus \mathcal{L%
}.$

Now we prove the implication iii) $\Rightarrow $ ii). Since\textbf{\ }$||%
\mathbf{G}^{T}||_{h}<\frac{1}{2|\eta -\tilde{\eta}|}$ and $|s|\leq ||\omega
^{\sharp }||_{h}=b=\frac{1}{\bar{g}}||\mathbf{G}^{T}||_{h}$, it turns out $%
|s|\leq b<\frac{1}{2|\eta -\tilde{\eta}|\bar{g}}.$ The implication ii) $%
\Rightarrow $ iii) is trivial. \end{proof}

It is worth mentioning that the statement $|s|\leq b<\frac{1}{2|\eta -\tilde{%
\eta}|\bar{g}}$ also implies that for any $(\eta ,\tilde{\eta})\in \mathcal{%
\tilde{S}}\smallsetminus \mathcal{L}$, $\phi (s)-s\phi ^{\prime }(s)>0$.
By the above findings and applying \cite[Lemma 1.1.2]{chern_shen} and %
\cref{Prop1}, we have stated the following result.

\begin{lemma}
\label{Lema2}For any $(\eta ,\tilde{\eta})\in \mathcal{\tilde{S}}%
\smallsetminus \mathcal{L}$, $F(x,y)=\frac{\alpha ^{2}}{\alpha -(\eta -%
\tilde{\eta})\bar{g}\beta }$ is a Finsler metric if and only if $||\mathbf{G}%
^{T}||_{h}<\frac{1}{2|\eta -\tilde{\eta}|}$.
\end{lemma}

\noindent Therefore, once we have \cref{Lema2}, we conclude that the
indicatrix $I_{F}$ is strongly convex if and only if $||\mathbf{G}^{T}||_{h}<%
\frac{1}{2|\eta -\tilde{\eta}|}$, for any $(\eta ,\tilde{\eta})\in \mathcal{%
\tilde{S}}\smallsetminus \mathcal{L}.$

\noindent

\paragraph{Case 2. $\left. |\protect\eta -\tilde{\protect\eta}|\text{ }||%
\mathbf{G}_{MAT}||_{h}=1\right. $}

We start by assuming that$\left. |\eta -\tilde{\eta}|\text{ }||\mathbf{G}%
_{MAT}||_{h}=1,\right. $ for any $(\eta ,\tilde{\eta})\in \mathcal{\tilde{S}}%
\smallsetminus \mathcal{L}.$ Observe that a traverse of the mountain, i.e.
when $\theta \in \{\frac{\pi }{2},\frac{3\pi }{2}\}$, cannot be followed
here. Indeed, when $\theta \in \{\frac{\pi }{2},\frac{3\pi }{2}\}$, $\mathbf{%
G}_{MAT}$ vanishes which contradicts our assumption. Moreover, since $||%
\mathbf{G}_{MAT}||_{h}\leq ||\mathbf{G}^{T}||_{h}$, it follows that $||%
\mathbf{G}^{T}||_{h}\geq \frac{1}{|\eta -\tilde{\eta}|}$, for any $(\eta ,%
\tilde{\eta})\in \mathcal{\tilde{S}}\smallsetminus \mathcal{L}.$ In the
sequel, we have to analyze the aforementioned possibilities for $\bar{\theta}
$:

\medskip \noindent i) when $\bar{\theta}=0,$ we clearly have $\theta \in
\lbrack 0,\frac{\pi }{2})\cup (\frac{3\pi }{2},2\pi )$ and $||\mathbf{G}%
_{MAT}||_{h}=||\mathbf{G}^{T}||_{h}\cos \theta .$ In addition, our
assumption implies that $u=|\eta -\tilde{\eta}|\mathbf{G}_{MAT}$, and thus 
\begin{equation*}
v=(\eta -\tilde{\eta}+|\eta -\tilde{\eta}|)\mathbf{G}_{MAT}=\left\{ 
\begin{array}{cc}
2(\eta -\tilde{\eta})\mathbf{G}_{MAT}, & \text{if }\eta >\tilde{\eta} \\ 
0, & \text{if }\eta <\tilde{\eta}%
\end{array}%
\right. .
\end{equation*}%
Based on this, we next get

(a) if $\eta >\tilde{\eta}$, then $||v||_{h}=2(\eta -\tilde{\eta})||\mathbf{G}%
_{MAT}||_{h}=2$ and%
\begin{equation*}
h(v,\mathbf{G}_{MAT})=||v||_{h}||\mathbf{G}_{MAT}||_{h}=||v||_{h}||\mathbf{G}%
^{T}||_{h}\cos \theta =h(v,\mathbf{G}^{T});
\end{equation*}

(b) if $\eta <\tilde{\eta},$ the resultant velocity $v$ vanishes, while
attempting to go down the slope.

\medskip \noindent ii) when $\bar{\theta}=\pi ,$ then $\theta \in (\frac{\pi 
}{2},\frac{3\pi }{2})$ and thus $u=-|\eta -\tilde{\eta}|\mathbf{G}_{MAT}$
and $||\mathbf{G}_{MAT}||_{h}=-||\mathbf{G}^{T}||_{h}\cos \theta $. It turns
out that 
\begin{equation*}
v=(\eta -\tilde{\eta}-|\eta -\tilde{\eta}|)\mathbf{G}_{MAT}=\left\{ 
\begin{array}{cc}
0, & \text{if }\eta >\tilde{\eta} \\ 
2(\eta -\tilde{\eta})\mathbf{G}_{MAT}, & \text{if }\eta <\tilde{\eta}%
\end{array}%
\right. .
\end{equation*}%
Also, this must be splitted into

(a) if $\eta >\tilde{\eta},$ the resultant velocity $v$ vanishes, while
attempting to go up the slope;

(b) if $\eta <\tilde{\eta},$ then $||v||_{h}=-2(\eta -\tilde{\eta})||\mathbf{%
G}_{MAT}||_{h}=2$ and%
\begin{equation*}
h(v,\mathbf{G}_{MAT})=-||v||_{h}||\mathbf{G}_{MAT}||_{h}=||v||_{h}||\mathbf{G%
}^{T}||_{h}\cos \theta =h(v,\mathbf{G}^{T}).
\end{equation*}

 Summing up the above findings, when $v$ does not vanish, we have 
$||v||_{h}=2$ and among the directions corresponding to $\theta$, only such
directions for which $\cos \theta =\frac{1}{(\eta -\tilde{\eta})\text{ }||%
\mathbf{G}^{T}||_{h}}$ (i.e. $\frac{h(v,\mathbf{G}^{T})}{||v||_{h}}=\frac{1}{\eta -\tilde{\eta}}$), for any $(\eta ,\tilde{\eta})\in \mathcal{\tilde{S}}%
\smallsetminus \mathcal{L}$, can be followed in this case. Let us consider $g_{2}(x,v)=0,$ where $g_{2}(x,v)=||v||_{h}-2.$ By Okubo's method \cite%
{matsumoto}, we get the function%
\begin{equation}
F(x,v)=\frac{1}{2}||v||_{h},  \label{KROP}
\end{equation}%
as the solution of the equation $g_{2}(x,\frac{v}{F})=0.$ The extension of $%
F(x,v)$ to an arbitrary nonzero vector $y\in \mathcal{A}_{x}=\mathcal{A\
\cap \ }T_{x}M,$ for any $x\in M,$ is $F(x,y)=\frac{1}{2}||y||_{h}$, where $%
\mathcal{A}=\{(x,y)~\in ~TM_{0}\ |\ ||y||_{h}-(\eta -\tilde{\eta})h(y,%
\mathbf{G}^{T})=0\}$. Since $||\mathbf{G}%
^{T}||_{h}\geq \frac{1}{|\eta -\tilde{\eta}|}$, it follows that $\mathbf{G}%
^{T}\in \mathcal{A}_{x}$ if and only if $\mathbf{G}^{T}=\mathbf{G}_{MAT}$
and $\eta >\tilde{\eta}$ (here the angle $\theta $ can only be $0)$ and $-%
\mathbf{G}^{T}\in \mathcal{A}_{x}$ if and only if $\mathbf{G}^{T}=\mathbf{G}%
_{MAT}$ and $\eta <\tilde{\eta}$ (here the angle $\theta $ can only be $\pi $%
). Therefore, this case cannot provide a proper Finsler metric defined on
the whole tangent bundle. 

\noindent

\paragraph{Case 3. $\left. |\protect\eta -\tilde{\protect\eta}|\text{ }||%
\mathbf{G}_{MAT}||_{h}>1\right. $}

The remaining case is $\left. |\eta -\tilde{\eta}|\text{ }||\mathbf{G}%
_{MAT}||_{h}>1\right. $. On one hand, it implies that $||\mathbf{G}%
^{T}||_{h}>\frac{1}{|\eta -\tilde{\eta}|},$ as well as the fact that $\theta 
$ cannot be $\frac{\pi }{2}$ or $\frac{3\pi }{2}$. Indeed, if $\theta \in \{%
\frac{\pi }{2},\frac{3\pi }{2}\}$, then $v=u$ and thus, $\mathbf{G}_{MAT}$ vanishes,
which contradicts our assumption. On the other hand, it follows that for any 
$(\eta ,\tilde{\eta})\in \mathcal{\tilde{S}}\smallsetminus \mathcal{L}$, the
resultant velocity $v$ and $(\eta -\tilde{\eta})\mathbf{G}_{MAT}$ point in
the same direction (downhill when $\eta >\tilde{\eta}$ and uphill when $%
\eta <\tilde{\eta}$) and $h(v,\mathbf{G}_{MAT})=\frac{|\eta -\tilde{\eta}|}{%
\eta -\tilde{\eta}}||v||_{h}||\mathbf{G}_{MAT}||_{h}$. Moreover, $\left.
|\eta -\tilde{\eta}|\text{ }||\mathbf{G}_{MAT}||_{h}>1\right. $ states that
there is not any direction where the resultant vector $v$ vanishes. Again,
we have to take into consideration both possibilities for $\bar{\theta}.$
Namely,

\medskip \noindent i) when $\bar{\theta}=0$, then $\theta \in \lbrack 0,%
\frac{\pi }{2})\cup (\frac{3\pi }{2},2\pi ).$ In particular, we have $||%
\mathbf{G}_{MAT}||_{h}=||\mathbf{G}^{T}||_{h}\cos \theta $ and due to the
required assumption, it follows that $|\eta -\tilde{\eta}|$ $||\mathbf{G}%
^{T}||_{h}\cos \theta >1.$ Two possibilities must still be analyzed:

(a) if $\eta >\tilde{\eta},$ then $\measuredangle (\mathbf{G}^{T},v)\in
\{\theta ,2\pi -\theta \}.$ Thus, we get%
\begin{equation*}
h(v,\mathbf{G}_{MAT})=||v||_{h}||\mathbf{G}_{MAT}||_{h}=||v||_{h}||\mathbf{G}%
^{T}||_{h}\cos \theta =h(v,\mathbf{G}^{T})
\end{equation*}%
and $\frac{h(v,\mathbf{G}^{T})}{||v||_{h}}>\frac{1}{\eta -\tilde{\eta}};$

(b) if $\eta <\tilde{\eta},$ then $\measuredangle (\mathbf{G}^{T},v)\in
\{\pi +\theta ,\theta -\pi \}$ as well as%
\begin{equation*}
h(v,\mathbf{G}_{MAT})=-||v||_{h}||\mathbf{G}_{MAT}||_{h}=-||v||_{h}||\mathbf{%
G}^{T}||_{h}\cos \theta =h(v,\mathbf{G}^{T})
\end{equation*}%
and $-\frac{h(v,\mathbf{G}^{T})}{||v||_{h}}>\frac{1}{\tilde{\eta}-\eta }.$

\medskip \noindent ii) when $\bar{\theta}=\pi ,$ one gets $\theta \in (%
\frac{\pi }{2},\frac{3\pi }{2}).$ In particular, we have $||\mathbf{G}%
_{MAT}||_{h}=-||\mathbf{G}^{T}||_{h}\cos \theta $ and $|\eta -\tilde{\eta}|$ 
$||\mathbf{G}^{T}||_{h}\cos \theta <-1$ since $|\eta -\tilde{\eta}|$ $||%
\mathbf{G}_{MAT}||_{h}>1.$ Moreover,

(a) if $\eta >\tilde{\eta},$ then $\measuredangle (\mathbf{G}^{T},v)=|\theta
-\pi |.$ Thus,%
\begin{equation*}
h(v,\mathbf{G}_{MAT})=||v||_{h}||\mathbf{G}_{MAT}||_{h}=-||v||_{h}||\mathbf{G%
}^{T}||_{h}\cos \theta =h(v,\mathbf{G}^{T})
\end{equation*}%
and $\frac{h(v,\mathbf{G}^{T})}{||v||_{h}}>\frac{1}{\eta -\tilde{\eta}}$;

(b) if $\eta <\tilde{\eta},$ then the angle between $\mathbf{G}^{T}$ and $v$
is also $\theta $. In consequence, it turns out that%
\begin{equation*}
h(v,\mathbf{G}_{MAT})=-||v||_{h}||\mathbf{G}_{MAT}||_{h}=||v||_{h}||\mathbf{G%
}^{T}||_{h}\cos \theta =h(v,\mathbf{G}^{T})
\end{equation*}%
and $-\frac{h(v,\mathbf{G}^{T})}{||v||_{h}}>\frac{1}{\tilde{\eta}-\eta }.$

By combining the above possibilities and since $\mathbf{G}_{MAT}$ cannot be vanished,
we get 
\begin{equation}
h(v,\mathbf{G}_{MAT})=h(v,\mathbf{G}^{T})=\frac{|\eta -\tilde{\eta}|}{\eta -%
\tilde{\eta}}||v||_{h}||\mathbf{G}^{T}||_{h}|\cos \theta |,\ \text{for any}\
\theta \in \lbrack 0,2\pi )\smallsetminus \{\pi /2,3\pi /2\},  \label{II.1}
\end{equation}%
and the condition $\left. |\eta -\tilde{\eta}|\text{ }||\mathbf{G}%
_{MAT}||_{h}>1\right. $ is equivalent to 
\begin{equation}
|\cos \theta |>\frac{1}{|\eta -\tilde{\eta}|\text{ }||\mathbf{G}^{T}||_{h}}%
\,\ \ \text{\ \ or \ \ }\left\{ 
\begin{array}{cc}
\frac{h(v,\mathbf{G}^{T})}{||v||_{h}}>\frac{1}{\eta -\tilde{\eta}}, & \text{%
if }\eta >\tilde{\eta} \\ 
-\frac{h(v,\mathbf{G}^{T})}{||v||_{h}}>\frac{1}{\tilde{\eta}-\eta }, & \text{%
if }\eta <\tilde{\eta}%
\end{array},
\right. \ \ \text{\ \ or \ \ } (\eta -\tilde{\eta})\frac{h(v,\mathbf{G}^{T})%
}{||v||_{h}}>1 ,\text{ }  \label{II.2}
\end{equation}%
for any $(\eta ,\tilde{\eta})\in \mathcal{\tilde{S}}\smallsetminus \mathcal{L%
}.$ Therefore, among the directions corresponding to $\theta \in \lbrack
0,2\pi )\smallsetminus \{\pi /2,3\pi /2\}$ only such directions for which $%
|\cos \theta |>\frac{1}{|\eta -\tilde{\eta}|\text{ }||\mathbf{G}^{T}||_{h}}$
can be followed in this case. By using \eqref{II.1}, a simple computation,
starting with $1=$ $||u||_{h}=||v-(\eta -\tilde{\eta})\mathbf{G}%
_{MAT}||_{h}, $ leads to the equation 
\begin{equation*}
||v||_{h}^{2}-2|\eta -\tilde{\eta}|\text{ }||v||_{h}||\mathbf{G}%
^{T}||_{h}\cos \theta -[1-(\eta -\tilde{\eta})^{2}||\mathbf{G}%
^{T}||_{h}^{2}\cos ^{2}\theta ]=0.
\end{equation*}%
Since $|\cos \theta |>\frac{1}{|\eta -\tilde{\eta}|\text{ }||\mathbf{G}%
^{T}||_{h}},$ for any$\ (\eta ,\tilde{\eta})\in \mathcal{\tilde{S}}%
\smallsetminus \mathcal{L}$ and $\theta \in \lbrack 0,2\pi )\smallsetminus
\{\pi /2,3\pi /2\},$ it follows that the last equation admits two positive
roots%
\begin{equation}
||v||_{h}=\pm 1+|\eta -\tilde{\eta}|\text{ }||\mathbf{G}^{T}||_{h}|\cos
\theta |.  \label{II.3}
\end{equation}%
Based on the property \eqref{II.1}, we can write \eqref{II.3} as $%
g_{3}(x,v)=0,$ where%
\begin{equation*}
g_{3}(x,v)=||v||_{h}^{2}\mp ||v||_{h}-(\eta -\tilde{\eta})h(v,\mathbf{G}%
^{T}).
\end{equation*}%
If we apply Okubo's method again (\cite{matsumoto}), we get the following
positive functions $F_{1,2}(x,v)=\frac{||v||_{h}^{2}}{\pm ||v||_{h}+(\eta -%
\tilde{\eta})h(v,\mathbf{G}^{T})}$ as the solutions of the equation $g_{3}(x,%
\frac{v}{F})=0.$ Next, we extend $F_{1,2}(x,v)$ to an arbitrary nonzero
vector $y\in \mathcal{A}_{x}^{\ast }=\mathcal{A}^{\ast }\mathcal{\ \cap \ }%
T_{x}M,$ for any $x\in M,$ where 
\begin{equation}
\mathcal{A}^{\ast }=\{(x,y)\in TM\text{ }|\text{ }\ ||y||_{h}-(\eta -\tilde{%
\eta})h(y,\mathbf{G}^{T})<0\}
\end{equation}%
is an open conic subset of $TM_{0},$ for any $(\eta ,\tilde{\eta})\in 
\mathcal{\tilde{S}}\smallsetminus \mathcal{L}.$ Thus, we obtain the positive
homogeneous $C^{\infty }$-functions%
\begin{equation}
F_{1,2}(x,y)=\frac{||y||_{h}^{2}}{\pm ||y||_{h}+(\eta -\tilde{\eta})h(y,%
\mathbf{G}^{T})}  \label{F12}
\end{equation}%
on $\mathcal{A}^{\ast }$, with $F_{1,2}(x,v)=1$. We notice that $\mathbf{G}%
^{T}\in \mathcal{A}_{x}^{\ast }$ iff $\eta >\tilde{\eta}$ and $-\mathbf{G}%
^{T}\in \mathcal{A}_{x}^{\ast }$ iff $\eta <\tilde{\eta}.$ Moreover, due to %
\eqref{II.2}, our initial assumption $|\eta -\tilde{\eta}|\text{ }||\mathbf{G}%
_{MAT}||_{h}>1$ is a necessary and sufficient condition for $F_{1,2}(x,y)$
to be positive on $\mathcal{A}^{\ast }.$

By the notation \eqref{NOT}, the functions $F_{1,2}(x,y)$ are of Matsumoto
type, namely 
\begin{equation}
F_{1,2}(x,y)=\frac{\alpha ^{2}}{\pm \alpha -(\eta -\tilde{\eta})\bar{g}\beta 
}  
\label{Fstrong}
\end{equation}%
on the conic domain $\mathcal{A}^{\ast }$, rewritten as $\mathcal{A}^{\ast
}=\{(x,y)\in TM$ $|$ $\alpha +(\eta -\tilde{\eta})\bar{g}\beta <0\}$.
However, $F_{1,2}$ can be at most conic Finsler metrics. By applying \cite[%
Corollary 4.15]{JS}, it turns out that both $F_{1,2}$ are strongly convex on 
$\mathcal{A}^{\ast }$ and thus, they are conic Finsler metrics on $\mathcal{A%
}^{\ast },$ for any $(\eta ,\tilde{\eta})\in \mathcal{\tilde{S}}%
\smallsetminus \mathcal{L}.$ Indeed, for $F_{1,2}$ the strong convexity
conditions $[\alpha \mp 2(\eta -\tilde{\eta})\bar{g}\beta ][\alpha \mp (\eta
-\tilde{\eta})\bar{g}\beta ]>0$ are satisfied for any $(x,y)\in \mathcal{A}%
^{\ast }$ and $(\eta ,\tilde{\eta})\in \mathcal{\tilde{S}}\smallsetminus 
\mathcal{L}.$

Summarizing the results from this step, beyond their intrinsic interest, it
turns out that the direction-dependent deformation of the background
Riemannian metric $h$ by the vector field $(\eta -\tilde{\eta})\mathbf{G}%
_{MAT}$, with $\left. |\eta -\tilde{\eta}|\text{ }||\mathbf{G}%
_{MAT}||_{h}<1\right.$ performed for each direction (note that the
converse inequality $\left. |\eta -\tilde{\eta}|\text{ }||\mathbf{G}%
_{MAT}||_{h}\geq 1\right. $ one may carry out only at some directions), for
any $(\eta ,\tilde{\eta})\in \mathcal{\tilde{S}}\smallsetminus \mathcal{L}$
provides the Finsler metric $F(x,y)=\frac{\alpha ^{2}}{\alpha -(\eta -\tilde{%
\eta})\bar{g}\beta }$ if and only if $||\mathbf{G}^{T}||_{h}<\frac{1}{2|\eta
-\tilde{\eta}|}$.

\paragraph{Step II.}

In attempting to use \cref{Prop3}, let us consider the navigation data $%
(F,(1-\eta )\mathbf{G}^{T})$ on the Finsler manifold $(M,F),$ where $F$ is
either the Finsler metric \eqref{Matsumoto_eta} if $(\eta ,\tilde{\eta})\in 
\mathcal{\tilde{S}}\smallsetminus \mathcal{L}$ or the background Riemannian
metric $h$ if $(\eta ,\tilde{\eta})\in \mathcal{L},$ assuming that 
\begin{equation}
F(x,-(1-\eta )\mathbf{G}^{T})<1.  \label{CC}
\end{equation}%
Exploring the Zermelo navigation on $(M,F)$ with the aforementioned
navigation data, we supply new Finsler
metrics, which we call the $(\eta ,\tilde{\eta})$-slope metrics, together with
the necessary and sufficient conditions for the strong convexity of their
indicatrices. More precisely, applying \cref{Prop3}, for each $(\eta ,\tilde{%
\eta})\in \mathcal{\tilde{S}}$, the $(\eta ,\tilde{\eta})$-slope metric has
to arise as the unique positive solution $\tilde{F}$ of the equation%
\begin{equation}
F(x,y-(1-\eta )\tilde{F}(x,y)\mathbf{G}^{T})=\tilde{F}(x,y),\text{ }
\label{II}
\end{equation}%
for any $(x,y)\in TM_{0}.$

Before doing this, a few details must be outlined. On one hand, the meaning of
our second step is that the addition of the scaled gravitational wind $%
(1-\eta )\mathbf{G}^{T}$ generates a rigid translation to the strongly convex
indicatrix provided by $v=u-(\eta -\tilde{\eta})\mathbf{G}_{MAT}$ in the
first step, for any $(\eta ,\tilde{\eta})\in \mathcal{\tilde{S}}%
\smallsetminus \mathcal{L}$ (see \cite{CJS}). We have already got rid of
the possibility that $|\eta -\tilde{\eta}|\ ||\mathbf{G}%
_{MAT}||_{h}\geq 1$ because if the first term $(\eta -\tilde{\eta})\mathbf{G}_{MAT}$
generates a non-strongly convex indicatrix, the second term $(1-\eta )\mathbf{G}^{T}$ (which is just a translation) cannot compensate that. On the other hand, the
condition \eqref{CC} plays an essential role for the resulting indicatrix,
obtained by the translation, to be strongly convex and to determine a new
Finsler metric as the unique solution of \eqref{II} (see \cite[p. 10
and Proposition 2.14]{CJS}). In other words, the condition \eqref{CC}
secures that for any $x\in M$, $y=0$ belongs to the region bounded by the
new (translated) indicatrix $I_{\tilde{F}}$. 

In the sequel, we expand the left-hand side of \eqref{II}. Let us
observe that the Finsler metric $F$\ can be written as $F(x,y)=\frac{\alpha
^{2}}{\alpha -(\eta -\tilde{\eta})\bar{g}\beta },$\ for any $(\eta
,\tilde{\eta})\in \mathcal{\tilde{S}}.$ In particular, if $\eta =1,$ it is
obvious that $\tilde{F}(x,y)=\frac{\alpha ^{2}}{\alpha -(1-\tilde{\eta})\bar{%
g}\beta},$ for any\textbf{\ }$\tilde{\eta}\in \lbrack 0,1]$, i.e. the metric in the reduced Matsumoto case (R-MAT). For arbitrary $\eta\in[0, 1]$, by taking $y-(1-\eta )\tilde{F}(x,y)%
\mathbf{G}^{T}$ instead of $y$ in \eqref{NOT}, some standard computations
yield 
\begin{equation*}
\alpha ^{2}\left( x,y-(1-\eta )\tilde{F}(x,y)\mathbf{G}^{T}\right) =\alpha
^{2}(x,y)+2(1-\eta )\bar{g}\beta (x,y)\tilde{F}(x,y)+(1-\eta )^{2}||\mathbf{G%
}^{T}||_{h}^{2}\tilde{F}^{2}(x,y)
\end{equation*}%
and%
\begin{equation*}
\beta \left( x,y-(1-\eta )\tilde{F}(x,y)\mathbf{G}^{T}\right) =\beta (x,y)+%
\frac{1-\eta}{\bar{g}}||\mathbf{G}^{T}||_{h}^{2}\tilde{F}(x,y),
\end{equation*}%
where we used the relation $\beta (x,\mathbf{G}^{T})=-\frac{1}{\bar{g}}||\mathbf{G}%
^{T}||_{h}^{2}$. Therefore, it turns out that the left-hand side of 
\eqref{II} is 
\begin{equation*}
\frac{\alpha ^{2}+2(1-\eta )\bar{g}\beta \tilde{F}+(1-\eta )^{2}||\mathbf{G}%
^{T}||_{h}^{2}\tilde{F}^{2}}{\sqrt{\alpha ^{2}+2(1-\eta )\bar{g}\beta \tilde{%
F}+(1-\eta )^{2}||\mathbf{G}^{T}||_{h}^{2}\tilde{F}^{2}}-(\eta -\tilde{\eta})%
\bar{g}\beta -(\eta -\tilde{\eta})(1-\eta )||\mathbf{G}^{T}||_{h}^{2}\tilde{F%
}},
\end{equation*}%
where $\alpha $, $\beta $ and $\tilde{F}$ are evaluated at $(x,y)$. Now, if
we substitute this into \eqref{II}, we get the irrational equation%
\begin{equation}
\tilde{F}\sqrt{\alpha ^{2}+2(1-\eta )\bar{g}\beta \tilde{F}+(1-\eta )^{2}||%
\mathbf{G}^{T}||_{h}^{2}\tilde{F}^{2}}=\alpha ^{2}+(2-\eta -\tilde{\eta})%
\bar{g}\beta \tilde{F}+(1-\eta )(1-\tilde{\eta})||\mathbf{G}^{T}||_{h}^{2}%
\tilde{F}^{2},  \label{MAMA_general}
\end{equation}%
which is equivalent to the following polynomial equation 
\begin{equation}
\begin{array}{c}
(1-\eta )^{2}||\mathbf{G}^{T}||_{h}^{2}[1-\left( 1-\tilde{\eta}\right) ^{2}||%
\mathbf{G}^{T}||_{h}^{2}]\tilde{F}^{4}+2(1-\eta )\left[ 1-\left( 2-\eta -%
\tilde{\eta}\right) \left( 1-\tilde{\eta}\right) ||\mathbf{G}^{T}||_{h}^{2}%
\right] \bar{g}\beta \tilde{F}^{3} \\ 
~ \\ 
+\{\left[ 1-2(1-\eta )\left( 1-\tilde{\eta}\right) ||\mathbf{G}^{T}||_{h}^{2}%
\right] \alpha ^{2}-\left( 2-\eta -\tilde{\eta}\right) ^{2}\bar{g}^{2}\beta
^{2}\}\tilde{F}^{2}-2\left( 2-\eta -\tilde{\eta}\right) \bar{g}\alpha
^{2}\beta \tilde{F}-\alpha ^{4}=0,%
\end{array}
\label{MAMA_4}
\end{equation}%
for any\textbf{\ }$(\eta ,\tilde{\eta})\in \mathcal{\tilde{S}}$.

In the special case where $\eta =\tilde{\eta}=1,$ we obviously
have $\tilde{F}=h.$ If $(1-\eta )^{2}[1-\left( 1-\tilde{\eta}\right) ^{2}||%
\mathbf{G}^{T}||_{h}^{2}]\neq 0$, the last equation admits four roots, and
thanks to the condition \eqref{CC}, we know precisely that among all roots
there is only one positive. For any $(\eta ,\tilde{\eta})\in \mathcal{%
\tilde{S}}$, it should be the $(\eta ,\tilde{\eta})$-slope metric. From now
on, we denote by $\tilde{F}_{\eta \tilde{\eta}}$ the $(\eta ,\tilde{\eta})$%
-slope metric, outlining that $\tilde{F}_{\eta \tilde{\eta}}$ satisfies \eqref{MAMA_general} and moreover, along any regular piecewise $C^{\infty }$%
-curve $\gamma ,$ parametrized by time that represents a trajectory in
Zermelo's problem, we have $\tilde{F}_{\eta \tilde{\eta}}(\gamma (t),\dot{%
\gamma}(t))=1.$ This is the time in which a walker, a craft or a vehicle goes along $%
\gamma $.

Now, it remains to provide explicitly the necessary and sufficient
conditions for the indicatrix of $\tilde{F}_{\eta \tilde{\eta}}$ to be
strongly convex, and thus we will outline the argument that the $\tilde{F}%
_{\eta \tilde{\eta}}$-geodesics locally minimize time. In order to handle
this issue, we need to characterize the inequality \eqref{CC} which is
equivalent to%
\begin{equation}
\frac{1-(1-\tilde{\eta})||\mathbf{G}^{T}||_{h}}{1-(\eta -\tilde{\eta})||%
\mathbf{G}^{T}||_{h}}>0.  \label{CC_SC}
\end{equation}%
Indeed, since $F(x,y)=\frac{\alpha ^{2}}{\alpha -(\eta -\tilde{\eta})\bar{g}%
\beta }=\frac{||y||_{h}^{2}}{||y||_{h}+(\eta -\tilde{\eta})h(y,\mathbf{G}%
^{T})},$\ it turns out that, for any\textbf{\ }$(\eta ,\tilde{\eta})\in 
\mathcal{\tilde{S}},$ the left-hand side of \eqref{CC} is%
\begin{equation*}
F(x,-(1-\eta )\mathbf{G}^{T})=\frac{||-(1-\eta )\mathbf{G}^{T}||_{h}^{2}}{%
||-(1-\eta )\mathbf{G}^{T}||_{h}+(\eta -\tilde{\eta})h(-(1-\eta )\mathbf{G}%
^{T},\mathbf{G}^{T})}=\frac{(1-\eta )||\mathbf{G}^{T}||_{h}}{1-(\eta -\tilde{%
\eta})||\mathbf{G}^{T}||_{h}}
\end{equation*}%
which together with \eqref{CC} conclude the claim \eqref{CC_SC}. Namely, we
prove the result

\begin{lemma}
\label{Lema3} The following statements are equivalent:

\begin{itemize}
\item[i)] for any $(\eta ,\tilde{\eta})\in \mathcal{\tilde{S}}$, the
indicatrix $I_{\tilde{F}_{\eta \tilde{\eta}}}$ of the $(\eta ,\tilde{\eta})$%
-slope metric $\tilde{F}_{\eta \tilde{\eta}}$ is strongly convex;

\item[ii)] the gravitational wind $\mathbf{G}^{T}$ is restricted to either $%
||\mathbf{G}^{T}||_{h}<\frac{1}{1-\tilde{\eta}}$ and $(\eta ,\tilde{\eta}%
)\in \mathcal{D}_{1}\cup \mathcal{D}_{2}$, or $||\mathbf{G}^{T}||_{h}<\frac{1%
}{2|\eta -\tilde{\eta}|}$ and $(\eta ,\tilde{\eta})\in \mathcal{D}_{3}\cup 
\mathcal{D}_{4}$, where%
\begin{equation}
\begin{array}{l}
\mathcal{D}_{1}=\left\{ (\eta ,\tilde{\eta})\in \mathcal{S}\text{ }|\text{ }%
\eta \geq \tilde{\eta}>2\eta -1\right\} ,\qquad \ \ \mathcal{D}_{2}=\left\{
(\eta ,\tilde{\eta})\in \mathcal{S}\text{ }|\text{ }\frac{3\tilde{\eta}-1}{2}%
<\eta <\tilde{\eta}\right\} , \\ 
\\ 
\mathcal{D}_{3}=\left\{ (\eta ,\tilde{\eta})\in \mathcal{S}\text{ }|\text{ }%
\eta \geq \frac{1}{2},\text{ }\tilde{\eta}\leq 2\eta -1\right\} ,\quad \ 
\mathcal{D}_{4}=\left\{ (\eta ,\tilde{\eta})\in \mathcal{S}\text{ }|\text{ }%
\tilde{\eta}\geq \frac{1}{3},\text{ }\eta \leq \frac{3\tilde{\eta}-1}{2}%
\right\} ,%
\end{array}
\label{DDD}
\end{equation}

$\mathcal{S}=\bigcup\limits_{i=1}^{4}\mathcal{D}_{i}$ and $\mathcal{D}%
_{i}\cap \mathcal{D}_{j}=\varnothing ,$ for any $i\neq j,$ $i,j=1,...,4.$ No
restriction should  be imposed on $||\mathbf{G}^{T}||_{h}$ if $\eta =\tilde{%
\eta}=1.$

\item[iii)] the active wind $\mathbf{G}_{\eta \tilde{\eta}}$ given by %
\eqref{wind_general} is restricted to either $||\mathbf{G}_{\eta \tilde{\eta}%
}||_{h}<\frac{1}{1-\tilde{\eta}}$ and $(\eta ,\tilde{\eta})\in \mathcal{D}%
_{1}\cup \mathcal{D}_{2}$, or $||\mathbf{G}_{\eta \tilde{\eta}}||_{h}<\frac{1%
}{2|\eta -\tilde{\eta}|}$ and $(\eta ,\tilde{\eta})\in \mathcal{D}_{3}\cup 
\mathcal{D}_{4}.$
\end{itemize}
\end{lemma}

\begin{proof}
To prove the equivalence i) $\Leftrightarrow $ ii) one has to take into
account \eqref{CC_SC}, for any $(\eta ,\tilde{\eta})\in \mathcal{\tilde{S}}$%
. \ Because of this, the following cases must be analyzed separately:
\medskip

\noindent a) if $\eta >\tilde{\eta}$ and $\eta \neq 1$, then  inequality %
\eqref{CC_SC} yields either $||\mathbf{G}^{T}||_{h}<\frac{1}{1-\tilde{\eta}}$
or $||\mathbf{G}^{T}||_{h}>\frac{1}{\eta -\tilde{\eta}}.$ If we combine
these with the strong convexity condition for the indicatrix $I_{F},$ i.e. $%
||\mathbf{G}^{T}||_{h}<\frac{1}{2(\eta -\tilde{\eta})}$ for $1>\eta >\tilde{%
\eta}\geq 0$, we obtain

\begin{itemize}
\item $||\mathbf{G}^{T}||_{h}<\frac{1}{1-\tilde{\eta}}$ if either $(\eta ,%
\tilde{\eta})\in \mathcal{R}_{1}$ or $(\eta ,\tilde{\eta})\in \mathcal{R}%
_{3},$ where $\mathcal{R}_{1}=\left\{ (\eta ,\tilde{\eta})\in \mathcal{S}%
\text{ }|\text{ }0\leq \tilde{\eta}<\eta <\frac{1}{2}\right\} $ and $%
\mathcal{R}_{3}=\left\{ (\eta ,\tilde{\eta})\in \mathcal{S}\text{ }|\text{ }%
\eta \in \left[ \frac{1}{2},1\right) ,\text{ }\tilde{\eta}\in (2\eta -1,\eta
)\right\}$; for clarity's sake, see the partition depicted in \cref{partition}. It is obvious that $\mathcal{R}_{1}$ and $\mathcal{R}_{3}$ are
subsets of $\mathcal{D}_{1}$ and $\mathcal{D}_{1}=\mathcal{R}_{1}\cup 
\mathcal{R}_{3}\cup \mathcal{L}_{0},$ where $\mathcal{L}_{0}$ denotes $%
\mathcal{L}\smallsetminus \{(1,1)\}$. Thus, $||\mathbf{G}^{T}||_{h}<\frac{1}{%
1-\tilde{\eta}}$ if $(\eta ,\tilde{\eta})\in \mathcal{R}_{1}\cup \mathcal{R}%
_{3}=\mathcal{D}_{1}\smallsetminus \mathcal{L}_{0}$.

\item $||\mathbf{G}^{T}||_{h}<\frac{1}{2(\eta -\tilde{\eta})}$ if $1>\eta
\geq \frac{1}{2}$ and $2\eta -1\geq \tilde{\eta}\geq 0$. Hence, we have $||%
\mathbf{G}^{T}||_{h}<\frac{1}{2(\eta -\tilde{\eta})}$ \ if $(\eta ,\tilde{%
\eta})\in \mathcal{D}_{3}\smallsetminus \mathcal{L}_{1}$, where $\mathcal{L}%
_{1}=\left\{ (\eta ,\tilde{\eta})\in \mathcal{S}\text{ }|\text{ }\eta
=1\right\} .$
\end{itemize}

\noindent b) if $\eta =\tilde{\eta}$ and $\eta \neq 1,$ then $F=h$ and the resultant
metric is a Randers one in this case. Moreover, for every $(\eta ,\tilde{\eta%
})\in \mathcal{L}_{0}$, inequality \eqref{CC_SC} is equivalent to $||%
\mathbf{G}^{T}||_{h}<\frac{1}{1-\eta }$.

\noindent c) if $\eta <\tilde{\eta}$ and $\tilde{\eta}\neq 1$, then \eqref{CC_SC}
implies that $||\mathbf{G}^{T}||_{h}<\frac{1}{1-\tilde{\eta}}.$ Combining
this with the strong convexity condition for the indicatrix $I_{F},$ i.e. $||%
\mathbf{G}^{T}||_{h}<\frac{1}{2(\tilde{\eta}-\eta )}$ for $0\leq \eta <%
\tilde{\eta}<1$, we get

\begin{itemize}
\item $||\mathbf{G}^{T}||_{h}<\frac{1}{1-\tilde{\eta}}$ if either $(\eta ,%
\tilde{\eta})\in \mathcal{R}_{2}$ or $(\eta ,\tilde{\eta})\in \mathcal{R}%
_{4},$ where $\mathcal{R}_{2}=\left\{ (\eta ,\tilde{\eta})\in \mathcal{S}%
\text{ }|\text{ }0\leq \eta <\tilde{\eta}<\frac{1}{3}\right\} $ and $%
\mathcal{R}_{4}=\left\{ (\eta ,\tilde{\eta})\in \mathcal{S}\text{ }|\text{ }%
\tilde{\eta}\in \left[ \frac{1}{3},1\right) ,\text{ }\eta \in \left( \frac{3%
\tilde{\eta}-1}{2},\tilde{\eta}\right) \right\}$. It is clear that $%
\mathcal{R}_{2}$ and $\mathcal{R}_{4}$ are subsets of $\mathcal{D}_{2}$ and $%
\mathcal{D}_{2}=\mathcal{R}_{2}\cup \mathcal{R}_{4}$ (see \cref{partition}). Thus, $||\mathbf{G}%
^{T}||_{h}<\frac{1}{1-\tilde{\eta}}$ if $(\eta ,\tilde{\eta})\in \mathcal{D}%
_{2}.$

\item $||\mathbf{G}^{T}||_{h}<\frac{1}{2(\tilde{\eta}-\eta )}$ if $\frac{1}{3%
}\leq \tilde{\eta}<1$ and $0\leq \eta \leq \frac{3\tilde{\eta}-1}{2}$. It follows that $||\mathbf{G}^{T}||_{h}<\frac{1}{2(\eta -\tilde{\eta})}$ if $(\eta ,%
\tilde{\eta})\in \mathcal{D}_{4}\smallsetminus \mathcal{L}_{2}$, where $%
\mathcal{L}_{2}=\left\{ (\eta ,\tilde{\eta})\in \mathcal{S}\text{ }|\text{ }%
\tilde{\eta}=1\right\}.$
\end{itemize}

\noindent d) if $\eta \neq \tilde{\eta}$ and $\eta =1,$ then $1>\tilde{\eta}$ and %
\eqref{CC_SC} is fulfilled. Accordingly, in this case, we have $(1,\tilde{\eta})$%
-slope metric which is the Matsumoto type metric $F$ (\eqref{Matsumoto_eta} with $\eta =1),$ and the strong convexity of its
indicatrix yields $||\mathbf{G}^{T}||_{h}<\frac{1}{2(1-\tilde{\eta})}.$
Consequently, $||\mathbf{G}^{T}||_{h}<\frac{1}{2(1-\tilde{\eta})}$ if 
$(\eta ,\tilde{\eta})\in \mathcal{L}_{1}.$

\noindent e) if $\eta \neq \tilde{\eta}$ and $\tilde{\eta}=1,$ then $\eta <1$ and inequality \eqref{CC_SC} holds. It turns out that the strong convexity
corresponding to the metric $F$ given by \eqref{Matsumoto_eta}, certified by %
\cref{Lema2}, ensures the strong convexity which corresponds to the $(\eta
,1)$-slope metric, namely $||\mathbf{G}^{T}||_{h}<\frac{1}{2(1-\eta )}$. So, we have
  $||\mathbf{G}^{T}||_{h}<\frac{1}{2(1-\eta )}$ if $(\eta ,\tilde{%
\eta})\in \mathcal{L}_{2}$.

\noindent f) if $\eta =\tilde{\eta}=1,$ then there is not any deformation for $h$, 
since $\mathbf{G}_{\eta \tilde{\eta}}=0$ and thus, there is no restriction on $||%
\mathbf{G}^{T}||_{h}.$

Summing up the above findings, we obtain that inequality \eqref{CC} is
equivalent to either $||\mathbf{G}^{T}||_{h}<\frac{1}{1-\tilde{\eta}}$ and $%
(\eta ,\tilde{\eta})\in \mathcal{D}_{1}\cup \mathcal{D}_{2}$, or $||\mathbf{G%
}^{T}||_{h}<\frac{1}{2|\eta -\tilde{\eta}|}$ and $(\eta ,\tilde{\eta})\in 
\mathcal{D}_{3}\cup \mathcal{D}_{4}$.

The argument which proves the equivalence ii) $\Leftrightarrow $ iii)
is that $||\mathbf{G}_{\eta \tilde{\eta}}||_{h}\leq ||\mathbf{G}^{T}||_{h},$
for any $(\eta ,\tilde{\eta})\in \mathcal{S}$ and, furthermore, the maximum
of $||\mathbf{G}_{\eta \tilde{\eta}}||_{h}$ coincides with $||\mathbf{G}^{T}||_{h}$ for $\eta=0$ (which is possible both in $\mathcal{D}_{1}\cup \mathcal{D}_{2}$ and in $\mathcal{D}_{3}\cup \mathcal{D}_{4}$), since $\mathbf{G}_{MAT}$ must vanish for some directions. 
\end{proof}
Based on the results stated in Steps I and II, we have performed the proof
of \cref{Theorem1}.

\bigskip We remark that, according to \cref{Lema3}, the force of the active
wind $\mathbf{G}_{\eta \tilde{\eta}}$ can be accounted for in terms of 
the force of the gravitational wind $\mathbf{G}^{T}$, i.e. $||\mathbf{G}%
^{T}||_{h}<\tilde{b}_{0}$, in the problem $\mathcal{P}_{\eta ,\tilde{\eta}}$
(\cref{partition}), for any $(\eta ,\tilde{\eta})\in 
\mathcal{S}$, where 
\begin{equation}
\ \tilde{b}_{0}=\left\{ 
\begin{array}{cc}
\frac{1}{1-\tilde{\eta}}, & \text{if }(\eta ,\tilde{\eta})\in \mathcal{D}%
_{1}\cup \mathcal{D}_{2} \\ 
\frac{1}{2|\eta -\tilde{\eta}|}, & \text{if }(\eta ,\tilde{\eta})\in 
\mathcal{D}_{3}\cup \mathcal{D}_{4}%
\end{array}%
.\right.  \label{Strong_C}
\end{equation}

It is worthwhile to mention a few observations regarding the range of $%
\tilde{b}_{0}$. For example, when $(\eta ,\tilde{\eta})\in \mathcal{R}_{1}$,
it follows that $\tilde{b}_{0}\in \lbrack 1,2)$ or when $(\eta ,\tilde{\eta}%
)\in \mathcal{R}_{2}$, we obtain $\tilde{b}_{0}\in (1,\frac{3}{2})$.
Moreover, for $(\eta ,\tilde{\eta})\in \mathcal{R}_{3}\cup \mathcal{R}_{4}$, 
$\tilde{b}_{0}\longrightarrow \infty $ as $\tilde{\eta}\nearrow 1.$
Similarly, for $(\eta ,\tilde{\eta})\in \mathcal{D}_{3}\cup \mathcal{D}_{4},$
$\tilde{b}_{0}\longrightarrow \infty $ as $|\eta -\tilde{\eta}%
|\longrightarrow 0$. In fact, once we are closer and closer to the point $%
(0,0)\in \mathcal{\tilde{S}}$, the admitted force of the gravitational wind is weaker because $\tilde{b}_{0}\longrightarrow 1$ as $\tilde{\eta}\searrow 0.$ On the other hand, the closer we approach to the point $(1,1)\in \mathcal{\tilde{S}}$, the stronger the allowed force of $\mathbf{G}^{T}$ becomes. 
 However, there is $(\eta ,%
\tilde{\eta})\in \mathcal{S}$ such that $||\mathbf{G}^{T}||_{h}>1$ and the
indicatrix $I_{\tilde{F}_{\eta \tilde{\eta}}}$ of the $(\eta ,\tilde{\eta})$%
-slope metric $\tilde{F}_{\eta \tilde{\eta}}$ is still strongly convex,
unlike the classical navigation problems, i.e. ZNP, where $||\mathbf{G}%
^{T}||_{h}<1$, or MAT, where $||\mathbf{G}^{T}||_{h}<\frac{1}{2}$.

The allowable gravitational wind force $||\mathbf{G}^{T}||_{h}<\tilde{b}_{0}$ for the general slippery slope model determined by the strong convexity conditions, including the influence of both traction coefficients, is illustrated in \cref{partition}. Moreover,  for comparison to the scenarios including only one varying parameter as in the preceding studies of SLIPPERY (\cite{slippery}) and S-CROSS (\cite{slipperyx}), we rediscover the previous findings as the particular cases in the current investigation and juxtaposed  them with the missing scenarios, i.e. R-MAT and R-CROSS so that all four boundary cases of $\mathcal{\tilde{S}}$ can now be pieced together. This is shown in \cref{VVV}.

\begin{figure}[h]
\centering
\includegraphics[width=0.355\textwidth]{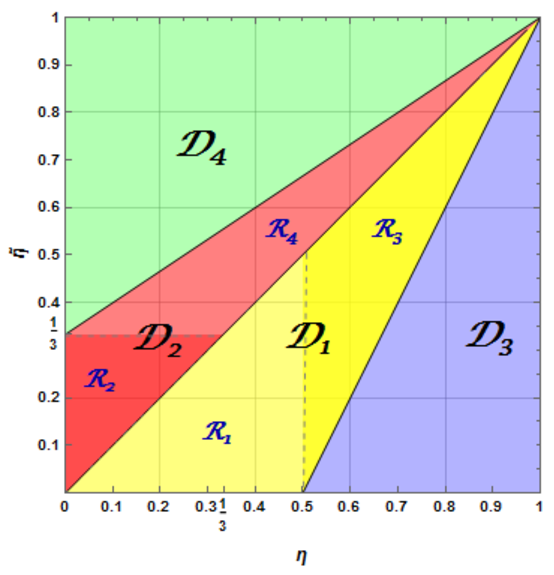}
\ \includegraphics[width=0.63\textwidth]{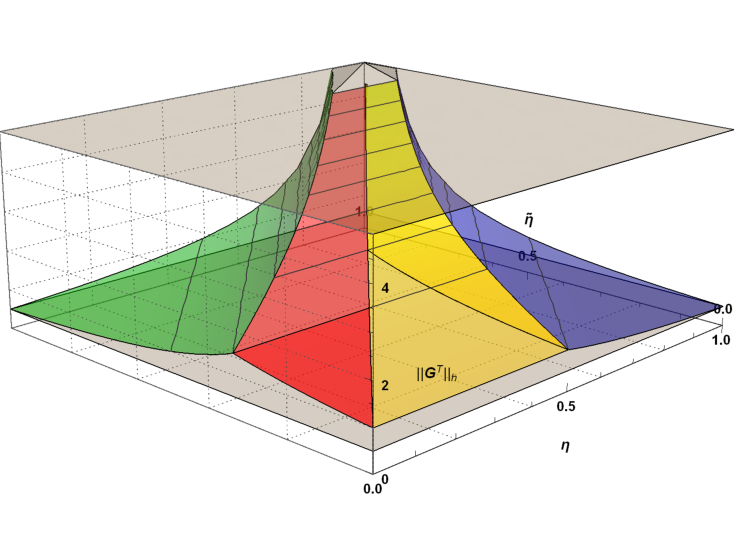} 
\caption{Left: The partition of the problem square diagram $\mathcal{\tilde{S}}=\mathcal{S}\cup \{(1,1)\}$ as in \eqref{DDD}, where $\mathcal{D}_{1}=\mathcal{R}_{1}\cup \mathcal{R}_{3}\cup \mathcal{L}_{0},$ $\mathcal{D}_{2}=\mathcal{R}_{2}\cup \mathcal{R}_{4}$, where $\mathcal{L}_{0}=\mathcal{L}\smallsetminus \{(1,1)\}$, and  $\mathcal{L}=\left\{ (\eta ,\tilde{\eta}
)\in \mathcal{\tilde{S}}\text{ }|\text{ }\eta =\tilde{\eta}\right\}$. 
Right: The allowable gravitational wind force $||\mathbf{G}^{T}||_h$ in the general slippery slope model determined by the strong convexity conditions given by \eqref{Strong_C}, including the influence of both traction coefficients ($\eta$, $\tilde{\eta}$), i.e. $||\mathbf{G}^{T}||_{h}<\tilde{b}_{0}$.  
For clarity of presentation, we limited the range $||\mathbf{G}^{T}||_h<5$ (the upper gray plane), remarking that $||\mathbf{G}^{T}||_h\longrightarrow\infty$ in the neighbourhood of the Riemannian corner, i.e. when $(\eta ,\tilde{\eta})\longrightarrow (1, 1)$. The lower plane (gray) represents $||\mathbf{G}^{T}||_h=0.5$ which refers to MAT, i.e. $(1, 0)$ as well as CROSS, i.e. $(0, 1)$. The colors of the related parts in both subfigures correspond to each other.} 
\label{partition}
\end{figure}
\begin{figure}[h!]
\centering
~\includegraphics[width=0.55\textwidth]{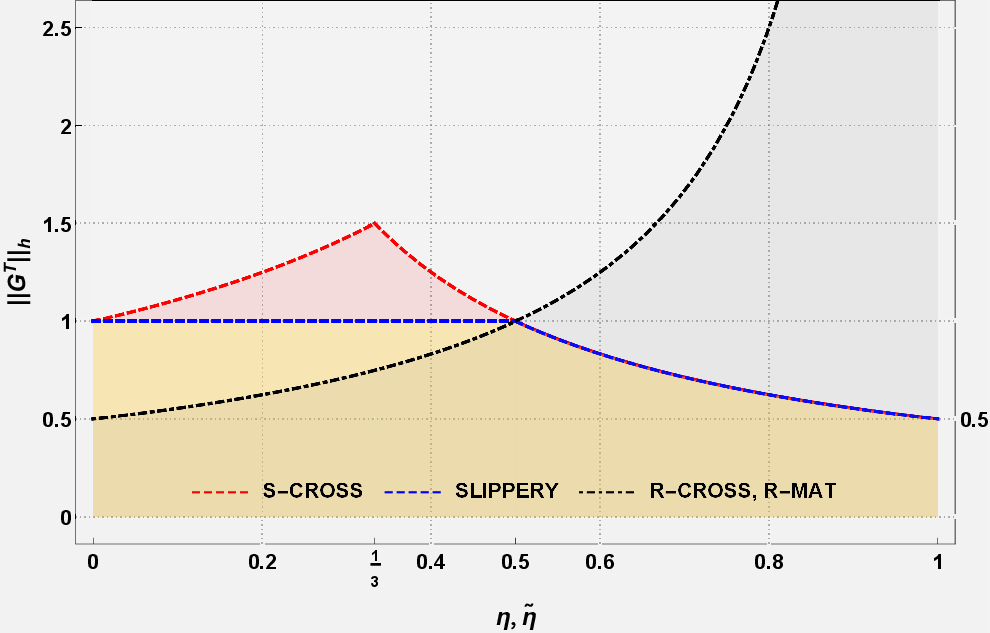}
\caption{A comparison of the orthographically projected images of the surface $||\mathbf{G}^{T}||_{h}<\tilde{b}_{0}(\eta ,\tilde{\eta})$ as in \cref{partition} (right) on the vertical planes which illustrate the least upper bounds $\tilde{b}_{0}(\eta ,\tilde{\eta})$ of the gravitational wind force $||\mathbf{G}^{T}||_h$ due to the strong convexity conditions as in \eqref{Strong_C} for four boundary cases of the problem square diagram $\mathcal{\tilde{S}}$, i.e. SLIPPERY ($\tilde{\eta}=0, \eta\in[0,1]$), S-CROSS ($\eta=0, \tilde{\eta}\in[0,1]$), R-MAT ($\eta=1,  \tilde{\eta}\in[0,1]$) and R-CROSS $(\tilde{\eta}=1,\eta\in[0,1]$). The outcomes related to R-MAT and R-CROSS overlap. For clarity of presentation, the upper range of $||\mathbf{G}^{T}||_h$ has been limited approximately to $2.5$, however $||\mathbf{G}^{T}||_h\longrightarrow\infty$, when $(\eta ,\tilde{\eta})\longrightarrow (1, 1)$.}
\label{VVV}
\end{figure}

Finally, we briefly discuss a kind of classification of the navigation
problems $\mathcal{P}_{\eta ,\tilde{\eta}}$, for any $(\eta ,\tilde{\eta})\in \mathcal{S}$. Taking into account the
decompositions of the active wind $\mathbf{G}_{\eta \tilde{\eta}},$ we can
give the following classification.

\begin{corollary}
Let $\mathcal{P}_{\eta ,\tilde{\eta}}$ be a navigation problem  on a slippery slope of a mountain $(M,h)$ under the action of an active wind $\mathbf{G}_{\eta \tilde{\eta}}$ given in \eqref{wind_general}, with a
cross-traction coefficient $\eta \in \lbrack 0,1]$, an along-traction
coefficient $\tilde{\eta}\in \lbrack 0,1]$ and a gravitational wind $\mathbf{%
G}^{T}$ on $M$. The following statements hold:

\begin{itemize}
\item[i)] For any $(\eta ,\tilde{\eta})\in \mathcal{S}$ with $\eta >\tilde{%
\eta}$, $\mathcal{P}_{\eta ,\tilde{\eta}}$ comes from SLIPPERY
 with a certain form for the cross-traction coefficient,
namely, $c_{1}=\frac{\eta -\tilde{\eta}}{1-\tilde{\eta}}\in (0,1];$

\item[ii)] For any $(\eta ,\tilde{\eta})\in \mathcal{S}$ with $\eta <\tilde{%
\eta}$, $\mathcal{P}_{\eta ,\tilde{\eta}}$ comes from S-CROSS
 with a certain form for the along-traction coefficient,
namely, $c_{2}=\frac{\tilde{\eta}-\eta }{1-\eta }\in (0,1].$
\end{itemize}
\end{corollary}

\begin{proof}
i) Making use of $\mathbf{G}^{T}=\mathbf{G}_{MAT}+\mathbf{G}_{MAT}^{\bot }$,
 we see that $\mathbf{G}_{\eta \tilde{\eta}}=(\tilde{\eta}-\eta )\mathbf{%
G}_{MAT}^{\bot }+(1-\tilde{\eta})\mathbf{G}^{T}$. Since $(\eta ,\tilde{\eta}%
)\in \mathcal{S}$ and $\eta >\tilde{\eta},$ then $\tilde{\eta}\neq 1.$ So,
we can find out that
\begin{equation*}
\mathbf{G}_{\eta \tilde{\eta}}=(1-\tilde{\eta})(-\frac{\eta -\tilde{\eta}}{1-%
\tilde{\eta}}\mathbf{G}_{MAT}^{\bot }+\mathbf{G}^{T})=(1-\tilde{\eta})%
\mathbf{G}_{c_{1}},
\end{equation*}%
where $\mathbf{G}_{c_{1}}=c_{1}\mathbf{G}_{MAT}+(1-c_{1})\mathbf{G}^{T}$ is
the active wind from SLIPPERY with a particular cross-traction coefficient $c_{1}=\frac{\eta -\tilde{\eta}}{1-\tilde{\eta}}\in (0,1],$ for any $(\eta ,\tilde{\eta})\in \mathcal{S}$, where $\eta >%
\tilde{\eta}.$ According to \cite[Theorem 1.1]{slippery}, the slippery slope
metric $\tilde{F}_{c_{1}}$ corresponding to the active wind $\mathbf{G}%
_{c_{1}}$ satisfies the equation%
\begin{equation}
\tilde{F}_{c_{1}}\sqrt{\alpha ^{2}+2(1-c_{1})\bar{g}\beta \tilde{F}%
_{c_{1}}+(1-c_{1})^{2}||\mathbf{G}^{T}||_{h}^{2}\tilde{F}_{c_{1}}^{2}}%
=\alpha ^{2}+(2-c_{1})\bar{g}\beta \tilde{F}_{c_{1}}+(1-c_{1})||\mathbf{G}%
^{T}||_{h}^{2}\tilde{F}_{c_{1}}^{2}.  \label{SLIPPERY_TH}
\end{equation}%
If we now substitute $(1-\tilde{\eta})\mathbf{G}^{T}$ for $\mathbf{G}^{T}$ everywhere in \eqref{SLIPPERY_TH} including also $\tilde{F}_{c_{1}}$, and $c_{1}=\frac{\eta -\tilde{\eta}}{1-\tilde{\eta}}$, it turns out that the new $\tilde{F}_{c_{1}}$ verifies identically \eqref{MAMA_general}. Thus, we have proved the claim i).

ii) Since $(\eta ,\tilde{\eta})\in \mathcal{S}$ and $\eta <\tilde{\eta},$
then $\eta \neq 1.$ To prove ii) we consider S-CROSS with the active wind $\mathbf{G}_{c_{2}}=-c_{2}\mathbf{G}_{MAT}+\mathbf{G}%
^{T}$, where $c_{2}=\frac{\tilde{\eta}-\eta }{1-\eta }\in (0,1]$ is a
certain along-traction coefficient. By using \cite[Theorem 1.1]%
{slipperyx}, the slippery-cross-slope metric $\tilde{F}_{c_{2}}$
corresponding to $\mathbf{G}_{c_{2}}$ verifies the equation%
\begin{equation}
\tilde{F}_{c_{2}}\sqrt{\alpha ^{2}+2\bar{g}\beta \tilde{F}_{c_{2}}+||\mathbf{%
G}^{T}||_{h}^{2}\tilde{F}_{c_{2}}^{2}}=\alpha ^{2}+(2-c_{2})\bar{g}\beta 
\tilde{F}_{c_{2}}+(1-c_{2})||\mathbf{G}^{T}||_{h}^{2}\tilde{F}_{c_{2}}^{2}.
\label{SLIPPERY_CROSS_TH}
\end{equation}%
It is immediate to verify that $\mathbf{G}_{\eta \tilde{\eta}}=(1-\eta )%
\mathbf{G}_{c_{2}}$, since $\mathbf{G}_{\eta \tilde{\eta}}=(\eta -\tilde{\eta%
})\mathbf{G}_{MAT}+(1-\eta )\mathbf{G}^{T}.$ By substituting $(1-\eta )\mathbf{G}^{T}$ for $\mathbf{G}^{T}$ everywhere in \eqref{SLIPPERY_CROSS_TH}  including  $\tilde{F}_{c_{2}},$ and $c_{2}=\frac{\tilde{\eta}-\eta }{1-\eta },$ we get that the new metric $\tilde{F}_{c_{2}}$ satisfies \eqref{MAMA_general}. Hence, it follows the required claim. 
\end{proof}

\begin{remark}
To solve the slippery slope problem $\mathcal{P}_{\eta ,\tilde{\eta}}$, for each $(\eta ,\tilde{\eta})\in \mathcal{S}$, the implicit form of the slippery slope metric included in equation \eqref{MAMA_4} is sufficient\footnote{We mention that making use of a computational software, one can generate
all four roots of \eqref{MAMA_4}, however their explicit
expressions are very complicated.}, as
will be shown in the next section. Obviously, from  \eqref{MAMA_4} we can extract  the simple explicit expressions for the slippery slope metric $\tilde{F}_{\eta \tilde{\eta}}$ in the reduced
MAT and reduced ZNP cases. Namely, if $\eta =1$ and $\tilde{\eta}\in \lbrack 0,1)$, 
then $\tilde{F}_{1\tilde{\eta}}=\tilde{F}_{_{R-MAT}}=\frac{\alpha ^{2}}{%
\alpha -(1-\tilde{\eta})\bar{g}\beta }$, as well as for any $(\eta ,\tilde{\eta}%
)\in \mathcal{S}$, where $\eta =\tilde{\eta}$, $\mathcal{P}_{\eta ,\tilde{%
\eta}}$ comes from the Zermelo navigation with $\mathbf{G}_{\eta \tilde{\eta}%
}=(1-\eta )\mathbf{G}^{T}$, i.e. R-ZNP. Indeed, the condition $\eta =\tilde{%
\eta}$ reduces  \eqref{MAMA_4} to%
\begin{equation*}
\lbrack 1-(1-\eta )^{2}||\mathbf{G}^{T}||_{h}^{2}]\tilde{F}^{2}-2(1-\eta )%
\bar{g}\beta \tilde{F}-\alpha ^{2}=0,
\end{equation*}%
which admits only the positive root $\tilde{F}_{\eta \eta }=\tilde{F}%
_{_{R-ZNP}}=\frac{\sqrt{[1-(1-\eta )^{2}||\mathbf{G}^{T}||_{h}^{2}]\alpha
^{2}+(1-\eta )^{2}\bar{g}^{2}\beta ^{2}}+(1-\eta )\bar{g}\beta }{1-(1-\eta
)^{2}||\mathbf{G}^{T}||_{h}^{2}},$ when $||(1-\eta )\mathbf{G}^{T}||_{h}<1.$
With the notation 
\begin{equation*}
\tilde{\alpha}^{2}=\frac{\alpha ^{2}}{1-(1-\eta )^{2}||\mathbf{G}%
^{T}||_{h}^{2}}+\tilde{\beta}^{2}\quad \text{ and }\quad \tilde{\beta}=\frac{%
(1-\eta )\bar{g}\beta }{1-(1-\eta )^{2}||\mathbf{G}^{T}||_{h}^{2}},
\end{equation*}%
it turns out that $\ \tilde{F}_{_{R-ZNP}}=\tilde{\alpha}+\tilde{\beta}$,
i.e. the Finsler metric of Randers type which solves Zermelo's navigation problem under the
weak wind ($||(1-\eta )\mathbf{G}^{T}||_{h}<1$); see \cite{SH,CJS}. 
\end{remark}

\section{Proof of Theorem 1.2}

\label{Sec_5}

Following the strategy presented in the recent research \cite%
{slippery,cross,slipperyx} and basing on some technical computations as well as applying \cref{Prop2}, we achieve the spray coefficients related to the $%
(\eta ,\tilde{\eta})$-slope metric $\tilde{F}_{\eta \tilde{\eta}}.$  
By \eqref{geo1}, it is immediate to supply the
equations of time geodesics of $\tilde{F}_{\eta \tilde{\eta}}$.
Moreover, the argument that any such time geodesic is unitary with respect to $%
\tilde{F}_{\eta \tilde{\eta}}$ because before all else, it is a trajectory
in Zermelo's navigation developed in Step II, will help us perform the proof
of \cref{Theorem2}.

More precisely, we derive the equations of time geodesics of $\tilde{F%
}_{\eta \tilde{\eta}}$ under the restriction $\tilde{F}_{\eta \tilde{\eta}%
}=1.$ Thus, the time-minimal paths on $(M,h)$ under the action of the active
wind $G_{\eta \tilde{\eta}}$ provided by \cref{Theorem2} represent time
geodesics of $\tilde{F}_{\eta \tilde{\eta}}$ restricted to the indicatrix $%
I_{\tilde{F}_{\eta \tilde{\eta}}}$.

We start by outlining an essential property regarding the $(\eta ,\tilde{\eta%
})$-slope metric $\tilde{F}_{\eta \tilde{\eta}}$, which allows then to use %
\cref{Prop2}. Namely, since $\tilde{F}_{\eta \tilde{\eta}}$ is the root of %
\eqref{MAMA_4}, for any $(\eta ,\tilde{\eta})\in \mathcal{\tilde{S}}$, it
seems to be a promising general $(\alpha ,\beta )$-metric.

\begin{proposition}
\label{PropXX}The $(\eta ,\tilde{\eta})$-slope metric $\tilde{F}%
_{\eta \tilde{\eta}}$ is a general $(\alpha ,\beta )$-metric.
\end{proposition}

\begin{proof}
To prove the claim that $\tilde{F}_{\eta \tilde{\eta}}$ is indeed a
general $(\alpha ,\beta )$-metric, we make the notation $\tilde{\phi}=%
\frac{\tilde{F}}{\alpha }$ and $s=\frac{\beta }{\alpha }.$ Now, if \ we
divide \eqref{MAMA_4} by $\alpha ^{4},$ we get the equation 
\begin{equation}
\begin{array}{c}
(1-\eta )^{2}||\mathbf{G}^{T}||_{h}^{2}[1-\left( 1-\tilde{\eta}\right) ^{2}||%
\mathbf{G}^{T}||_{h}^{2}]\tilde{\phi}^{4}+2(1-\eta )\left[ 1-\left( 2-\eta -%
\tilde{\eta}\right) \left( 1-\tilde{\eta}\right) ||\mathbf{G}^{T}||_{h}^{2}%
\right] \bar{g}s\tilde{\phi}^{3} \\ 
~ \\ 
+[1-2(1-\eta )\left( 1-\tilde{\eta}\right) ||\mathbf{G}^{T}||_{h}^{2}-\left(
2-\eta -\tilde{\eta}\right) ^{2}\bar{g}^{2}s^{2}]\tilde{\phi}^{2}-2\left(
2-\eta -\tilde{\eta}\right) \bar{g}s\tilde{\phi}-1=0.%
\end{array}
\label{PHI}
\end{equation}%
This is obviously equivalent to \eqref{MAMA_4}. Furthermore, since $\tilde{F}%
_{\eta \tilde{\eta}}$ is the sole positive root of \eqref{MAMA_4}, it
follows that \eqref{PHI} also admits a unique positive root, denoted by $%
\tilde{\phi}_{\eta \tilde{\eta}},$ for any $(\eta ,\tilde{\eta})\in \mathcal{%
\tilde{S}}$. Pointing out that $\eta $ and $\tilde{\eta}$ are only
parameters, it turns out that $\tilde{\phi}_{\eta \tilde{\eta}}$ depends on
the variables $||\mathbf{G}^{T}||_{h}^{2}=\ \bar{g}^{2}b^{2}$ and $s=\frac{\beta 
}{\alpha },$ where $\alpha $ and $\beta $ are given by \eqref{NOT}, i.e. $%
\tilde{\phi}_{\eta \tilde{\eta}}=\tilde{\phi}_{\eta \tilde{\eta}}(||\mathbf{G%
}^{T}||_{h}^{2},s)$ as well as $\tilde{\phi}_{\eta \tilde{\eta}}$ is a
positive $C^{\infty }$-function because $\tilde{F}_{\eta \tilde{\eta}%
}(x,y)=\alpha \tilde{\phi}_{\eta \tilde{\eta}}(||\mathbf{G}%
^{T}||_{h}^{2},s). $ Thus, the requested claim is proved.
\end{proof}

There are still some emerging properties regarding the function $%
\tilde{\phi}_{\eta \tilde{\eta}}$ as well as its derivatives. An essential
role in our study is played by the following identity%
\begin{equation}
\begin{array}{c}
(1-\eta )^{2}||\mathbf{G}^{T}||_{h}^{2}[1-\left( 1-\tilde{\eta}\right) ^{2}||%
\mathbf{G}^{T}||_{h}^{2}]\tilde{\phi}_{\eta \tilde{\eta}}^{4}+2(1-\eta )%
\left[ 1-\left( 2-\eta -\tilde{\eta}\right) \left( 1-\tilde{\eta}\right) ||%
\mathbf{G}^{T}||_{h}^{2}\right] \bar{g}s\tilde{\phi}_{\eta \tilde{\eta}}^{3}
\\ 
~ \\ 
+\{\left[ 1-2(1-\eta )\left( 1-\tilde{\eta}\right) ||\mathbf{G}^{T}||_{h}^{2}%
\right] \alpha ^{2}-\left( 2-\eta -\tilde{\eta}\right) ^{2}\bar{g}^{2}s^{2}\}%
\tilde{\phi}_{\eta \tilde{\eta}}^{2}-2\left( 2-\eta -\tilde{\eta}\right) 
\bar{g}s\tilde{\phi}_{\eta \tilde{\eta}}-1=0,%
\end{array}
\label{3.3}
\end{equation}%
which follows from the fact that $\tilde{\phi}_{\eta \tilde{\eta}}$ checks
identically \eqref{PHI}, for any $(\eta ,\tilde{\eta})\in \mathcal{%
\tilde{S}}$. Having the inequality $||\mathbf{G}^{T}||_{h}<\tilde{b}_{0},$
with $\tilde{b}_{0}$ defined in \eqref{Strong_C}, which  secures the strong convexity of the indicatrix $I_{\tilde{F}_{\eta \tilde{\eta}}}$ according to \cref{Theorem1}, we can apply the direct implication of \cref{Prop1}.
Hence, for any $(\eta ,\tilde{\eta})\in \mathcal{\tilde{S}}$ and $s$
satisfying $|s|\leq \frac{||\mathbf{G}^{T}||_{h}}{\bar{g}}<\frac{\tilde{b}%
_{0}}{\bar{g}}$, we have guaranteed the validity of the following
inequalities%
\begin{equation*}
\tilde{\phi}_{\eta \tilde{\eta}}-s\tilde{\phi}_{\eta \tilde{\eta}2}>0,\qquad 
\bar{g}^{2}(\tilde{\phi}_{\eta \tilde{\eta}}-s\tilde{\phi}_{\eta \tilde{\eta}%
2})+(||\mathbf{G}^{T}||_{h}^{2}-\bar{g}^{2}s^{2})\tilde{\phi}_{\eta \tilde{%
\eta}22}>0,
\end{equation*}%
when $n\geq 3,$ or only the right-hand side inequality, when $n=2$.

\begin{lemma}
\label{Lema4} Let $M$ be an $n$-dimensional manifold, $n>1,$ with the $(\eta
,\tilde{\eta})$-slope metric $\tilde{F}_{\eta \tilde{\eta}}=\alpha \tilde{%
\phi}_{\eta \tilde{\eta}}(||\mathbf{G}^{T}||_{h}^{2},s).$ For any $(\eta ,%
\tilde{\eta})\in \mathcal{\tilde{S}}$, the function $\tilde{\phi}_{\eta 
\tilde{\eta}}$ and its derivative with respect to $s,$ i.e. $\tilde{\phi}%
_{\eta \tilde{\eta}2}$ hold the following relations:%
\begin{equation}
C\tilde{\phi}_{\eta \tilde{\eta}2}=\bar{g}A\tilde{\phi}_{\eta \tilde{\eta}},%
\text{ \ \ }C(\tilde{\phi}_{\eta \tilde{\eta}}-s\tilde{\phi}_{\eta \tilde{%
\eta}2})=B,\text{ \ \ }C\tilde{\phi}_{\eta \tilde{\eta}}=B+\bar{g}sA\tilde{%
\phi}_{\eta \tilde{\eta}},\text{ \ \ }\left( 2-\eta -\tilde{\eta}\right)
B-2A=(\tilde{\eta}-\eta )\tilde{\phi}_{\eta \tilde{\eta}}^{2},  \label{3.1}
\end{equation}%
where%
\begin{equation}
\begin{array}{l}
A=-\left( 1-\eta \right) \left[ 1-\left( 2-\eta -\tilde{\eta}\right) \left(
1-\tilde{\eta}\right) ||\mathbf{G}^{T}||_{h}^{2}\right] \tilde{\phi}_{\eta 
\tilde{\eta}}^{2}+\left( 2-\eta -\tilde{\eta}\right) ^{2}\bar{g}s\tilde{\phi}%
_{\eta \tilde{\eta}}+ 2-\eta -\tilde{\eta}, \\ 
~ \\ 
B=-\left[ 1-2\left( 1-\eta \right) (1-\tilde{\eta})||\mathbf{G}^{T}||_{h}^{2}%
\right] \tilde{\phi}_{\eta \tilde{\eta}}^{2}+2\left( 2-\eta -\tilde{\eta}%
\right) \bar{g}s\tilde{\phi}_{\eta \tilde{\eta}}+2, \\ 
~ \\ 
\begin{split}
C=& \ 2\left( 1-\eta \right) ^{2}||\mathbf{G}^{T}||_{h}^{2}\left[ 1-\left( 1-%
\tilde{\eta}\right) ^{2}||\mathbf{G}^{T}||_{h}^{2}\right] \tilde{\phi}_{\eta 
\tilde{\eta}}^{3}+3(1-\eta )\left[ 1-\left( 2-\eta -\tilde{\eta}\right)
\left( 1-\tilde{\eta}\right) ||\mathbf{G}^{T}||_{h}^{2}\right] \bar{g}s%
\tilde{\phi}_{\eta \tilde{\eta}}^{2} \\
& +\{\left[ 1-2(1-\eta )\left( 1-\tilde{\eta}\right) ||\mathbf{G}%
^{T}||_{h}^{2}\right] -\left( 2-\eta -\tilde{\eta}\right) ^{2}\bar{g}%
^{2}s^{2}\}\tilde{\phi}_{\eta \tilde{\eta}}-\left( 2-\eta -\tilde{\eta}%
\right) \bar{g}s
\end{split}%
\end{array}
\label{3.2}
\end{equation}%
and $A$, $B$, $C$ are evaluated at $(||\mathbf{G}^{T}||_{h}^{2},s).$
\end{lemma}

\begin{proof}
By differentiating the identity \eqref{3.3} with respect to $s,$ it follows
the first relation in \eqref{3.1}. The proof of the second identity in %
\eqref{3.1} is based on the first one and on some simple computations.
Finally, by using the notation \eqref{3.2} and \eqref{3.3}, it turns
out the last two relations in \eqref{3.1}. \end{proof}

We notice that the functions $A,$ $B,$ $C$ defined in \eqref{3.2} are
homogeneous of degree zero with respect to $y$ as a consequence of the same
homogeneity degree of $\tilde{\phi}_{\eta \tilde{\eta}},$ for any $(\eta ,%
\tilde{\eta})\in \mathcal{\tilde{S}}$.

\begin{lemma}
\label{Lema45} Let $M$ be an $n$-dimensional manifold, $n>1,$ with the $%
(\eta ,\tilde{\eta})$-slope metric $\tilde{F}_{\eta \tilde{\eta}}=\alpha 
\tilde{\phi}_{\eta \tilde{\eta}}(||\mathbf{G}^{T}||_{h}^{2},s).$ For any $%
(\eta ,\tilde{\eta})\in \mathcal{\tilde{S}}$ and $s$ such that $|s|\leq 
\frac{||\mathbf{G}^{T}||_{h}}{\bar{g}}<\frac{\tilde{b}_{0}}{\bar{g}},$ the
following statements hold:

\begin{itemize}
\item[i)] $C(||\mathbf{G}^{T}||_{h}^{2},s)\neq 0$ and%
\begin{equation}
\tilde{\phi}_{\eta \tilde{\eta}2}=\frac{\bar{g}A}{C}\tilde{\phi}_{\eta 
\tilde{\eta}},\quad \tilde{\phi}_{\eta \tilde{\eta}}-s\tilde{\phi}_{\eta 
\tilde{\eta}2}=\frac{B}{C}.  \label{333}
\end{equation}

\item[ii)] $B(||\mathbf{G}^{T}||_{h}^{2},s)\neq 0.$
\end{itemize}
\end{lemma}

\begin{proof}
i) Clearly, if $\eta =\tilde{\eta}=1,$ then $C=\tilde{\phi}_{\eta \tilde{\eta%
}}>0.$ Now we prove that $C(||\mathbf{G}^{T}||_{h}^{2},s)\neq 0$, for any $%
(\eta ,\tilde{\eta})\in \mathcal{S}.$ We assume by contradiction that there
exists $s_{0}\in \lbrack -b,b],$ $b=\frac{||\mathbf{G}^{T}||_{h}}{\bar{g}}<%
\frac{\tilde{b}_{0}}{\bar{g}}$, with $\tilde{b}_{0}$ defined by %
\eqref{Strong_C}, such that $C(||\mathbf{G}^{T}||_{h}^{2},s_{0})=0.$ With
this assumption, due to \eqref{3.1}, we get 
\begin{equation}
A(||\mathbf{G}^{T}||_{h}^{2},s_{0})=B(||\mathbf{G}^{T}||_{h}^{2},s_{0})=(%
\eta -\tilde{\eta})\tilde{\phi}_{\eta \tilde{\eta}}^{2}(||\mathbf{G}%
^{T}||_{h}^{2},s_{0})=0.  \label{00}
\end{equation}%
Since $\tilde{\phi}_{\eta \tilde{\eta}}(||\mathbf{G}^{T}||_{h}^{2},s_{0})>0$
for any $(\eta ,\tilde{\eta})\in \mathcal{\tilde{S}}$, the last equality in %
\eqref{00} turns out that the only possibility is that $\eta =\tilde{\eta}%
\neq 1$, i.e. $(\eta ,\tilde{\eta})\in \mathcal{L}_{0}$. By using this fact
and \eqref{00}, the identity \eqref{3.3} is reduced to%
\begin{equation}
(1-\eta )^{2}||\mathbf{G}^{T}||_{h}^{2}[1-(1-\eta )^{2}||\mathbf{G}%
^{T}||_{h}^{2}]\tilde{\phi}_{\eta \tilde{\eta}}^{4}+[2(1-\eta )\bar{g}s_{0}%
\tilde{\phi}_{\eta \tilde{\eta}}+1]^{2}=0,  \label{3.4}
\end{equation}%
where $\tilde{\phi}_{\eta \tilde{\eta}}$ is evaluated at $(||\mathbf{G}%
^{T}||_{h}^{2},s_{0})$ and it is with $\eta =\tilde{\eta}\neq 1.$ As $||%
\mathbf{G}^{T}||_{h}<\frac{1}{1-\eta }$ for $(\eta ,\tilde{\eta})\in 
\mathcal{L}_{0}$, \eqref{3.4} contradicts the fact that $(1-\eta )^{2}||%
\mathbf{G}^{T}||_{h}^{2}[1-(1-\eta )^{2}||\mathbf{G}^{T}||_{h}^{2}]\tilde{%
\phi}_{\eta \tilde{\eta}}^{4}(||\mathbf{G}^{T}||_{h}^{2},s_{0})>0.$ Thus, we
have shown that $C\neq 0$ everywhere. Moreover, making use of \eqref{3.1},
the claims \eqref{333} are fulfilled.

ii) If $\eta =\tilde{\eta}=1,$ then $B=\tilde{\phi}_{\eta \tilde{\eta}%
}^{2}=1.$ However, it remains to prove that $B(||\mathbf{G}%
^{T}||_{h}^{2},s)\neq 0$, for any $(\eta ,\tilde{\eta})\in \mathcal{S}$. We
assume towards a contradiction that there is $\tilde{s}\in \lbrack -b,b],$ 
$b=\frac{||\mathbf{G}^{T}||_{h}}{\bar{g}}<\frac{\tilde{b}_{0}}{\bar{g}}$,
with $\tilde{b}_{0}$ defined by \eqref{Strong_C}, such that $B(||\mathbf{G}%
^{T}||_{h}^{2},\tilde{s})=0$. So, we are searching now for such an $\tilde{s}
$.

On one hand, since $\tilde{\phi}_{\eta \tilde{\eta}}(||\mathbf{G}%
^{T}||_{h}^{2},\tilde{s})>0,$ $C(||\mathbf{G}^{T}||_{h}^{2},\tilde{s})\neq 0$
and $B(||\mathbf{G}^{T}||_{h}^{2},\tilde{s})=0$, the third formula in %
\eqref{3.1} with $s=\tilde{s}$ implies that $\tilde{s}$ $\neq 0$. On the
other hand, due to our assumption, by the second formula in \eqref{3.2}, it
follows that $\tilde{\phi}_{\eta \tilde{\eta}}(||\mathbf{G}^{T}||_{h}^{2},%
\tilde{s})$ satisfies the polynomial equation%
\begin{equation}
\lbrack 1-2(1-\eta )(1-\tilde{\eta})||\mathbf{G}^{T}||_{h}^{2}]\tilde{\phi}%
_{\eta \tilde{\eta}}^{2}-2(2-\eta -\tilde{\eta})\bar{g}\tilde{s}\tilde{\phi}%
_{\eta \tilde{\eta}}-2=0.  \label{SI}
\end{equation}%
Moreover, for $s=\tilde{s}$ and for any $(\eta ,\tilde{\eta})\in \mathcal{S}$, \eqref{3.3} is reduced to 
\begin{equation}
\begin{array}{c}
2(1-\eta )^{2}\zeta ||\mathbf{G}^{T}||_{h}^{2}\tilde{\phi}_{\eta \tilde{\eta}%
}^{2}+\left[ 2-3\eta +\tilde{\eta}-2\left( 2-\eta -\tilde{\eta}\right)
(1-\eta )\left( 1-\tilde{\eta}\right) ||\mathbf{G}^{T}||_{h}^{2}\right] \bar{%
g}\tilde{s}\tilde{\phi}_{\eta \tilde{\eta}} \\ 
\\ 
+1-2(1-\eta )\left( 1-\tilde{\eta}\right) ||\mathbf{G}^{T}||_{h}^{2}=0,%
\end{array}
\label{SII}
\end{equation}%
where $\zeta $ denotes the term $1-\left( 1-\tilde{\eta}\right) ^{2}||%
\mathbf{G}^{T}||_{h}^{2}$. Since $||\mathbf{G}^{T}||_{h}<\tilde{b}_{0}$,
with $\tilde{b}_{0}$ defined by \eqref{Strong_C}, it turns out that $\zeta
=1-\left( 1-\tilde{\eta}\right) ^{2}||\mathbf{G}^{T}||_{h}^{2}\neq 0$ for
any $(\eta ,\tilde{\eta})\in \mathcal{\tilde{S}}.\,$\ Nevertheless, there
may exist some $(\eta ,\tilde{\eta})\in \mathcal{S}\smallsetminus (\mathcal{L%
}_{1}\cup \mathcal{L}_{2})$ such that $1-2(1-\eta )(1-\tilde{\eta})||\mathbf{%
G}^{T}||_{h}^{2}=0.$ Thus, we have to analyze two cases.
\paragraph{Case a.} If $1-2(1-\eta )(1-\tilde{\eta})||\mathbf{G}^{T}||_{h}^{2}\neq 0,$ for
any $(\eta ,\tilde{\eta})\in \mathcal{S},$ then due to equations \eqref{SI} and %
\eqref{SII}, we get\bigskip 
\begin{equation}
\lbrack 1+4(1-\eta )\left( \tilde{\eta}-\eta \right) ||\mathbf{G}%
^{T}||_{h}^{2}]\tilde{\phi}_{\eta \tilde{\eta}}+4\left( \tilde{\eta}-\eta
\right) \bar{g}\tilde{s}=0.  \label{SS}
\end{equation}%
The last equation provides a contradiction when $(\eta ,\tilde{\eta})\in 
\mathcal{L}_{0}.$ Thus, $\eta \neq \tilde{\eta}$ and moreover, since $\tilde{%
s}\neq 0$ and $\tilde{\phi}_{\eta \tilde{\eta}}(||\mathbf{G}^{T}||_{h}^{2},%
\tilde{s})>0$, it turns out that $1+4(1-\eta )\left( \tilde{\eta}-\eta
\right) ||\mathbf{G}^{T}||_{h}^{2}\neq 0$ and 
\begin{equation}
\tilde{\phi}_{\eta \tilde{\eta}}(||\mathbf{G}^{T}||_{h}^{2},\tilde{s})=-%
\frac{4\left( \tilde{\eta}-\eta \right) \bar{g}\tilde{s}}{1+4(1-\eta )\left( 
\tilde{\eta}-\eta \right) ||\mathbf{G}^{T}||_{h}^{2}}.  \label{SIII}
\end{equation}%
Once we have \eqref{SIII} for any $(\eta ,\tilde{\eta})\in \mathcal{%
S\smallsetminus L}_{0}$, we can go with it into \eqref{SI} and the result
is  
\begin{equation*}
4\left( \tilde{\eta}-\eta \right) \bar{g}^{2}\tilde{s}^{2}[2-3\eta +\tilde{%
\eta}+4(\tilde{\eta}-\eta )(1-\eta )^{2}||\mathbf{G}^{T}||_{h}^{2}]=[1+4(1-%
\eta )\left( \tilde{\eta}-\eta \right) ||\mathbf{G}^{T}||_{h}^{2}]^{2}.
\end{equation*}%
Since the right-hand side of this is positive, it follows that%
\begin{equation*}
\left( \tilde{\eta}-\eta \right) [2-3\eta +\tilde{\eta}+4(\tilde{\eta}-\eta
)(1-\eta )^{2}||\mathbf{G}^{T}||_{h}^{2}]>0
\end{equation*}%
and therefore, we obtain%
\begin{equation*}
\tilde{s}^{2}=\frac{[1+4(1-\eta )\left( \tilde{\eta}-\eta \right) ||\mathbf{G%
}^{T}||_{h}^{2}]^{2}}{4\left( \tilde{\eta}-\eta \right) \bar{g}^{2}[2-3\eta +%
\tilde{\eta}+4(\tilde{\eta}-\eta )(1-\eta )^{2}||\mathbf{G}^{T}||_{h}^{2}]}
\end{equation*}%
which contradicts $\tilde{s}^{2}\in (0,b^{2}],$ for any $(\eta ,\tilde{\eta}%
)\in \mathcal{S\smallsetminus L}_{0},$ due to the condition $||\mathbf{G}%
^{T}||_{h}<\tilde{b}_{0}$, where $\tilde{b}_{0}$ is defined by %
\eqref{Strong_C}. Indeed, since we must have $\tilde{s}^{2}\leq b^{2},$ this
implies $||\mathbf{G}^{T}||_{h}\geq \frac{1}{2|\eta -\tilde{\eta}|}$ for any 
$(\eta ,\tilde{\eta})\in \mathcal{S}\smallsetminus \mathcal{L}_{0},$
noticing that $\frac{1}{2|\eta -\tilde{\eta}|}>\frac{1}{1-\tilde{\eta}}$ for
any $(\eta ,\tilde{\eta})\in \mathcal{D}_{1}\cup \mathcal{D}%
_{2}\smallsetminus \mathcal{L}_{0}.$

\paragraph{Case b.} If $1-2(1-\eta )(1-\tilde{\eta})||\mathbf{G}^{T}||_{h}^{2}=0,$ for some $%
(\eta ,\tilde{\eta})\in \mathcal{S}\smallsetminus (\mathcal{L}_{1}\cup 
\mathcal{L}_{2})$, then by \eqref{SI}, it follows that 
\begin{equation}
\tilde{\phi}_{\eta \tilde{\eta}}(||\mathbf{G}^{T}||_{h}^{2},\tilde{s})=-%
\frac{1}{(2-\eta -\tilde{\eta})\bar{g}\tilde{s}}.  \label{SIV}
\end{equation}%
The last equation together with \eqref{SII} lead to $\tilde{s}$ which
satisfies the relation
\begin{equation}
\tilde{s}^{2}\left( \tilde{\eta}-\eta \right) =\frac{1-2\eta +\tilde{\eta}}{%
4(1-\tilde{\eta})(2-\eta -\tilde{\eta})\bar{g}^{2}}.  \label{s_tilde}
\end{equation}%
Obviously, when $\eta =\tilde{\eta}$, \eqref{s_tilde} provides a
contradiction. Therefore, it remains to study \eqref{s_tilde} when $(\eta ,%
\tilde{\eta})\in \mathcal{S}\smallsetminus (\bigcup\limits _{i=0}^{2}\mathcal{L}_{i}).$
Since $\tilde{s}\neq 0$ and $\eta \neq \tilde{\eta},$ it follows that $%
1-2\eta +\tilde{\eta}\neq 0.$ Now, making use of $\tilde{s}^{2}\leq b^{2}=\frac{%
||\mathbf{G}^{T}||_{h}^{2}}{\bar{g}^{2}}$, it turns out that there exist $%
(\eta ,\tilde{\eta})\in \mathcal{\tilde{D}}\subset \mathcal{S}\smallsetminus
(\bigcup\limits _{i=0}^{2}\mathcal{L}_{i})$ such that%
\begin{equation*}
1-2(1-\eta )(1-\tilde{\eta})||\mathbf{G}^{T}||_{h}^{2}=0,
\end{equation*}%
where $\mathcal{\tilde{D}}=\mathcal{\tilde{D}}_{3}\cup \mathcal{\tilde{D}}%
_{4}$ and 
\begin{equation*}
\begin{array}{l}
\mathcal{\tilde{D}}_{3}=\left\{ (\eta ,\tilde{\eta})\in \mathcal{S}\text{ }|%
\text{ }\frac{1}{2}<\eta <1,\text{ }\tilde{\eta}<2\eta -1\right\} \subset 
\mathcal{D}_{3}, \\ 
~ \\ 
\mathcal{\tilde{D}}_{4}=\left\{ (\eta ,\tilde{\eta})\in \mathcal{S}\text{ }|%
\text{ }\eta \leq 2\tilde{\eta}-1,\text{ }\frac{1}{2}\leq \tilde{\eta}%
<1\right\} \subset \mathcal{D}_{4}.%
\end{array}%
\end{equation*}%
However, the fact that $||\mathbf{G}^{T}||_{h}<\frac{1}{2|\eta -\tilde{\eta}|%
}$ on $\mathcal{\tilde{D}}$ is contradicted.

Summing up the above findings, we have proved that $B(||\mathbf{G}%
^{T}||_{h}^{2},s)\neq 0$, for any $s\in \lbrack -b,b],$ $b=\frac{||\mathbf{%
G}^{T}||_{h}}{\bar{g}}<\frac{\tilde{b}_{0}}{\bar{g}}$, with $\tilde{b}_{0}$
defined by \eqref{Strong_C}.

\end{proof}

We remark that basing on \cref{Prop1} it is known that $\tilde{\phi}_{\eta 
\tilde{\eta}}-s\tilde{\phi}_{\eta \tilde{\eta}2}>0,$ when $n\geq 3,$ for any 
$(\eta ,\tilde{\eta})\in \mathcal{\tilde{S}}$ and $s$ such that $|s|\leq 
\frac{||\mathbf{G}^{T}||_{h}}{\bar{g}}<\frac{\tilde{b}_{0}}{\bar{g}}$. Now 
according to \cref{Lema4}, since $\tilde{\phi}_{\eta \tilde{\eta}}-s\tilde{%
\phi}_{\eta \tilde{\eta}2}=\frac{B}{C}$, we have proved that $\tilde{\phi}%
_{\eta \tilde{\eta}}-s\tilde{\phi}_{\eta \tilde{\eta}2}\neq 0$ also when $%
n=2$, for any $(\eta ,\tilde{\eta})\in \mathcal{\tilde{S}}$ and $|s|\leq 
\frac{||\mathbf{G}^{T}||_{h}}{\bar{g}}<\frac{\tilde{b}_{0}}{\bar{g}}.$

\begin{lemma}
\label{Lema5} Let $M$ be an $n$-dimensional manifold, $n>1,$ with the $(\eta
,\tilde{\eta})$-slope metric $\tilde{F}_{\eta \tilde{\eta}}=\alpha \tilde{%
\phi}_{\eta \tilde{\eta}}(||\mathbf{G}^{T}||_{h}^{2},s).$ For any $(\eta ,%
\tilde{\eta})\in \mathcal{\tilde{S}}$, the first order derivative of the
function $\tilde{\phi}_{\eta \tilde{\eta}}$ with respect to $b^{2}=\frac{||%
\mathbf{G}^{T}||_{h}^{2}}{\bar{g}^{2}},$ i.e. $\tilde{\phi}_{\eta \tilde{\eta%
}1}$ and the second order derivatives $\tilde{\phi}_{\eta \tilde{\eta}12}$
and $\tilde{\phi}_{\eta \tilde{\eta}22}$ hold the following relations: 
\begin{equation}
\begin{array}{l}
\tilde{\phi}_{\eta \tilde{\eta}1}=\frac{(1-\eta )\bar{g}^{2}}{2C}[\left( 1-%
\tilde{\eta}\right) B-(\tilde{\eta}-\eta )\tilde{\phi}_{\eta \tilde{\eta}%
}^{2}]\tilde{\phi}_{\eta \tilde{\eta}}^{2},\quad \\ 
\\ 
\tilde{\phi}_{\eta \tilde{\eta}12}=\frac{(1-\eta )\bar{g}^{3}}{2C^{3}}%
\left\{ A(B+C\tilde{\phi}_{\eta \tilde{\eta}})[\left( 1-\tilde{\eta}\right)
B-(\tilde{\eta}-\eta )\tilde{\phi}_{\eta \tilde{\eta}}^{2}]+(\tilde{\eta}%
-\eta )^{2}[2+(1-\eta )\bar{g}s\tilde{\phi}_{\eta \tilde{\eta}}]\tilde{\phi}%
_{\eta \tilde{\eta}}^{4}\right\} \tilde{\phi}_{\eta \tilde{\eta}}, \\ 
\\ 
\tilde{\phi}_{\eta \tilde{\eta}22}=\frac{\bar{g}^{2}}{C^{3}}[A^{2}B+(\tilde{%
\eta}-\eta )^{2}\tilde{\phi}_{\eta \tilde{\eta}}^{4}].%
\end{array}
\label{RRR}
\end{equation}
\end{lemma}

\begin{proof}
By differentiating the identity \eqref{3.3} with respect to $||\mathbf{G}%
^{T}||_{h}^{2}$, we get%
\begin{equation*}
\frac{\partial \tilde{\phi}_{\eta \tilde{\eta}}}{\partial ||\mathbf{G}%
^{T}||_{h}^{2}}=\frac{1-\eta }{2C}[(1-\tilde{\eta})B-(\tilde{\eta}-\eta )%
\tilde{\phi}_{\eta \tilde{\eta}}^{2}]\tilde{\phi}_{\eta \tilde{\eta}}^{2}.
\end{equation*}%
If we now substitute this into $\tilde{\phi}_{\eta \tilde{\eta}1}=\bar{g}%
^{2}\frac{\partial \tilde{\phi}_{\eta \tilde{\eta}}}{\partial ||\mathbf{G}%
^{T}||_{h}^{2}},$ the first relation in \eqref{RRR} follows. 
 Differentiating the functions \eqref{3.2} with respect to $s$, together with \eqref{3.1}
yield the following identities 
\begin{eqnarray*}
A_{2} &=&\frac{\bar{g}}{C}[2A^{2}+(2-\eta -\tilde{\eta})(\tilde{\eta}-\eta )%
\tilde{\phi}_{\eta \tilde{\eta}}^{2}],\qquad B_{2}=\frac{2\bar{g}}{C}[AB+(%
\tilde{\eta}-\eta )\tilde{\phi}_{\eta \tilde{\eta}}^{2}], \\
C_{2} &=&-\frac{\bar{g}}{C\tilde{\phi}_{\eta \tilde{\eta}}}\{AB-(\tilde{\eta}%
-\eta )[2+(2-\eta -\tilde{\eta})\bar{g}s\tilde{\phi}_{\eta \tilde{\eta}}]%
\tilde{\phi}_{\eta \tilde{\eta}}^{2}\}+3\bar{g}A,
\end{eqnarray*}%
where $A_{2}=\frac{\partial A}{\partial s},$ $B_{2}=\frac{\partial B}{%
\partial s}$ and $C_{2}=\frac{\partial B}{\partial s}.$ All these, along with%
\begin{equation*}
\begin{array}{l}
\tilde{\phi}_{\eta \tilde{\eta}12}=\frac{(1-\eta )\bar{g}^{2}}{2C^{2}}\{(1-%
\tilde{\eta})B_{2}C+2\bar{g}A[\left( 1-\tilde{\eta}\right) B-2(\tilde{\eta}%
-\eta )\tilde{\phi}_{\eta \tilde{\eta}}^{2}]-[\left( 1-\tilde{%
\eta}\right) B-(\tilde{\eta}-\eta )\tilde{\phi}_{\eta \tilde{\eta}%
}^{2}]C_{2}\}\tilde{\phi}_{\eta \tilde{\eta}}^{2}, \\ 
\\ 
\tilde{\phi}_{\eta \tilde{\eta}22}=\frac{\bar{g}}{C^{2}}(A_{2}C+\bar{g}%
A^{2}-AC_{2})\tilde{\phi}_{\eta \tilde{\eta}},%
\end{array}%
\end{equation*}%
give the last two formulas in \eqref{RRR}. \end{proof}

We are now in position to provide the spray coefficients corresponding to
the general $(\alpha ,\beta )$-metric $\tilde{F}_{\eta \tilde{\eta}}.$

\begin{lemma}
\label{Prop5} Let $M$ be an $n$-dimensional manifold, $n>1,$ with the $(\eta
,\tilde{\eta})$-slope metric $\tilde{F}_{\eta \tilde{\eta}},$ having a 
cross-traction coefficient $\eta \in \lbrack 0,1]$ and an along-traction
coefficient $\tilde{\eta}\in \lbrack 0,1]$. Then the relationship between
the spray coefficients $\tilde{\mathcal{G}}_{\eta \tilde{\eta}}^{i}$ of $%
\tilde{F}_{\eta \tilde{\eta}}$ and the spray coefficients \linebreak $%
\mathcal{G}_{\alpha }^{i}=\frac{1}{4}h^{im}\left( 2\frac{\partial h_{jm}}{%
\partial x^{k}}-\frac{\partial h_{jk}}{\partial x^{m}}\right) y^{j}y^{k}$ of 
$\alpha $ is given by 
\begin{equation}
\tilde{\mathcal{G}}_{\eta \tilde{\eta}}^{i}(x,y)=\mathcal{G}_{\alpha
}^{i}(x,y)+\left[ \mathit{\Theta }(r_{00}+2\alpha ^{2}Rr)+\alpha \mathit{%
\Omega }r_{0}\right] \frac{y^{i}}{\alpha }-\left[ \mathit{\Psi }%
(r_{00}+2\alpha ^{2}Rr)+\alpha \mathit{\Pi }r_{0}\right] \frac{w^{i}}{\bar{g}%
}-\alpha ^{2}Rr^{i},  \label{SPRAY}
\end{equation}
$i=1,...,n$, where

\begin{equation}
\begin{array}{l}
r_{00}=-\frac{1}{\bar{g}}w_{i|j}y^{i}y^{j},\text{ \ \ }r_{0}=\frac{1}{\bar{g}%
^{2}}w_{i|j}w^{j}y^{i},\text{ \ \ }r=-\frac{1}{\bar{g}^{3}}w_{i|j}w^{i}w^{j},%
\text{ \ \ }r^{i}=\frac{1}{\bar{g}^{2}}w_{\text{ }|j}^{i}w^{j}, \\ 
~ \\ 
R=\frac{\left( 1-\eta \right) \bar{g}^{2}}{2\alpha ^{4}B}[\left( 1-\tilde{%
\eta}\right) \alpha ^{2}B-\left( \tilde{\eta}-\eta \right) \tilde{F}_{\tilde{%
\eta}}^{2}]\tilde{F}_{\eta \tilde{\eta}}^{2}, \\ 
\\ 
\mathit{\Theta }=\frac{\bar{g}\alpha }{2E\tilde{F}_{\eta \tilde{\eta}}}%
[\alpha ^{6}AB^{2}-\left( \tilde{\eta}-\eta \right) ^{2}\bar{g}\beta \tilde{F%
}_{\eta \tilde{\eta}}^{5}],\text{ \ \ }\mathit{\Psi }=\frac{\bar{g}%
^{2}\alpha ^{2}}{2E}[\alpha ^{4}A^{2}B+\left( \tilde{\eta}-\eta \right) ^{2}%
\tilde{F}_{\eta \tilde{\eta}}^{4}], \\ 
~ \\ 
\mathit{\Omega }=\frac{\left( 1-\eta \right) \bar{g}^{2}}{\alpha ^{2}BE}%
\left\{ [\left( 1-\tilde{\eta}\right) \alpha ^{2}B-\left( \tilde{\eta}-\eta
\right) \tilde{F}_{\eta \tilde{\eta}}^{2}][\alpha ^{6}B^{3}+\left( \tilde{%
\eta}-\eta \right) ^{2}||\mathbf{G}^{T}||_{h}^{2}\tilde{F}_{\eta \tilde{\eta}%
}^{6}]\right. \\ \\
\text{ \ \ \ \ }\left. -\left( \tilde{\eta}-\eta \right) ^{2}\alpha ^{2}\tilde{F}%
_{\eta \tilde{\eta}}^{5}(\bar{g}\beta B+||\mathbf{G}^{T}||_{h}^{2}A\tilde{F}%
_{\eta \tilde{\eta}})\right\} , \\ 
~ \\ 
\mathit{\Pi }=\frac{\left( 1-\eta \right) \bar{g}^{3}}{2\alpha ^{3}BE}%
\left\{ [\left( 1-\tilde{\eta}\right) \alpha ^{2}B-\left( \tilde{\eta}-\eta
\right) \tilde{F}_{\eta \tilde{\eta}}^{2}][2\alpha ^{6}AB^{2}-\left( \tilde{%
\eta}-\eta \right) ^{2}\bar{g}\beta \tilde{F}_{\eta \tilde{\eta}}^{5}]\right.
\\ \\
\text{ \ \ \ \ }\left. +\left( \tilde{\eta}-\eta \right) ^{2}\alpha ^{2}B\tilde{F%
}_{\eta \tilde{\eta}}^{4}[2\alpha ^{2}+\left( 1-\eta \right) \bar{g}\beta 
\tilde{F}_{\eta \tilde{\eta}}]\right\} \tilde{F}_{\eta \tilde{\eta}},%
\end{array}
\label{terms}
\end{equation}%
with 
\begin{equation}
\begin{array}{l}
A=-\frac{1}{\alpha ^{2}}\{\left( 1-\eta \right) \left[ 1-\left( 2-\eta -%
\tilde{\eta}\right) \left( 1-\tilde{\eta}\right) ||\mathbf{G}^{T}||_{h}^{2}%
\right] \tilde{F}_{\eta \tilde{\eta}}^{2}-(2-\eta -\tilde{\eta})^{2}\bar{g}%
\beta \tilde{F}_{\eta \tilde{\eta}}-(2-\eta -\tilde{\eta})\alpha ^{2}\}, \\ 
~ \\ 
B=-\frac{1}{\alpha ^{2}}\{[1-2(1-\eta )(1-\tilde{\eta})||\mathbf{G}%
^{T}||_{h}^{2}]\tilde{F}_{\eta \tilde{\eta}}^{2}-2(2-\eta -\tilde{\eta})\bar{%
g}\beta \tilde{F}_{\eta \tilde{\eta}}-2\alpha ^{2}\}, \\ 
~ \\ 
C=\frac{1}{\alpha \tilde{F}_{\eta \tilde{\eta}}}\left( \alpha ^{2}B+\bar{g}\beta A\tilde{%
F}_{\eta \tilde{\eta}}\right) ,\text{ \ } \\ 
~ \\ 
E=\alpha ^{6}BC^{2}+(||\mathbf{G}^{T}||_{h}^{2}\alpha ^{2}-\bar{g}^{2}\beta
^{2})[\alpha ^{4}A^{2}B+(\eta -\tilde{\eta})^{2}\tilde{F}_{\eta \tilde{\eta}%
}^{4}].%
\end{array}
\label{ABC}
\end{equation}
\end{lemma}

\begin{proof}
Having the derivatives $\tilde{\phi}_{\eta \tilde{\eta}1},$ $\tilde{\phi}%
_{\eta \tilde{\eta}2},$ $\tilde{\phi}_{\eta \tilde{\eta}12}$ and $\tilde{\phi%
}_{\eta \tilde{\eta}22}$ given by equations \eqref{333} and \eqref{RRR}, a simple
computation shows that%
\begin{equation}
\begin{array}{l}
s\tilde{\phi}_{\eta \tilde{\eta}}+(b^{2}-s^{2})\tilde{\phi}_{\eta \tilde{\eta%
}2}=\frac{1}{\bar{g}C}(\bar{g}sB+||\mathbf{G}^{T}||_{h}^{2}A\tilde{\phi}%
_{\eta \tilde{\eta}}), \\ 
~ \\ 
(\tilde{\phi}_{\eta \tilde{\eta}}-s\tilde{\phi}_{\eta \tilde{\eta}2})\tilde{%
\phi}_{\eta \tilde{\eta}2}-s\tilde{\phi}_{\eta \tilde{\eta}}\tilde{\phi}%
_{\eta \tilde{\eta}22}=\frac{\bar{g}}{C^{3}}[AB^{2}-(\tilde{\eta}-\eta )^{2}%
\bar{g}s\tilde{\phi}_{\eta \tilde{\eta}}^{5}], \\ 
~ \\ 
\tilde{\phi}_{\eta \tilde{\eta}}-s\tilde{\phi}_{\eta \tilde{\eta}%
2}+(b^{2}-s^{2})\tilde{\phi}_{\eta \tilde{\eta}22}=\frac{1}{C^{3}}%
\{BC^{2}+(||\mathbf{G}^{T}||_{h}^{2}-\bar{g}^{2}s^{2})[A^{2}B+(\tilde{\eta}%
-\eta )^{2}\tilde{\phi}_{\eta \tilde{\eta}}^{4}]\}, \\ 
~ \\ 
(\tilde{\phi}_{\eta \tilde{\eta}}-s\tilde{\phi}_{\eta \tilde{\eta}2})\tilde{%
\phi}_{\eta \tilde{\eta}12}-s\tilde{\phi}_{\eta \tilde{\eta}1}\tilde{\phi}%
_{\eta \tilde{\eta}22}=\frac{\left( 1-\eta \right) \bar{g}^{3}}{2C^{4}}%
\{[\left( 1-\tilde{\eta}\right) B-(\tilde{\eta}-\eta )\tilde{\phi}_{\eta 
\tilde{\eta}}^{2}][2AB^{2}-(\tilde{\eta}-\eta )^{2}\bar{g}s\tilde{\phi}%
_{\eta \tilde{\eta}}^{5}] \\ \\
\text{ \ \ \ \ \ \ \ \ \ \ \ \ \ \ \ \ \ \ \ \ \ \ \ \ \ \ \ \ \ \ \ \ \ \ \
\ \ \ \ \ \ \ \ \ \ \ }+(\tilde{\eta}-\eta )^{2}[2+(1-\eta )\bar{g}s%
\tilde{\phi}_{\eta \tilde{\eta}}]B\tilde{\phi}_{\eta \tilde{\eta}}^{4}\}%
\tilde{\phi}_{\eta \tilde{\eta}}.%
\end{array}
\label{RRRR}
\end{equation}%
Denoting by $w^{i}$ the components of $\mathbf{G}^{T}=-\bar{g}h^{ji}\frac{%
\partial p}{\partial x^{j}}\frac{\partial }{\partial x^{i}}$ and using
the notation $w_{i}=h_{ij}w^{j}$, it follows that $w_{i}=-\bar{g}\frac{%
\partial p}{\partial x^{i}}$ and $\frac{\partial w_{i}}{\partial x^{j}}=%
\frac{\partial w_{j}}{\partial x^{i}}$. Moreover, according to \cite[Lemma 4.3%
]{slippery} we have $s_{ij}=s_{i}=s^{i}=s_{0}^{i}=s_{0}=0$ as well as the
following relations%
\begin{equation}
r_{ij}=-\frac{1}{\bar{g}}w_{i|j},\qquad r_{i}=\frac{1}{\bar{g}^{2}}%
w_{i|j}w^{j},\qquad r^{i}=\frac{1}{\bar{g}^{2}}w_{\text{ }%
|j}^{i}w^{j},\qquad r=-\frac{1}{\bar{g}^{3}}w_{i|j}w^{i}w^{j},  \label{RR}
\end{equation}%
where $w_{i|j}=\frac{\partial w_{i}}{\partial x^{j}}-\Gamma _{ij}^{k}w_{k},$ 
$w_{\text{ }|j}^{i}=h^{ik}w_{k|j}$ and $\Gamma _{ij}^{k}=\frac{1}{2}%
h^{km}\left( \frac{\partial h_{jm}}{\partial x^{i}}+\frac{\partial h_{im}}{%
\partial x^{j}}-\frac{\partial h_{ij}}{\partial x^{m}}\right) $. Collecting
the findings \eqref{RRRR} and \eqref{RR}, one can apply \cref{Prop2} and
thus, our claim yields at once. \end{proof}

We notice that a simplified form of the spray coefficients $\tilde{\mathcal{%
G}}_{\eta \tilde{\eta}}^{i}(x,y)$ occurs when $||\mathbf{G}^{T}||_{h}$ is
constant. Indeed, making use  of \cite[Lemma 4.3]{slippery} again, since $%
\frac{\partial ||\mathbf{G}^{T}||_{h}}{\partial x^{i}}=\frac{2}{\bar{g}^{2}}%
w_{i|j}w^{j}=2r_{i},$ we clearly have that $r_{i}=0$ if and only if $||%
\mathbf{G}^{T}||_{h}$ is constant and furthermore, the statement $r_{i}=0$
implies $r^{i}=r=r_{0}=0$. All these particularities reduce the formula %
\eqref{SPRAY} to 
\begin{equation}
\tilde{\mathcal{G}}_{\eta \tilde{\eta}}^{i}(x,y)=\mathcal{G}_{\alpha
}^{i}(x,y)+r_{00}\left( \Theta \frac{y^{i}}{\alpha }-\Psi \frac{w^{i}}{\bar{g%
}}\right) .
\end{equation}

It remains only to catch the ODE system which provides the time-minimizing 
trajectories $\gamma (t)=(\gamma ^{i}(t)),$ $i=1,...,n$ on the slippery
slope under the influence of the active wind $\mathbf{G}_{\eta \tilde{\eta}%
}.$ Namely, if we substitute the spray coefficients $\tilde{\mathcal{G}}_{\eta \tilde{\eta}}^{i}(\gamma (t),\dot{\gamma}(t))$ from \eqref{SPRAY} into \eqref{geo1}, with $\tilde{F}_{\eta \tilde{\eta}}(\gamma (t),\dot{\gamma}(t))=1$, it turns out the system \eqref{GGG}. This ends the proof of  \cref{Theorem2}.

\begin{remark}
Once we have the spray coefficients $\tilde{\mathcal{G}}_{\eta \tilde{\eta}%
}^{i}(x,y)$ provided by \cref{Prop5}, we can immediately supply the
equations of time geodesics which correspond to $\tilde{F}_{\eta \tilde{%
\eta}}.$ It is worthwhile to mention that the condition $C=\frac{1}{\alpha 
\tilde{F}_{\eta \tilde{\eta}}}\left( \alpha ^{2}B+\bar{g}\beta A\tilde{F}%
_{\eta \tilde{\eta}}\right) $ included in \cref{Prop5} is equivalent to %
\eqref{MAMA_4} with $\tilde{F}$ substituted by $\tilde{F}_{\eta \tilde{\eta}}
$. In particular, under restriction $\tilde{F}_{\eta \tilde{\eta}}=1$, the
condition $\tilde{C}=\frac{1}{\alpha }\left( \alpha ^{2}\tilde{B}+\bar{g}%
\beta \tilde{A}\right) $ included in \eqref{geo_tilde} is equivalent to 
\begin{equation}
\begin{array}{c}
(1-\eta )^{2}||\mathbf{G}^{T}||_{h}^{2}[1-\left( 1-\tilde{\eta}\right) ^{2}||%
\mathbf{G}^{T}||_{h}^{2}]+2(1-\eta )\left[ 1-\left( 2-\eta -\tilde{\eta}%
\right) \left( 1-\tilde{\eta}\right) ||\mathbf{G}^{T}||_{h}^{2}\right] \bar{g%
}\beta  \\ 
~ \\ 
+\{\left[ 1-2(1-\eta )\left( 1-\tilde{\eta}\right) ||\mathbf{G}^{T}||_{h}^{2}%
\right] \alpha ^{2}-\left( 2-\eta -\tilde{\eta}\right) ^{2}\bar{g}^{2}\beta
^{2}\}-2\left( 2-\eta -\tilde{\eta}\right) \bar{g}\alpha ^{2}\beta -\alpha
^{4}=0,%
\end{array}
\label{INDIC}
\end{equation}%
for any\textbf{\ }$(\eta ,\tilde{\eta})\in \mathcal{\tilde{S}},$ which
defines the indicatrix $I_{\tilde{F}_{\eta \tilde{\eta}}}.$
\end{remark}

Finally, we mention some potential generalizations of the navigation
problems $\mathcal{P}_{\eta ,\tilde{\eta}}$ which refer to a non-uniform
slippery slope, i.e. either only one or both traction coefficients
 could depend on the position $%
x\in M$, namely $\eta =\eta (x),$ $\tilde{\eta}=\tilde{\eta}(x)\in \lbrack
0,1]$. Varying one or both traction coefficients, the resultant metrics will
be more extensive than the general $(\alpha ,\beta )$-metrics, namely, $%
\tilde{F}_{\eta \tilde{\eta}}(x,y)=\alpha \tilde{\phi}_{\eta \tilde{\eta}%
}(||G^{T}||_{h}^{2},s,\eta (x))$ or $\tilde{F}_{\eta \tilde{\eta}%
}(x,y)=\alpha \tilde{\phi}_{\eta \tilde{\eta}}(||G^{T}||_{h}^{2},s,\tilde{%
\eta}(x))$ or $\tilde{F}_{\eta \tilde{\eta}}(x,y)=\alpha \tilde{\phi}_{\eta 
\tilde{\eta}}(||G^{T}||_{h}^{2},s,\eta (x),\tilde{\eta}(x))$, because of the
fact that $\tilde{\phi}_{\eta \tilde{\eta}}$ depends in addition on a third
variable or two more variables. Another further extension can occur if
we consider a varying self-speed $\tilde{f}$ of a craft on the slippery slope, i.e. $%
||u||_{h}=\tilde{f}(x)\in (0,1]$. Moreover, both traction coefficients, the self-velocity as well as  the gravitational wind can not only depend on the position, but also vary over time. This would lead to time-dependent Finsler metrics, whose time-minimizing trajectories can be modeled as lightlike geodesics of an associated Lorentz-Finsler metric (see, e.g \cite{JPS2}). 

\section{Examples}

\label{Sec_Examples}

The aim of this section is to provide some examples which support the
applicability of the above obtained results by highlighting the two-dimensional case.
First, we give a brief overview of the general model of the hill slope in dimension 2 which was also described in \cite{slippery,cross,slipperyx}. Then we focus in particular on an inclined plane as well as a triple hill based on the Gaussian bell-shaped surfaces.

Let $(M,h)$ be a surface embedded in $\mathbb{R}^{3}$, representing a model for a slippery slope of a mountain. Since $(M,h)$ is a $2$-dimensional Riemannian manifold, $(x^{1},x^{2})$ denotes a   local coordinate system on a local chart at an arbitrary point $O\in M$. If we parametrize $M$ by $(x^{1},x^{2})\in M \longmapsto (x, y, z)\in \mathbb{R}^{3},$ where  $x=x^{1},y=x^{2},z=f(x^{1},x^{2})$ and $f$ is a smooth function on $M$, then the Riemannian metric $h$ induced on $M$ is defined by $\left(
h_{ij}(x^{1},x^{2})\right) =\left( 
\begin{array}{cc}
1+f_{x^{1}}^{2} & f_{x^{1}}f_{x^{2}} \\ 
f_{x^{1}}f_{x^{2}} & 1+f_{x^{2}}^{2}%
\end{array}%
\right)$, $i,j=1,2.$ We notice that the functions $f_{x^{1}}$ and $%
f_{x^{2}} $ denote the partial derivatives of $f$ with respect to $x^{1}$
and $x^{2}$, respectively. By considering the plane $\pi _{O}$ tangent to $M$
at the point $O\in M$, it is spanned by the vectors $\frac{\partial }{%
\partial x^{1}}=(1,0,f_{x^{1}})$ and $\frac{\partial }{\partial x^{2}}%
=(0,1,f_{x^{2}})$.

Let $\mathbf{G}$ be a gravitational field in $\mathbb{R}^{3}$ which acts on
the mountain slope $M$ perpendicularly to its base, i.e. the horizontal plane $z=0$. Thus, it can be decomposed into two orthogonal components, i.e.  $\mathbf{G}=\mathbf{G}^{T}+\mathbf{G}^{\perp }$, where $\mathbf{G%
}^{T}$ is tangent to $M$ in $O$, pointing in the steepest downhill direction and $\mathbf{G}^{\perp }$ is normal to $M$ in $O$; for clarity's sake, see also \cref{fig_slope_general} in this regard. As the component $\mathbf{G}^{T}$ acts along the negative surface gradient, it stands for the gravitational wind which locally can be written as 
\begin{equation}
\mathbf{G}^{T}=-\frac{\bar{g}}{q+1}(f_{x^{1}},f_{x^{2}},q)=-\frac{\bar{g}}{%
q+1}\left( f_{x^{1}}\frac{\partial }{\partial x^{1}}+f_{x^{2}}\frac{\partial 
}{\partial x^{2}}\right),  
\label{wind}
\end{equation}%
with $||\mathbf{G}^{T}||_{h}=\bar{g}\sqrt{\frac{q}{q+1}},$ where $%
q=f_{x^{1}}^{2}+f_{x^{2}}^{2}.$ We remark that in the critical points of $M$, one has $q=0$ and thus, $\mathbf{G}^{T}$ vanishes.

Similarly to what has been done in \cite{slippery,cross,slipperyx}, having
constructed the rectangular basis $\{e_{1},e_{2}\}$ in the tangent plane $%
\pi _{O}$, where $e_{1}$ has the same direction as $\mathbf{G}^{T}$, the
parametric equations of the indicatrix of the $(\eta ,\tilde{\eta})$-slope
metric $\tilde{F}_{\eta \tilde{\eta}}$ in the coordinates $(X,Y)$, which
correspond to $\{e_{1},e_{2}\}$, can be read as 
\begin{equation}
\left\{ 
\begin{array}{rll}
X & = & [1+(\eta -\tilde{\eta})||\mathbf{G}^{T}||_{h}\cos \theta ]\cos
\theta +(1-\eta )||\mathbf{G}^{T}||_{h} \\ 
~ &  &  \\ 
Y & = & [1+(\eta -\tilde{\eta})||\mathbf{G}^{T}||_{h}\cos \theta ]\sin \theta%
\end{array}%
,\right.  \label{2_indicatrix}
\end{equation}%
for any clockwise direction $\theta \in \lbrack 0,2\pi )$ of the
self-velocity $u$, with $||u||_{h}=1,$ and for any pair $(\eta ,\tilde{\eta})\in 
\mathcal{\tilde{S}}.$ If we eliminate  the
parameter $\theta$ in the system \eqref{2_indicatrix}, then we get an equivalent equation to 
\eqref{2_indicatrix}, namely%
\begin{equation}
\sqrt{\left( X-(1-\eta )||\mathbf{G}^{T}||_{h}\right) ^{2}+Y^{2}}%
=X^{2}+Y^{2}-(2-\eta -\tilde{\eta})X||\mathbf{G}^{T}||_{h}+(1-\eta )(1-%
\tilde{\eta})||\mathbf{G}^{T}||_{h}^{2}.  \label{EE}
\end{equation}%
Moreover, having the basis $\{\frac{\partial }{\partial x^{1}},\frac{%
\partial }{\partial x^{2}}\}$ and $\{e_{1},e_{2}\}$ in $\pi _{O}$, every
tangent vector of $\pi _{O}$ can be expressed as%
\begin{equation*}
y^{1}\frac{\partial }{\partial x^{1}}+y^{2}\frac{\partial }{\partial x^{2}}%
=Xe_{1}+Ye_{2},
\end{equation*}%
with $e_{1}=-\frac{1}{\sqrt{q(q+1)}}(f_{x^{1}},f_{x^{2}},q)$ and $e_{2}=%
\frac{1}{\sqrt{q}}(f_{x^{2}},-f_{x^{1}},0).$ Accordingly, it turns out the
following link between the coordinates $(X,Y)$ and $(y^{1},y^{2})$%
\begin{equation}
X=-\sqrt{\frac{q+1}{q}}\left( y^{1}f_{x^{1}}+y^{2}f_{x^{2}}\right) ,\text{ \
\ \ \ }Y=\frac{1}{\sqrt{q}}\left( y^{1}f_{x^{2}}-y^{2}f_{x^{1}}\right).
\label{EE1}
\end{equation}%
Furthermore, since $y^{1}f_{x^{1}}+y^{2}f_{x^{2}}=-\frac{1}{\bar{g}}h(y,%
\mathbf{G}^{T})$, we obtain 
\begin{equation}
\begin{array}{l}
X^{2}+Y^{2}=(y^{1})^{2}+(y^{2})^{2}+\beta ^{2}=h_{ij}y^{i}y^{j}=\alpha ^{2},
\\ 
~ \\ 
y^{1}f_{x^{1}}+y^{2}f_{x^{2}}=\beta .%
\end{array}
\label{AB}
\end{equation}%
Taking into account the last relations, \eqref{EE} is equivalent to 
\begin{equation}
\sqrt{\alpha ^{2}+2(1-\eta )\bar{g}\beta +(1-\eta )^{2}||\mathbf{G}%
^{T}||_{h}^{2}}=\alpha ^{2}+(2-\eta -\tilde{\eta})\bar{g}\beta +(1-\eta )(1-%
\tilde{\eta})||\mathbf{G}^{T}||_{h}^{2},  \label{III}
\end{equation}%
and by applying Okubo's method, it leads to \eqref{MAMA_general}.

We notice that in the aforementioned study, the $(\eta ,\tilde{\eta})$-slope
metric $\tilde{F}_{\eta \tilde{\eta}}$, which satisfies \eqref{MAMA_general}, is described only at the regular points $O$ of $M$, i.e. where $q(O)\neq 0$. Nevertheless, it is well defined at the critical
points of $M$ as well because it is reduced to the Riemannian metric $h$ in this case.

In the sequel, based on the explicit formulas \eqref{AB} for $\alpha $ and $\beta $ as well as the relations \eqref{S1}~and~\eqref{RR}, by some
computations we get%
\begin{equation}
\begin{array}{rll}
\Gamma _{ij}^{k}&=&\frac{1}{q+1}f_{x^{k}}f_{x^{i}x^{j}},\text{ \ \ \ }\mathcal{%
G}_{\alpha }^{k}(x^{1},x^{2},y^{1},y^{2})=\frac{1}{4}r_{00}f_{x^{k}},\text{
\ }i,j,k=1,2,  \\ 
~ \\ 
r_{00}&=&\frac{1}{q+1}f_{x^{i}x^{j}}y^{i}y^{j}=\frac{1}{q+1}%
[f_{x^{1}x^{1}}(y^{1})^{2}+2f_{x^{1}x^{2}}y^{1}y^{2}+f_{x^{2}x^{2}}(y^{2})^{2}],
\\ 
\\ 
r_{0}&=&\frac{1}{(q+1)^{2}}%
[(f_{x^{1}}f_{x^{1}x^{1}}+f_{x^{2}}f_{x^{1}x^{2}})y^{1}+(f_{x^{1}}f_{x^{1}x^{2}}+f_{x^{2}}f_{x^{2}x^{2}})y^{2}],
\\ 
\\ 
r&=&\frac{1}{(q+1)^{3}}%
(f_{x^{1}}^{2}f_{x^{1}x^{1}}+2f_{x^{1}}f_{x^{2}}f_{x^{1}x^{2}}+f_{x^{2}}^{2}f_{x^{2}x^{2}}),
\\ 
\\ 
r^{1}&=&\frac{1}{(q+1)^{3}}%
[(1+f_{x^{1}}^{2})(f_{x^{1}}f_{x^{1}x^{1}}+f_{x^{2}}f_{x^{1}x^{2}})-f_{x^{1}}f_{x^{2}}(f_{x^{1}}f_{x^{1}x^{2}}+f_{x^{2}}f_{x^{2}x^{2}})],
\\ 
\\ 
r^{2}&=&\frac{1}{(q+1)^{3}}%
[-f_{x^{1}}f_{x^{2}}(f_{x^{1}}f_{x^{1}x^{1}}+f_{x^{2}}f_{x^{1}x^{2}})+(1+f_{x^{2}}^{2})(f_{x^{1}}f_{x^{1}x^{2}}+f_{x^{2}}f_{x^{2}x^{2}})],
\end{array}
\label{ABRR}
\end{equation}%
which are the essential terms in the formulas for the spray coefficients $\tilde{\mathcal{G}}_{\eta \tilde{\eta}}^{i}$ of the metric $\tilde{F}_{\eta \tilde{\eta}}.$

\subsection{Inclined plane}
\label{Sec_6.1}

We start with an inclined plane (a ramp) because this example allows us to show clearly the behavior of the indicatrix of the $(\eta ,\tilde{\eta})$-slope
metrics $\tilde{F}_{\eta \tilde{\eta}},$ for any pair $(\eta ,\tilde{\eta}%
)\in \mathcal{\tilde{S}}$. We consider the planar slope given by $%
z=x/2$, (i.e. $f(x^{1},x^{2})=x/2$, where $x=x^{1},$ $y=x^{2}$), having the
slope angle\footnote{Curiously, the Descending (Ascending) Passage of the Great Pyramid inclines at the slope of about $26.5^{\circ}$ ($26.1^{\circ}$) according to measurements, which corresponds to the rise over run ratio being very close to $0.5$, i.e. $z\approx x/2$.} $26.6^{\circ }$ and taking the regular point $O=(0,0)$ as the center of the indicatrix. In this setting, it turns out that $h=\sqrt{h_{ij}y^{i}y^{j}}$ has $\left(h_{ij}(x^{1},x^{2})\right) =\left( 
\begin{array}{cc}
5/4 & 0 \\ 
0 & 1%
\end{array}
\right), i, j=1, 2$ as well as $q=1/4,$ $\mathbf{G}^{T}=-\frac{2\bar{g}}{5}%
\frac{\partial }{\partial x^{1}}$ and $||\mathbf{G}^{T}||_{h}=\frac{\bar{g}}{%
\sqrt{5}}$. Moreover, it follows that $y^{1}=-2X/\sqrt{5}$ and $y^{2}=-Y$
and \eqref{2_indicatrix} are reduced to 
\begin{equation}
\left\{ 
\begin{array}{rll}
-\frac{\sqrt{5}}{2}y^{1} & = & [1+(\eta -\tilde{\eta})\frac{\bar{g}}{\sqrt{5}%
}\cos \theta ]\cos \theta +(1-\eta )\frac{\bar{g}}{\sqrt{5}} \\ 
~ &  &  \\ 
-y^{2} & = & [1+(\eta -\tilde{\eta})\frac{\bar{g}}{\sqrt{5}}\cos \theta
]\sin \theta%
\end{array}%
,\right.  \label{INCL_IND}
\end{equation}%
for any direction $\theta \in \lbrack 0,2\pi)$ of the velocity $u$. By applying the general theory
presented in the previous sections, the strong convexity condition $||%
\mathbf{G}^{T}||_{h}<\tilde{b}_{0},$ with $\tilde{b}_{0}$ defined in %
\eqref{Strong_C}, which corresponds to the inclined plane is equivalent to $%
\bar{g}<\delta _{1}(\eta ,\tilde{\eta})$, where 
\begin{equation}
\delta _{1}(\eta ,\tilde{\eta})=\left\{ 
\begin{array}{cc}
\frac{\sqrt{5}}{1-\tilde{\eta}}, & \text{if}\quad (\eta ,\tilde{\eta})\in 
\mathcal{D}_{1}\cup \mathcal{D}_{2} \\ 
~ &  \\ 
\frac{\sqrt{5}}{2|\eta -\tilde{\eta}|}, & \text{if}\quad (\eta ,\tilde{\eta}%
)\in \mathcal{D}_{3}\cup \mathcal{D}_{4}%
\end{array}%
.\right.  \label{plane_convexity}
\end{equation}
\begin{figure}[H]
\centering
\includegraphics[width=0.3\textwidth]{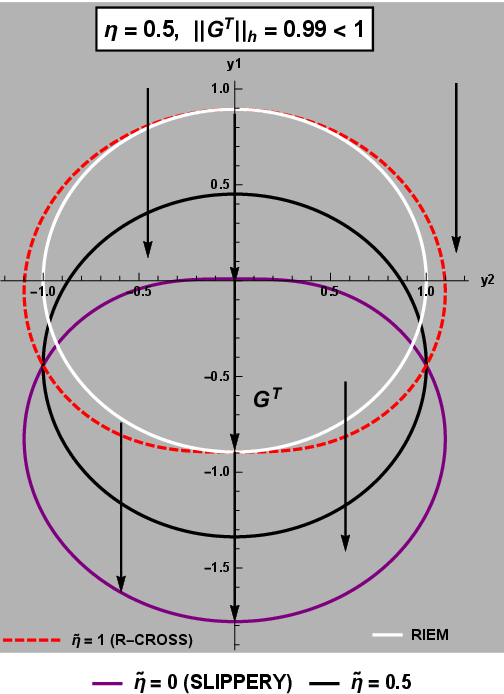}
\includegraphics[width=0.33\textwidth]{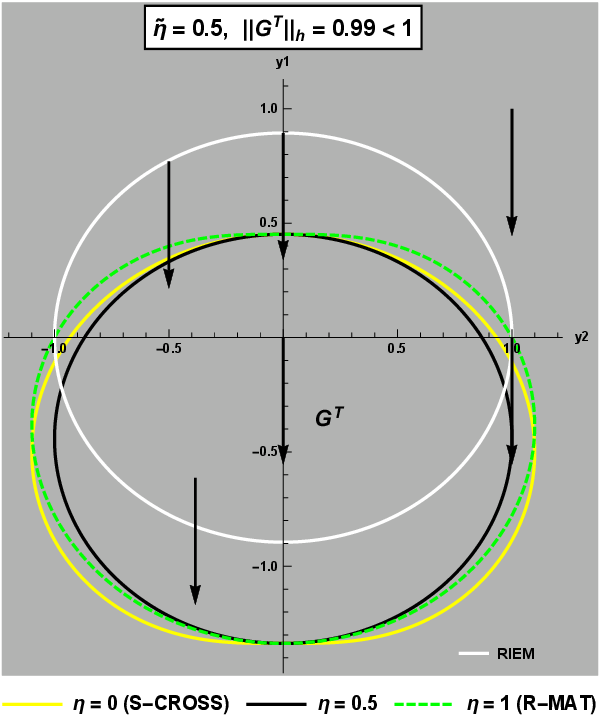}
\includegraphics[width=0.33\textwidth]{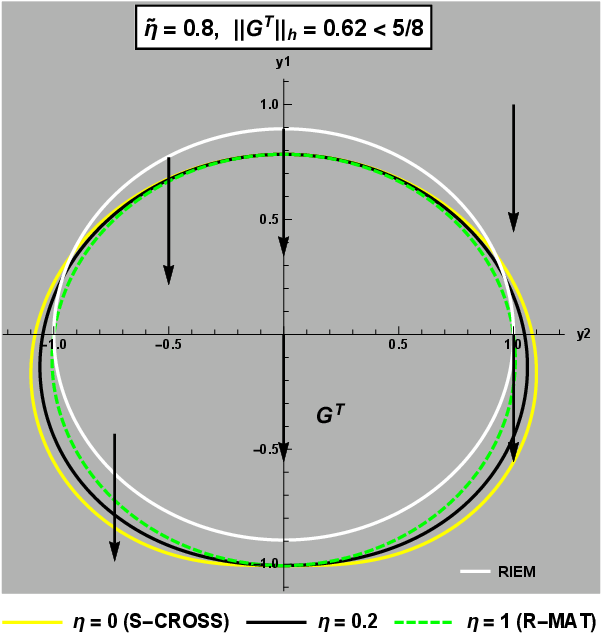}
\caption{The Finslerian indicatrices (the lima\c{c}ons) centered at $(0, 0)$ in the coordinate system $y^1Oy^2$ on the planar slippery $(\eta, \tilde{\eta})$-slope given by $z=x/2$ under almost the strongest allowable gravitational wind $\mathbf{G}^{T}$ (black arrows) in each case and for various combinations of the along- and cross-traction coefficients, i.e. $\eta$ fixed and $\tilde{\eta}$ variable (left) as well as $\tilde{\eta}$ fixed and $\eta$ variable (middle, right);  $t=1$. The steepest downhill direction on the slope is indicated by the negative axis $y^1$. The unperturbed Riemannian indicatrix on the slope is represented by a white ellipse. The supremum of the gravitational wind force $||\mathbf{G}^{T}||_{h}$ due to the strong convexity conditions \eqref{Strong_C} is equal to 1 (left, middle) and 5/8 (right).  
}
\label{fig_plane_indi}
\end{figure}

The indicatrices on the planar slippery $(\eta, \tilde{\eta})$-slope for various combinations of the along- and cross-traction coefficients in the presence of the same wind force are shown 
 in \cref{fig_plane_indi}. The behavior of the indicatrix can be clearly observed while one traction coefficient is fixed and another one is varying. In order to make the differences more visible we took almost the strongest gravitational wind admitted by the conditions for strong convexity, i.e. $||\mathbf{G}^{T}||_{h}=0.99$ (left and middle) and $||\mathbf{G}^{T}||_{h}=0.62$ (right). For comparison, the unperturbed Riemannian indicatrix is presented by a white ellipse. Clearly, all time-minimizing paths ($\tilde{F}_{\eta \tilde{\eta}}$-geodesics) are the Euclidean straight lines in this example because the Finsler metric does not depend on the position on the planar slope.  
We remark that in all presented figures including the indicatrices we applied a bit weaker wind force than its admitted supremum in each most restrictive case so that it is possible to compare all the analyzed indicatrices. 

In what follows, we select five problems determined by the pairs of the fixed (both) traction coefficients, which are shown on the partition of the problem square diagram in \cref{partition2}. For each case $P_1-P_5$ the deformations of the corresponding indicatrices are presented in \cref{fig_plane_indi2}, taking into account the impact of varying gravitational wind forces $||\mathbf{G}^{T}||_{h}$. The Riemannian elliptical $h$-circle, where the influence of gravity vanishes, is marked in white. In particular, observe that, unlike the classical Zermelo (Randers) case\footnote{The strong convexity condition in ZNP reads $\bar{g}<\sqrt{5}$ because $||\mathbf{G}^{T}||_{h}<1$.}, the wind force can exceed $||u||_h=1$, e.g. $||\mathbf{G}^{T}||_{h}<2$ for $(0.5, 0.5)$-slope (top left) or $||\mathbf{G}^{T}||_{h}<1.5$ for $(0.2, 1/3)$-slope (top middle); for clarity, see \eqref{Strong_C} in this regard. All outcomes being the convex lima\c{c}ons are compared on the bottom right in \cref{fig_plane_indi2}.

\begin{figure}[h!]
\centering
\includegraphics[width=0.39\textwidth]{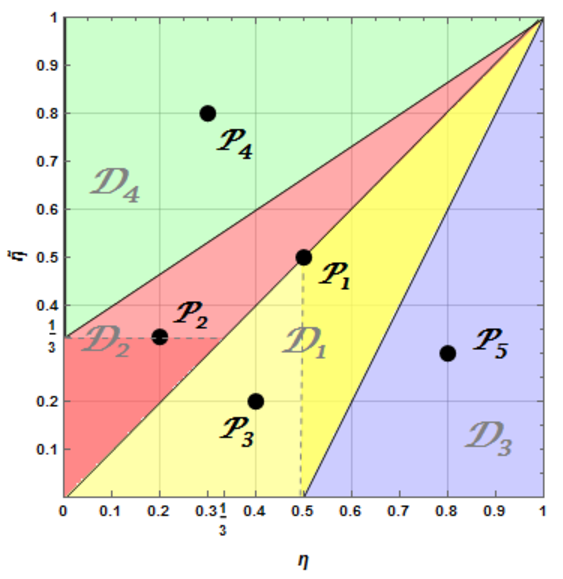}
\caption{The navigation problems related to the $(\eta, \tilde{\eta})$-slopes, i.e. $\mathcal{P}_1=(0.5, 0.5)$, $\mathcal{P}_2=(0.2, 1/3)$, $\mathcal{P}_3=(0.4, 0.2)$, $\mathcal{P}_4=(0.3, 0.8)$, $\mathcal{P}_5=(0.8, 0.3)$, selected from different subareas of the problem square diagram $\mathcal{\tilde{S}}$ (see also \cref{fig_square2}), for which the behavior of the corresponding indicatrices under the action of gravitational wind of the varying force are next compared in \cref{fig_plane_indi2}. In the background, the color-coded partition of $\mathcal{\tilde{S}}$ as in \cref{partition} (left).}
\label{partition2}
\end{figure}

Notice that in general we can arrive at the Riemannian case twofold. First, no matter the gravitational wind force, while both parameters are equal to 1. This means that there is the greatest compensation of the gravitational wind due to traction(s) (i.e. the only non-slippery
slope in our model), so the dead wind is maximal. Second, no matter the traction, while $\mathbf{G}^{T}=\overrightarrow{\mathbf{0}}$. This particular scenario can occur if $\bar{g}=0$ (no gravity), since the slope is not horizontal, so its gradient is not vanished.

\begin{figure}[h!]
\centering
\includegraphics[width=0.316\textwidth]{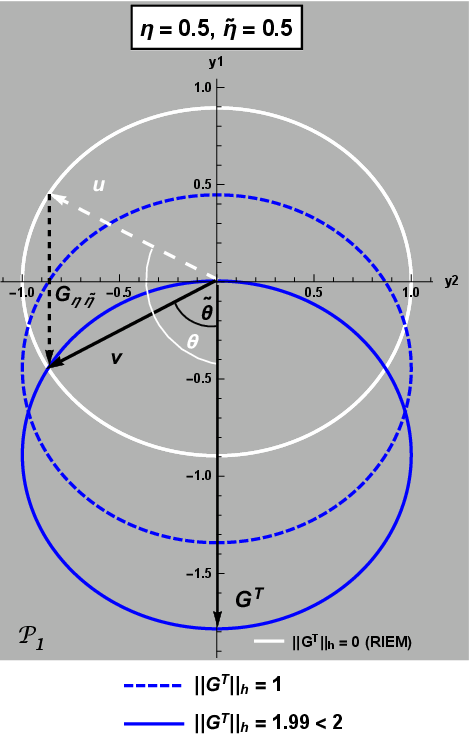} 
\includegraphics[width=0.32\textwidth]{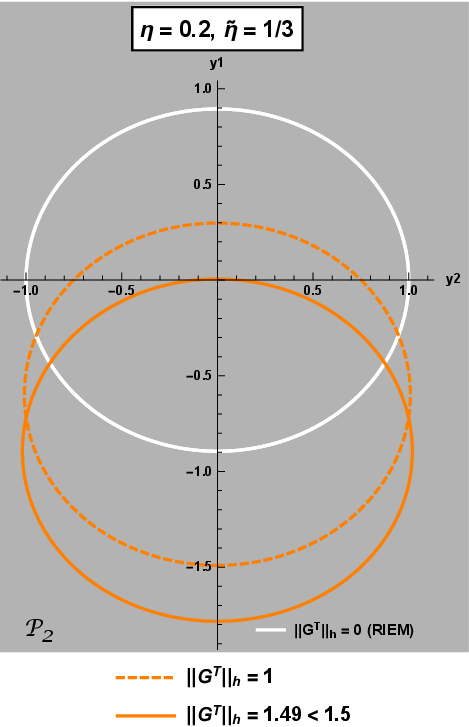} 
\includegraphics[width=0.323\textwidth]{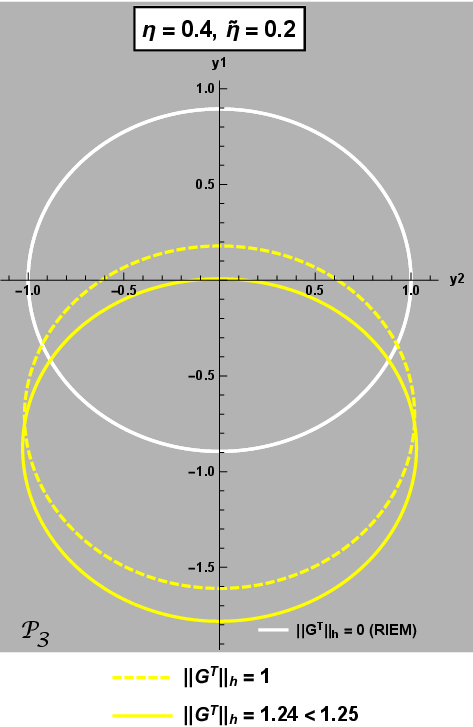}
\\~\\
 ~\includegraphics[width=0.344\textwidth]{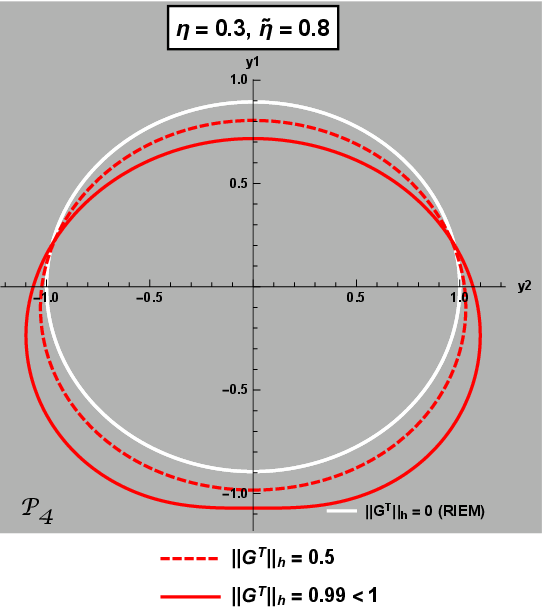} 
\includegraphics[width=0.295\textwidth]{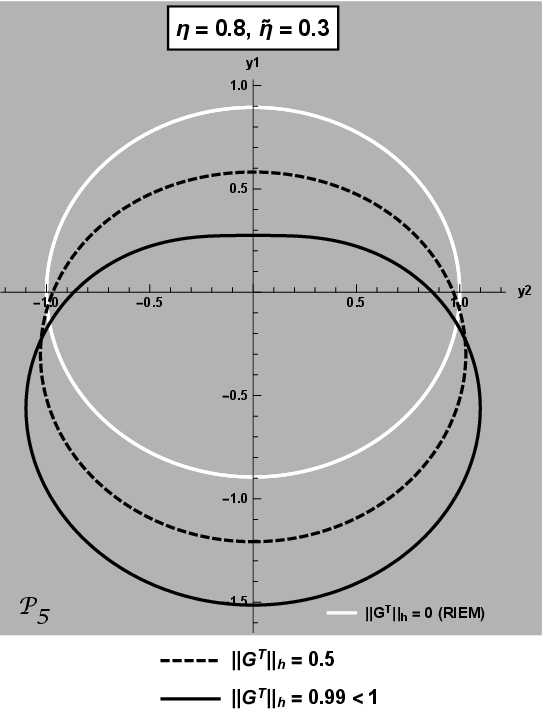} 
\includegraphics[width=0.329\textwidth]{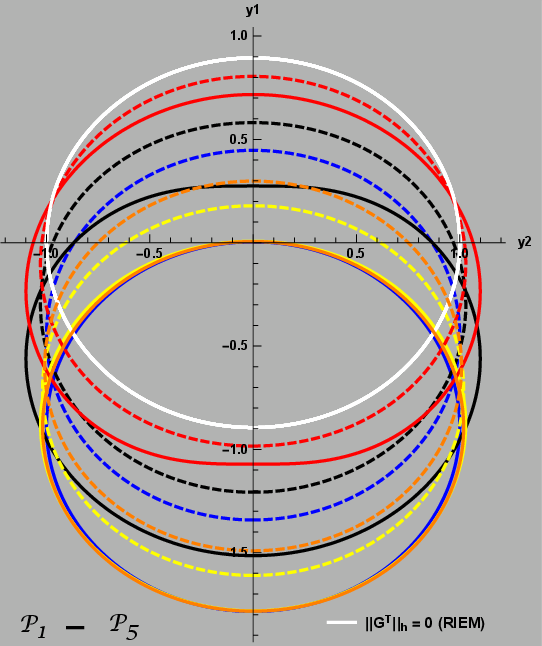}
\caption{The evolution of the Finslerian indicatrices centered at the origin of the coordinate system $y^1Oy^2$ on the planar slippery $(\eta, \tilde{\eta})$-slope given by $z=x/2$, under the action of the gravitational wind of various forces, where the traction coefficients correspond to the problems  $P_1-P_5$ from \cref{partition2}; $t=1$. All the outcomes are compared with each other on the bottom right.  The steepest downhill direction on the slope is indicated by the negative axis $y^1$. The wind force $||\mathbf{G}^{T}||_{h}$ related to the solid colors represent approximately the strongest gravitational wind admitted by the restrictions on strong convexity in each case being considered. On the other end, the elliptical Riemannian indicatrix ($h$-circle) is marked in white in each subfigure.}
\label{fig_plane_indi2}
\end{figure}
 
Finally, we compare all specific types of the Finslerian indicatrices considered in our study in \cref{fig_plane_indirange} (left), i.e. ZNP, MAT, CROSS, RIEM, SLIPPERY, S-CROSS and three new cases coming from the interior of the problem square diagram $\mathcal{\tilde{S}}$, i.e. $(0.25, 0.75)-, (0.2, 0.5)-$ and $(0.5, 0.2)-$slope. The force of the gravitational wind blowing on the planar slope under consideration equals 0.49, which is due to the conditions for strong convexity in the most stringent cases, i.e. MAT and CROSS, where $||\mathbf{G}^{T}||_{h}<0.5$. Therefore, $\bar{g}<\sqrt{5}/2\approx1.118$, since $||\mathbf{G}^{T}||_{h}=\bar{g}/\sqrt{5}$ for the ramp $z=x/2$. Interestingly, the maximum range of an arbitrary $(\eta, \tilde{\eta})$-indicatrix in any direction is created by MAT and CROSS as well as the minimum range by ZNP and RIEM. Namely, all $(\eta, \tilde{\eta})$-indicatrices are located in between those boundaries; for the sake of clarity, see \cref{fig_plane_indirange} (right) in this regard. 

\begin{figure}[h!]
\centering
\includegraphics[width=0.472\textwidth]{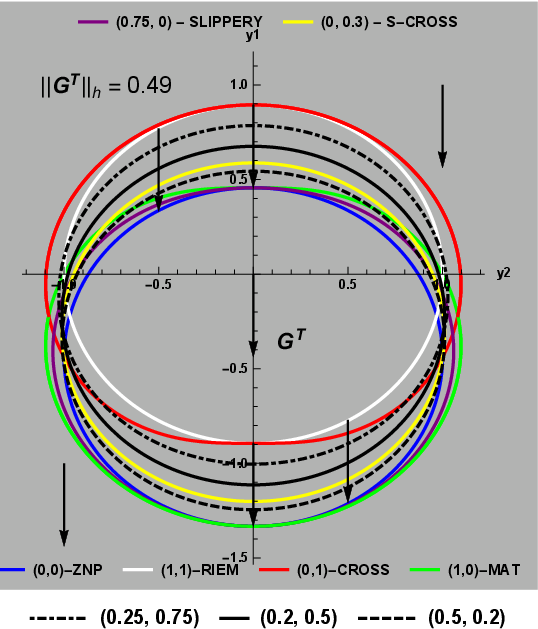}
~\includegraphics[width=0.492\textwidth]{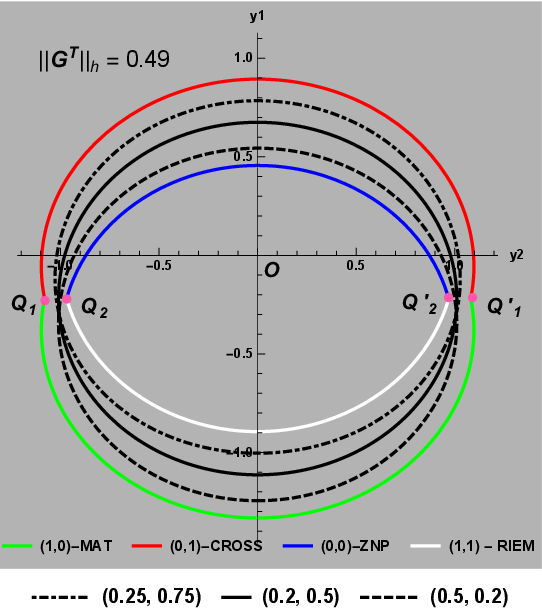}
\caption{Left: A comparison of all specific types of the Finslerian indicatrices (the color-coded lima\c{c}ons) centered at the origin of the coordinate system $y^1Oy^2$ on the planar $(\eta, \tilde{\eta})$-slope given by  $z=x/2$, under the action of the gravitational wind (indicated by black arrows) of constant force $||\mathbf{G}^{T}||_{h}=0.49$; $t=1$. The steepest downhill direction is indicated by the negative axis $y^1$. 
Right (the set-up as on the left): All $(\eta, \tilde{\eta})$-indicatrices (black) are located between the boundaries consisting  of MAT (the lower part, green) and CROSS (the upper part, red), i.e. the maximum (outer) range, as well as ZNP (the upper part, blue) and RIEM (the lower part, white), i.e. the minimum (inner) range. The MAT and CROSS indicatrices intersect each other in the points $Q_1$ and $Q_1'$, which correspond to the directions of the self-velocity $u$: $\theta_{MAT}\in\{77.2^{\circ}, 282.8^{\circ}\}$ and $\theta_{CROSS}\in\{102.8^{\circ}, 257.2^{\circ}\}$ or, equivalently, the directions of the resultant velocity $v_{\eta\tilde{\eta}}$: $\tilde{\theta}\in\{77.2^{\circ}, 282.8^{\circ}\}$, respectively, and $||v_{\eta\tilde{\eta}}||_h\approx1.108$, where $||u||_h=1$. The ZNP and RIEM indicatrices intersect each other in the points $Q_2$ and $Q_2'$, which correspond to the directions of the self-velocity $u$: $\theta_{RIEM}\in\{75.8^{\circ}, 284.2^{\circ}\}$ and $\theta_{ZNP}\in\{104.2^{\circ}, 255.8^{\circ}\}$  or, equivalently, the directions of the resultant velocity $v_{\eta\tilde{\eta}}$: $\tilde{\theta}\in\{75.8^{\circ}, 284.2^{\circ}\}$, respectively, and $||v_{\eta\tilde{\eta}}||_h=1$.}
\label{fig_plane_indirange}
\end{figure}

\subsection{Triple hill}

As the second example we consider a triple Gaussian bell-shaped hill $\mathfrak{G}$ given by
the function%
\begin{equation*}
z=f(x_{1},x_{2})=\frac{1}{4}\sum_{k=1}^{3}(k+1)e^{-\rho _{k}}=\frac{1}{2}%
e^{-\rho _{1}}+\frac{3}{4}e^{-\rho _{2}}+e^{-\rho _{3}},
\end{equation*}%
where for simplicity, we used $x_{1}$ and $x_{2}$ instead of $x^{1}$ and $%
x^{2},$ respectively, and $\rho _{k}=\rho _{k}(x_{1},x_{2})$, $k=1,2,3,$ with 
\begin{equation*}
\rho _{1}=(x_{1}-1)^{2}+(x_{2}+1)^{2},\text{ \ \ }\rho
_{2}=(x_{1}+1)^{2}+(x_{2}+1)^{2},\text{ \ \ }\rho
_{3}=x_{1}^{2}+(x_{2}-1)^{2}.
\end{equation*}
A straightforward computation shows that%
\begin{equation*}
\begin{array}{rll}
f_{x_{1}}&=&-(x_{1}-1)e^{-\rho _{1}}-\frac{3}{2}(x_{1}+1)e^{-\rho
_{2}}-2x_{1}e^{-\rho _{3}}, \\ 
~ \\ 
f_{x_{2}}&=&-(x_{2}+1)e^{-\rho _{1}}-\frac{3}{2}(x_{2}+1)e^{-\rho
_{2}}-2(x_{2}-1)e^{-\rho _{3}}, \\ 
\\ 
f_{x_{1}x_{1}}&=&[2(x_{1}-1)^{2}-1]e^{-\rho _{1}}+\frac{3}{2}%
[2(x_{1}+1)^{2}-1]e^{-\rho _{2}}+2(2x_{1}^{2}-1)e^{-\rho _{3}}, \\ 
\\ 
f_{x_{2}x_{2}}&=&[2(x_{2}+1)^{2}-1]e^{-\rho _{1}}+\frac{3}{2}%
[2(x_{2}+1)^{2}-1]e^{-\rho _{2}}+2[2(x_{2}-1)^{2}-1]e^{-\rho _{3}}, \\ 
\\ 
f_{x_{1}x_{2}}&=&2(x_{1}-1)(x_{2}+1)e^{-\rho _{1}}-3(x_{1}+1)(x_{2}+1)e^{-\rho
_{2}}-4x_{1}(x_{2}-1)e^{-\rho _{3}},%
\end{array}%
\end{equation*}%
the functions $f_{x_{1}x_{1}}$, $f_{x_{2}x_{2}}$ and $f_{x_{1}x_{2}}$
denoting the partial derivatives of $f_{x_{1}}$ and $f_{x_{2}}$ with respect
to $x_{1}$ and $x_{2}$, respectively.

According to \eqref{wind} the gravitational wind acting on $\mathfrak{G}$
is 
\begin{equation}
\mathbf{G}^{T}=-\frac{\bar{g}}{q+1}\left( f_{x_{1}}\frac{\partial }{\partial
x_{1}}+f_{x_{2}}\frac{\partial }{\partial x_{2}}\right) ,  \label{GGGG}
\end{equation}%
with $||\mathbf{G}^{T}||_{h}=\bar{g}\sqrt{\frac{q}{q+1}},$ where%
\begin{eqnarray}
q &=&\frac{1}{4}\sum_{k=1}^{3}(k+1)^{2}\rho _{k}e^{-2\rho _{k}}+\frac{3}{2}%
(\rho _{1}+\rho _{2}-4)e^{-(\rho _{1}+\rho _{2})}  \label{Q} \\
&&+3(\rho _{2}+\rho _{3}-5)e^{-(\rho _{2}+\rho _{3})}+2(\rho _{1}+\rho
_{3}-5)e^{-(\rho _{1}+\rho _{3})}.  \notag
\end{eqnarray}%
Let us denote  the maximum value of the function $\mathcal{A}%
(x_{1},x_{2})=\sqrt{\frac{q}{q+1}}$ by $m$, considering, for example, $x_{1},x_{2}\in
\lbrack -3,3]$ ($m=\underset{x_{1},x_{2}\in \lbrack -3,3]}{\max }\mathcal{A}%
(x_{1},x_{2}))$. Making use of a mathematical software, the approximate value
for $m$ is $0.653$ which is achieved at $(x_{1},x_{2})\approx (0.652,1.272).$
Thus, for $x_{1},x_{2}\in \lbrack -3,3]$, $||\mathbf{G}^{T}||_{h}\leq 
\underset{x_{1},x_{2}\in \lbrack -3,3]}{\max }||\mathbf{G}^{T}||_{h}\approx
0.653\bar{g}$. The rescaled magnitude of the acceleration of gravity $\bar{g}$ needs
handling with greater care to ensure that the geodesics will be indeed
optimal in the sense of time. According to \cref{Theorem1}, the indicatrix
of the $(\eta ,\tilde{\eta})$-slope metric $\tilde{F}_{\eta \tilde{\eta}}$
on the entire triple Gaussian bell-shaped hillside $\mathfrak{G}$, with $%
x_{1},x_{2}\in \lbrack -3,3]$, is strongly convex for any $(\eta ,\tilde{\eta%
})\in \mathcal{S}$ if and only if $\bar{g}<\frac{1}{2\cdot \approx0.653}\approx 0.766$.
Nevertheless, by using the general condition $||\mathbf{G}^{T}||_{h}<\tilde{b%
}_{0},$ where $\tilde{b}_{0}$ is defined in \eqref{Strong_C}, it is immediate
to verify the following results.

\begin{lemma}
\label{lemma_gauss} The indicatrix of the $(\eta ,\tilde{\eta})$-slope
metric $\tilde{F}_{\eta \tilde{\eta}}$ is strongly convex on the entire
surface $\mathfrak{G}$, with $x_{1},x_{2}\in \lbrack -3,3]$, if and only if $%
\ \bar{g}<\ \delta _{2}(\eta ,\tilde{\eta}),$ where%
\begin{equation*}
\ \delta _{2}(\eta ,\tilde{\eta})=\left\{ 
\begin{array}{cc}
\frac{1}{m(1-\tilde{\eta})}, & \text{if}\quad (\eta ,\tilde{\eta})\in 
\mathcal{D}_{1}\cup \mathcal{D}_{2} \\ 
~ &  \\ 
\frac{1}{2m|\eta -\tilde{\eta}|}, & \text{if}\quad (\eta ,\tilde{\eta})\in 
\mathcal{D}_{3}\cup \mathcal{D}_{4}%
\end{array}%
\right.
\end{equation*}%
and $m=$ $\underset{x_{1},x_{2}\in \lbrack -3,3]}{\max }\mathcal{A}%
(x_{1},x_{2}).$
\end{lemma}

In the sequel, we can write the $\tilde{F}_{\eta \tilde{\eta}}$%
-geodesic equations which correspond to $\mathfrak{G}$. According to %
\cref{Theorem2} and \cref{Prop5} the time geodesics $\gamma
(t)=(x_{1}(t),x_{2}(t))$ on the slippery $(\eta ,\tilde{\eta})$-slope of the
surface $\mathfrak{G}$ are provided by the solutions of the ODE system 
\begin{equation}
\ddot{x}_{i}+f_{x_{i}}r_{00}+\frac{2\dot{x}_{i}}{\alpha }\left[\mathit{%
\tilde{\Theta}}(r_{00}+2\alpha ^{2}\tilde{R}r)+\alpha \mathit{\tilde{\Omega}}%
r_{0}\right] +\frac{2f_{x_{i}}}{q+1}\left[ \mathit{\tilde{\Psi}}
(r_{00}+2\alpha ^{2}\tilde{R}r)+\alpha \mathit{\tilde{\Pi}}r_{0}\right]
-2\alpha ^{2}\tilde{R}r^{i}=0,  \label{GGG_bell}
\end{equation}%
$i=1,2,$ where $\mathit{\tilde{\Theta}},$ $\tilde{R},$ $\mathit{\tilde{\Omega}},$ $\mathit{\tilde{\Pi}}$ and $\mathit{\tilde{\Psi}}$ are given by \eqref{geo_tilde}, $q$ by \eqref{Q}, $r^{i}$, $r,$ $r_{0}$ and $r_{00}$ by \eqref{ABRR} with $x_{1}$ and $x_{2}$ instead of $x^{1}$ and $x^{2}$,
respectively, and everywhere in \eqref{GGG_bell}, $x_{1}=x_{1}(t)$, $x_{2}=x_{2}(t)$.

For example, we consider time geodesics and time fronts for $t\in\{1, 2\}$ for the case $\eta=0.7$ and $\tilde{\eta}=0.8$ on the surface $\mathfrak{G}$, where $\bar{g}=0.76$ which corresponds to $||\mathbf{G}^{T}||_{h}<0.5$. The graphical outcome is presented in \cref{3hills}. Remark that the strong convexity condition implies $||\mathbf{G}^{T}||_{h}<5$ for the $(0.7, 0.8)$-slope (see \cref{3hills_winds}). Thus, the maximum value of the rescaled gravitational acceleration can be relaxed, that is, $\bar{g}<\frac{5}{\approx0.653}\approx 7.658$. Moreover, having projected the obtained time geodesics and time fronts for the $(0.7, 0.8)$-slope onto the horizontal plane, the above scenario on the color-coded maps including the gravitational wind force $||\mathbf{G}^{T}||_{h}$ (the steepness of the mountain $\mathfrak{G}$ as well)  are shown in \cref{3hills_0}; $\bar{g}=0.76$. 
It is worth mentioning that the results presented in this subsection can also be applied for the analogous models of a sinkhole or a valley, where the gravity effect is reverse, that is, with the uphill action of the vector field along the surface gradient; see, e.g. the right-hand side image in \cref{3hills} in this regard. 

\begin{figure}[h!]
\centering
\includegraphics[width=0.47\textwidth]{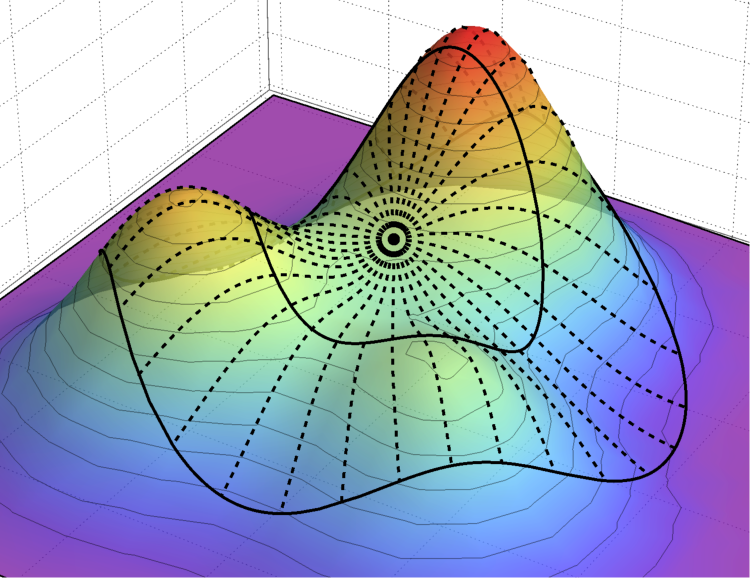}
~\includegraphics[width=0.038\textwidth]{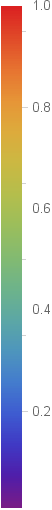}
\includegraphics[width=0.47\textwidth]{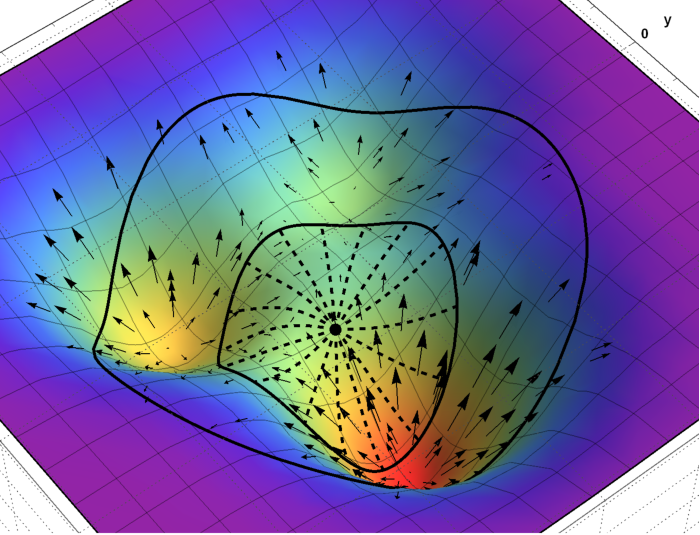}
\caption{Left: The time geodesics (dashed black, with a step of $\Delta\protect\theta=\protect\pi/16$, i.e. 32 optimal paths)  together with the related time fronts for $t=1$ (inner, solid black) and $t=2$ (outer, solid black) centered at $(0, 0)$ on the slippery slope of the mountain $\mathfrak{G}$, where $\eta=0.7$ and $\tilde{\eta}=0.8$, being color-coded with respect to its height measured from the horizontal plane $z=0$. Right: as on the left, in inside and reversed view of the surface $\mathfrak{G}$, where the gravitational wind $\mathbf{G}^{T}$ (black arrows) blows in the steepest ``uphill'' direction; $\bar{g}=0.76$.}
\label{3hills}
\end{figure}

Next, we compare time fronts corresponding to all types of $(\eta, \bar{\eta})$-slopes distinguished in our study, including their projections on the color-coded maps ($z=0$) of both height and gravitational wind force (the steepness of the $\mathfrak{G}$-slopes) on the triple Gaussian bell-shaped hill. This is done in \cref{3hills_2}. The mutual relations between the $(\eta, \bar{\eta})$-isochrones are analogous to the inclined plane and in line with the findings obtained in the previous studies on the generalization of the Matsumoto's slope-of-a-mountain problem (\cite{slippery, cross, slipperyx}). 
 \begin{figure}[h!]
\centering
\includegraphics[width=0.55\textwidth]{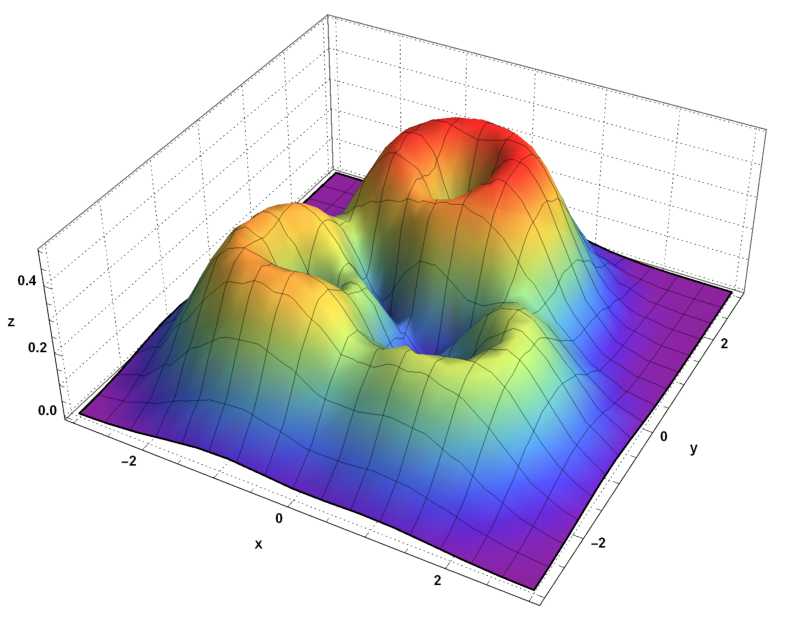}
\includegraphics[width=0.44\textwidth]{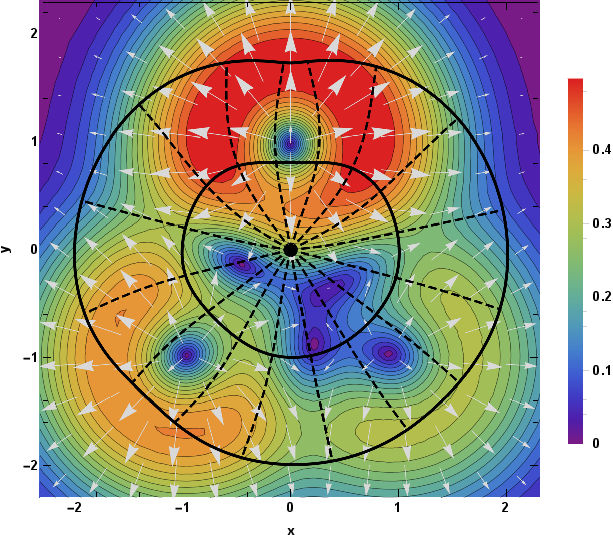}
\caption{Left: The color-coded graph of the gravitational wind force $||\mathbf{G}^{T}||_{h}$, blowing in the steepest downhill directions, i.e. along the negative surface gradient, where $\bar{g}=0.76$. This also describes the steepness of the mountain slopes expressed by the norm of the position-dependent $\mathfrak{G}$-surface gradient.   
Right: The time fronts for $t=1$ (inner, solid black) and $t=2$ (outer, solid black) centered at $(0, 0)$, together with the related time geodesics (dashed black, drawn with a step of $\Delta\protect\theta=\protect\pi/8$, i.e. 16 time-minimizing paths) on the $(0.7, 0.8)$-slope of the mountain. In the background, the flat map which is color-coded with respect to the gravitational wind force (the steepness) on $\mathfrak{G}$ as on the left, including the vector field showing the action of the gravitational wind (indicated by the grey arrows).}
\label{3hills_0}
\end{figure}

Note that in order to compare all types of the slippery slopes on $\mathfrak{G}$ the strong convexity conditions for the most restrictive cases, i.e. CROSS and MAT require that $\bar{g}<~0.766$. Therefore, we applied $\bar{g}=0.76$ in all comparative expositions presented. We remark that it is more stringent for the surface $\mathfrak{G}$ than the inclined plane investigated in the preceding subsection, where $\bar{g}<\sqrt{5}/2\approx1.118$.

\begin{figure}[h!]
\centering
\includegraphics[width=0.50\textwidth]{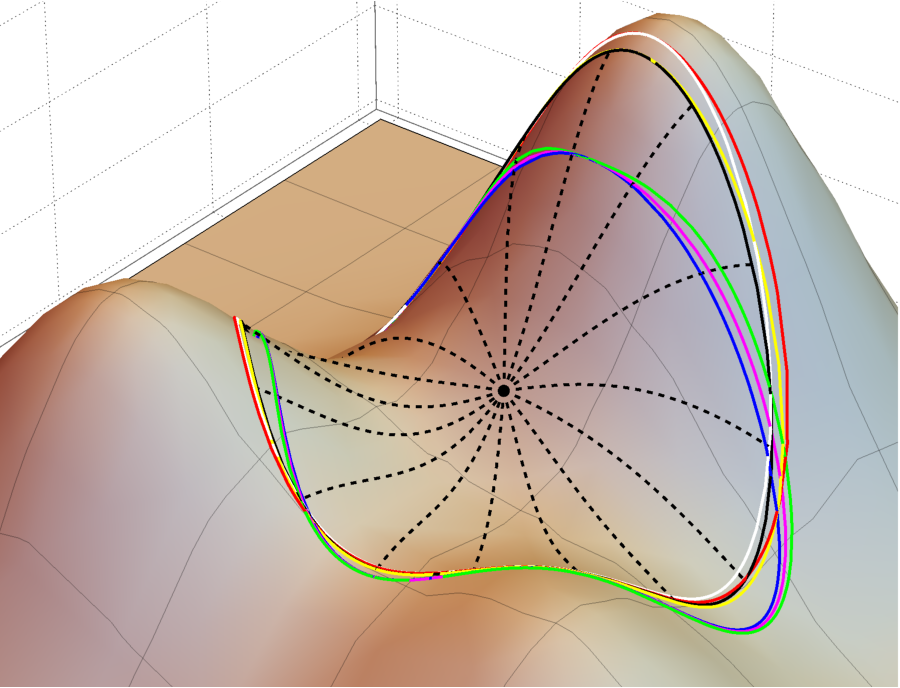}
\includegraphics[width=0.49\textwidth]{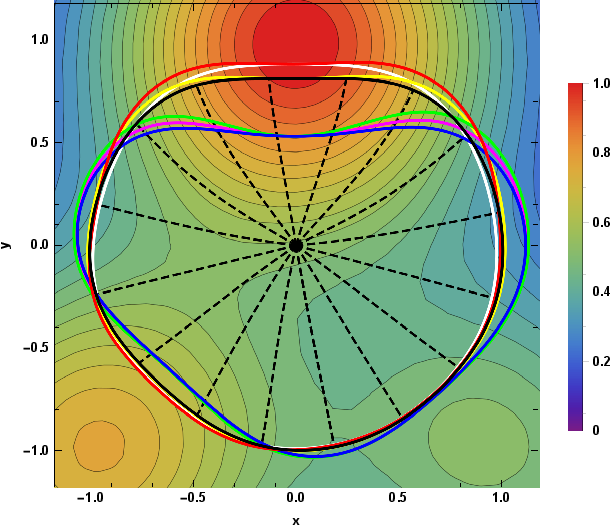}  \\~\\
\includegraphics[width=0.51\textwidth]{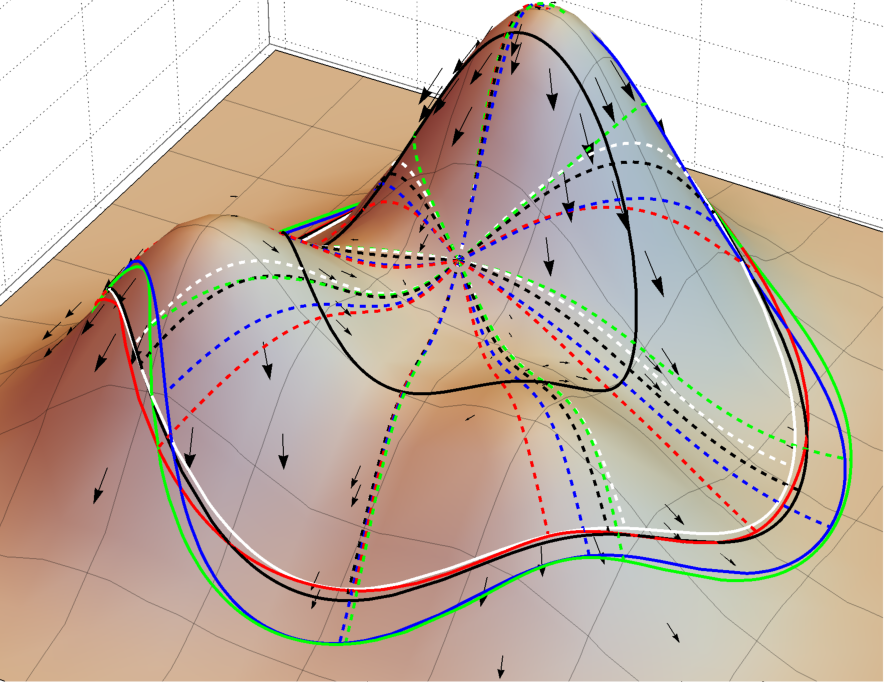}
\includegraphics[width=0.48\textwidth]{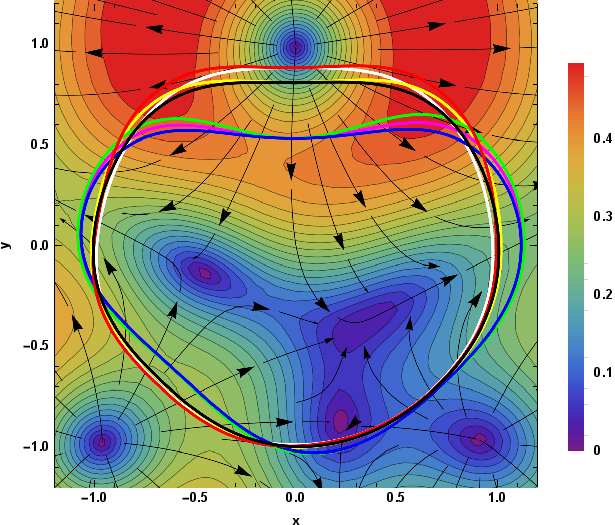}
\caption{The triple Gaussian bell-shaped hill $\mathfrak{G}$ (left) and its color-coded maps of height (top right) and the gravitational wind force (also presenting the steepness of the slope; bottom right), including time fronts centered at $(0, 0)$ for the cases: MAT (green), ZNP (blue), RIEM (white), SLIPPERY (with $\eta=0.7$, magenta), S-CROSS (with $\tilde{\eta}=0.8$, yellow), CROSS (red) and $(0.7, 0.8)$-slope (black), as well as the related time geodesics (dashed colours, respectively), where $t=1$ (top left) and $t=2$ (bottom left; in addition, $(0.7, 0.8)$-case for $t=1$); $\bar{g}=0.76$. The time geodesics are drawn with a step of $\Delta\protect\theta=\protect\pi/8$ (16 paths) for $t=1$ and $\Delta\protect\theta=\protect\pi/4$ (8 paths) for $t=2$. The action of the gravitational wind is indicated by the black arrows. The projections of the time fronts presented on the maps are for $t=1$.}
\label{3hills_2}
\end{figure}

Furthermore, the effect of a variable gravitational wind force by changing the rescaled acceleration of gravity $\bar{g}$ on behavior of the unit time front is pointed out on the slippery slope of $\mathfrak{G}$  in \cref{3hills_winds}, where the initial point is located now on the hillside, i.e. $(1, 0)$ and both traction coefficients are fixed, i.e. $\eta=0.7$ and $\tilde{\eta}=0.8$. The related unit time fronts are presented for $\bar{g}\in\{0.76\  (\textnormal{black}), 3\  (\textnormal{magenta}), 5 \ (\textnormal{orange}), 7.65\  (\textnormal{yellow})\}$. Recall that the condition for strong convexity for the $(0.7, 0.8)$-slope yields $\bar{g}<\approx 7.658$. The deformations are also compared on the related map of height in \cref{3hills_winds} (right), where the action of gravitational wind is marked by grey arrows. As expected, if $\mathbf{G}^{T}$-force increases, then the isochrone on the hill slope is in particular ``shifted'' downhill, roughly speaking.  

\begin{figure}[h!]
\centering
\includegraphics[width=0.55\textwidth]{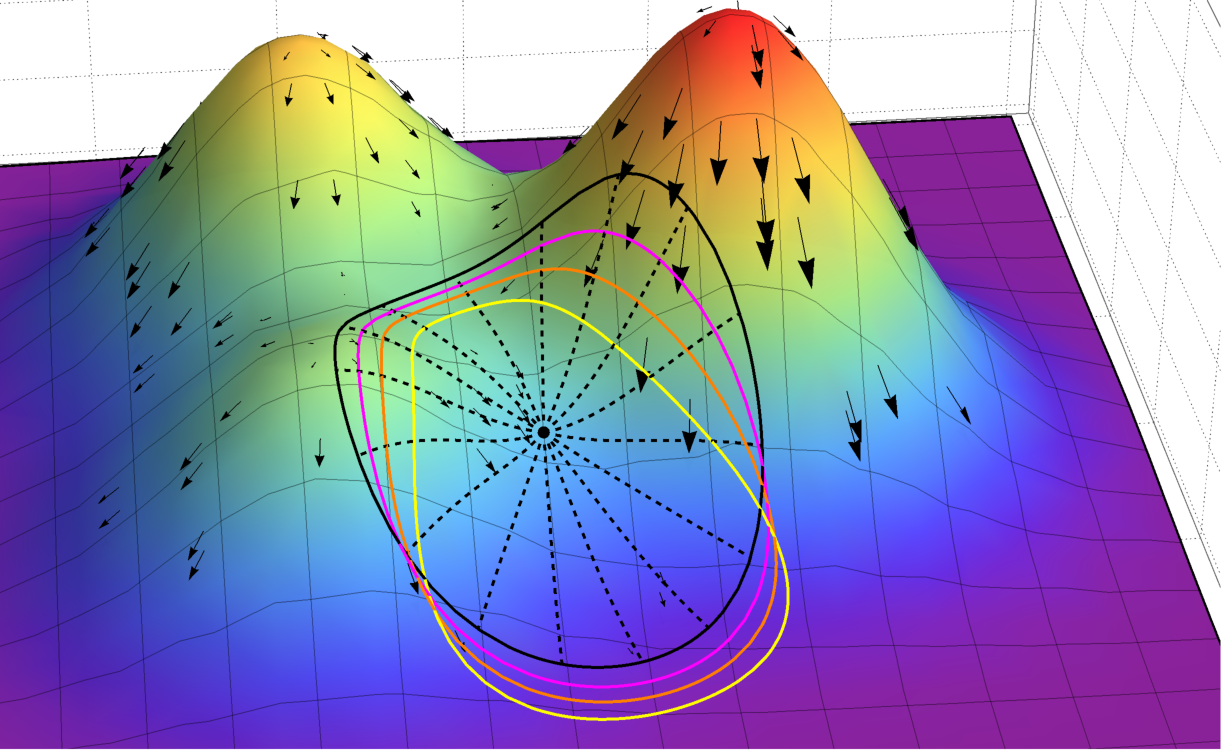}
\includegraphics[width=0.036\textwidth]{map_1_bar.eps}
~\includegraphics[width=0.39\textwidth]{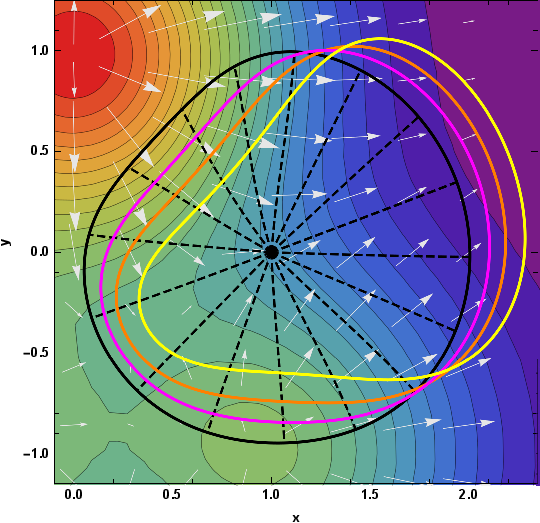}\\~\\
\caption{Left: The evolution of the unit time front on the slippery slope of the surface $\mathfrak{G}$ with respect to the variable force of the gravitational wind (due to changing the rescaled acceleration of gravity $\bar{g}$), where $\bar{g}\in\{0.76\  (\textnormal{black}), 3\  (\textnormal{magenta}), 5 \ (\textnormal{orange}), 7.65\  (\textnormal{yellow})\}$. The initial point is located at $(1, 0)$ and the traction coefficients are fixed, i.e. $\eta=0.7$ and $\tilde{\eta}=0.8$. The corresponding time geodesics in the initial setting ($\bar{g}=0.76$) are presented in dashed black and drawn with a step of $\Delta\protect\theta=\protect\pi/8$ (16 paths). 
Right: The scenario as on the left, presented on the color-coded map of height. The action of gravitational wind is marked by arrows in both subfigures.}
\label{3hills_winds}
\end{figure}

We end this section by comparing the minimum (inner) and maximum (outer) ranges of unit time fronts for any direction of motion on the slope under the gravitational wind ($\bar{g}=0.76$), which are presented on a flat map of the 3-hill in \cref{3hills_range} (left). The former consists of the ZNP (blue) and RIEM (white) parts of the related isochrones as well as the latter of the CROSS (red) and MAT (green) ones; cf. \cref{fig_plane_indirange} (right) for the analogous situation on the inclined plane. This kind of outcomes can be applied in particular to model a range of propagation of some natural phenomena, physical and chemical processes on the hill slopes, which depend on gravity. On the other hand, such analysis also gives information on what type of $(\eta, \tilde{\eta})$-navigation (motion) should be applied in order to minimize or maximize a requested range in arbitrary direction, distinguishing, for instance, the uphill and downhill cases. All $(\eta, \tilde{\eta})$-time fronts on the slippery slope modeled by the triple Gaussian bell-shaped hill $\mathfrak{G}$ are located between the above mentioned boundaries. For the sake of clarity, this is shown by two individuals, where the traction coefficients are equal to $(0.6, 0.4)$ - dashed black and $(0.7, 0.8)$ - solid black in \cref{3hills_range} (right). 

\begin{figure}[H]
\centering
\includegraphics[width=0.49\textwidth]{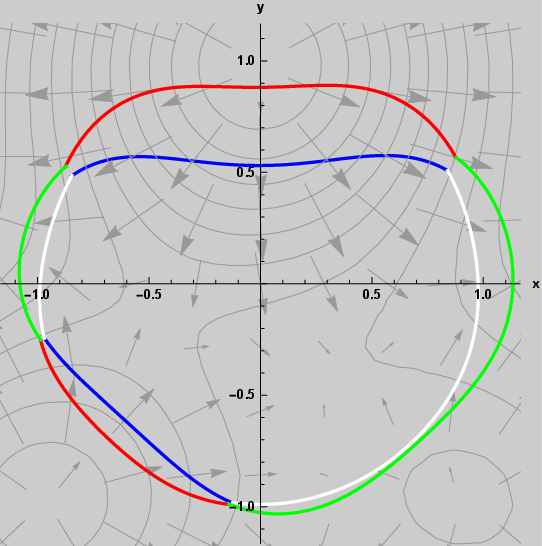}
\includegraphics[width=0.49\textwidth]{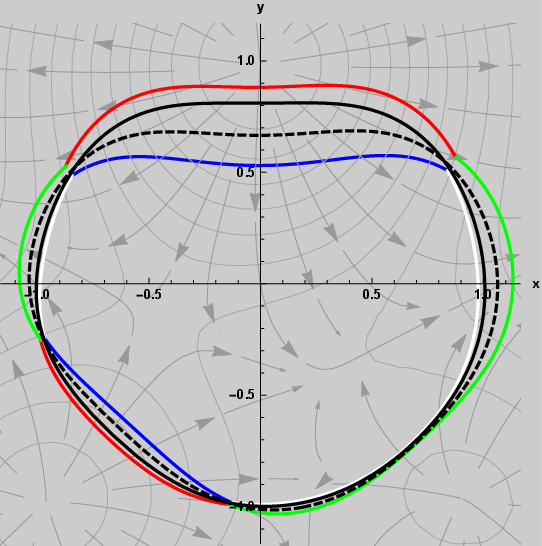}
\caption{Left: A comparison of the minimum (inner) and maximum (outer) ranges (boundaries) in any direction of unit time fronts centered at $(0, 0)$ and related to time geodesics under the gravitational wind, presented on a flat map of the triple Gaussian bell-shaped hill $\mathfrak{G}$, consisting of the ZNP (blue) and RIEM (white) parts as well as the MAT (green) and CROSS (red) parts, respectively; $t=1$. In the background, the vector field (grey arrows) representing the gravitational wind $\mathbf{G}^{T}$ `'blowing'' in the steepest downhill directions on $\mathfrak{G}$ as well as the contours of equal height. 
Right: As on the left, with two individual time fronts in addition, i.e. $(0.6, 0.4)$ - dashed black and $(0.7, 0.8)$ - solid black; $\bar{g}=0.76$. In general, all time fronts are located between the minimum (inner) boundary consisting of the ZNP (blue) and RIEM (white) cases as well as the maximum (outer) boundary consisting of the MAT (green) and CROSS (red) cases. In the background, the stream plot of the gravitational wind, indicating its action in the steepest downhill directions on $\mathfrak{G}$ by grey arrows and the contours of equal height. 
}
\label{3hills_range}
\end{figure}

\section{Conclusions}

The research presented a general model for minimum-time navigation on a slippery mountain slope under the action of gravity, which effectively unifies and extends the previous studies on the navigation problems in Finsler geometry by the same authors. The introduced two-parameter model admits both components of the gravitational wind acting in the steepest downhill direction to vary in the full ranges so that the transverse and longitudinal gravity-additives relative to the direction of motion can be counterbalanced simultaneously by the incorporated cross- and along-traction coefficients. 
The  comprehensive study established the strong convexity conditions 
 under which the resultant velocity defines a new Finsler metric, i.e. $(\eta, \tilde{\eta})$-slope metric, which turned out to be of the general $(\alpha, \beta)$-type, and determined its geodesics, which correspond to the time-minimizing trajectories between two arbitrary points on a slippery mountainside. The main result encompasses as particular cases the solutions to Matsumoto's slope-of-a-mountain problem, inducing a Matsumoto metric, as well as Zermelo's navigation problem on a Riemannian manifold under a gravitational wind, inducing a Randers metric.  
This generalization has substantial physical applicability, since it more precisely models real-world scenarios entailing movement on a slippery slope under gravity.  
Furthermore, the study contains a variety of figures and two-dimensional examples that help to explain the novel model and effectively illustrate the theory's applications to specific scenarios, such as the inclined plane and triple hill cases. This includes the evolution of time fronts and the behavior of  time geodesics in relation to various gravity effects, gravitational wind force and direction of motion, which have been thoroughly discussed. 
 
The results have been encouraging enough to merit further investigation. In particular, the model with some modifications can involve other types of the natural forces (interactions) than gravity in the context of time-optimal navigation. It would also be interesting to see whether new simpler explicit  expressions for the $(\eta, \tilde{\eta})$-slope metric can be obtained for some values of the traction coefficients (i.e. the Finsler metrics different from the classical Randers and Matsumoto types) or they can be fitted to model other natural phenomena, process or related motion effectively, similarly to the Matsumoto and Randers metrics, both of which have been extensively studied in the literature. Taking a step further, the research can be dedicated to geometric properties of the slippery slope metric, although it has been defined implicitly, as well as the extensions of the presented model in which the traction coefficients or/and the self-velocity are position-dependent in addition. Moreover, they as well as the gravitational wind can also vary over time. This is a tantalising area for future research.


\bibliographystyle{plain}

\end{document}